\documentclass{article}
\usepackage[a4paper,top=3cm,bottom=3.5cm,left=3.5cm,right=3.5cm,marginparwidth=1.75cm]{geometry}
\usepackage[hidelinks]{hyperref}
\usepackage{amsmath, amssymb, amsthm, bbm, comment, graphicx, cleveref, caption, siunitx, color, tikz, bigints, mathtools, cleveref}
\usepackage{xcolor}
\usepackage{tikz}
\usepackage{wrapfig}
\usetikzlibrary{intersections,decorations.pathreplacing,decorations.markings,calc,positioning}

\definecolor{mycolor1}{rgb}{0.20000,0.30000,0.70000}
\usepackage{pgfplots} 
\usetikzlibrary{intersections,calc}

\newcommand{\Ro}{\mathcal{R}}
\newtheorem{theorem}{Theorem}
\newtheorem{lemma}[theorem]{Lemma}
\newtheorem{proposition}[theorem]{Proposition}
\newtheorem{defn}[theorem]{Definition}
\newtheorem{remark}[theorem]{Remark}
\def\cH{{\mathcal{H}}}

\newcommand{\ddp}[2]{\frac{\partial#1}{\partial#2}}
\newcommand{\outside}{\mathbb{R}^3\setminus \overline{D}}
\newcommand{\D}{\partial D}
\newcommand{\K}{\mathcal{K}}
\renewcommand{\S}{\mathcal{S}}
\newcommand{\C}{C^\mathrm{vol}}
\newcommand{\Crystal}{\mathcal{C}}
\newcommand{\de}{\: \mathrm{d}}
\renewcommand*{\Re}{\operatorname{Re}}
\renewcommand*{\Im}{\operatorname{Im}}
\renewcommand{\i}{\mathrm{i}}

\renewcommand{\S}{\mathcal{S}}

\renewcommand*{\Re}{\operatorname{Re}}
\renewcommand*{\Im}{\operatorname{Im}}

\newcommand{\R}{\mathbb{R}}
\newcommand{\iu}{\mathrm{i}\mkern1mu}
\graphicspath{{images/}}
\renewcommand{\L}{\mathcal{L}}

\newcommand{\N}{\mathbb{N}}
\renewcommand{\O}{\mathcal{O}}

\renewcommand{\S}{\mathcal{S}}

\newcommand{\Z}{\mathbb{Z}}
\newcommand{\p}{\partial}
\renewcommand{\epsilon}{\varepsilon}
\newcommand{\dx}{\: \mathrm{d}}

\renewcommand{\b}[1]{\textbf{#1}}

\newcommand{\ie}{\textit{i.e.}}
\newcommand{\nm}{\noalign{\smallskip}}
\newcommand{\ds}{\displaystyle}

\numberwithin{equation}{section}
\numberwithin{theorem}{section}
\numberwithin{figure}{section}

\graphicspath{ {figures/} }

\title{Wave interaction with subwavelength resonators}

\author{Habib Ammari\thanks{\footnotesize Department of Mathematics, ETH Z\"urich, R\"amistrasse 101, CH-8092 Z\"urich, Switzerland (habib.ammari@math.ethz.ch, bryn.davies@sam.math.ethz.ch, erik.orvehed.hiltunen@sam.math.ethz.ch).}\and Bryn Davies\footnotemark[1]  \and Erik Orvehed Hiltunen\footnotemark[1] \and  Hyundae Lee\thanks{\footnotesize  Department of Mathematics, Inha University,  253 Yonghyun-dong Nam-gu,  Incheon 402-751,  Korea (hdlee@inha.ac.kr).} \and Sanghyeon Yu\thanks{\footnotesize Department of Mathematics, Korea University, Seoul 02841, S. Korea (sanghyeon\_yu@korea.ac.kr).}}

\date{}

\begin{document}
	\maketitle
	
	\begin{abstract} 
	
	The aim of this review is to cover recent developments in the mathematical analysis of subwavelength resonators. The use of sophisticated mathematics in the field of metamaterials is reported, which provides a mathematical framework for focusing, trapping, and guiding of waves at subwavelength scales. Throughout this review, the power of layer potential techniques combined with asymptotic analysis for solving challenging wave propagation problems at subwavelength scales is demonstrated. 	
		
\end{abstract}

\maketitle

\section{Introduction}

The ability to focus, trap, and guide the propagation of waves on subwavelength scales is of fundamental importance in physics. Systems of subwavelength resonators have, in particular, been shown to have desirable and sometimes remarkable properties thanks to their tendency to interact very strongly with waves on small length scales \cite{kaina2015negative,phononic1, pendry}. A subwavelength resonator is a cavity with material parameters that are greatly different from the background medium and which experiences resonance in response to wavelengths much greater than its size.  The large material contrast  is an essential prerequisite for the subwavelength resonant response.

In this review, we consider  wave interaction with systems of subwavelength resonators. At particular low frequencies, known as subwavelength resonances, subwavelength resonators behave as strong wave scatterers. Using layer potential techniques and Gohberg-Sigal theory, we first derive  a formula for the  resonances of a system of  resonators of arbitrary shapes. Then, we  derive an effective medium theory for  wave propagation in  systems of resonators. We start with a multiple scattering formulation of the scattering problem in which an incident wave impinges on a large number of small, identical resonators in a homogeneous medium. Under certain conditions on the configuration of the resonator distribution, the point interaction approximation holds and yields an effective medium theory for the system of resonators as the number of resonators tends to infinity. As a consequence, below the resonant frequency of a single resonator, the obtained effective media may be highly refractive, making the focusing of waves at subwavelength scales achievable.

Then,  we provide a mathematical theory for understanding the mechanism behind the double-negative refractive index phenomenon in systems of subwavelength resonators. The design of double-negative metamaterials generally requires the use of two different kinds of subwavelength resonators, which may limit the applicability of double-negative metamaterials. Herein we rely on media that consists of only a single type of resonant element, and show how to turn the  metamaterial with a single negative effective property  into a negative refractive index metamaterial, which acts as a superlens. Using  dimers made of two identical resonators, we show that both  effective material parameters can be negative near the anti-resonance of the two hybridized  resonances for a single constituent dimer of subwavelength resonators. 

Furthermore,  we consider periodic structures of subwavelength resonators where subwavelength band gap opening typically  occurs. This can induce rich physics on the subwavelength scale which cannot be understood by the standard homogenization theory. To demonstrate the opening of a subwavelength band gap, we exploit the strong interactions produced by subwavelength resonators among the cells in a periodic structure. We derive an approximate formula in terms of the contrast for the quasi-periodic subwavelength resonant frequencies of an arbitrarily shaped subwavelength resonator. Then, we consider the behavior of the first Bloch eigenfunction near the critical frequency where a subwavelength band
gap of the periodic structure opens. For a square lattice, we show that the critical frequency occurs at the corner of the Brillouin zone where the Bloch eigenfunctions are antiperiodic. We develop a high-frequency homogenization technique to describe the rapid variations of the Bloch eigenfunctions on the microscale (the scale of the elementary crystal cell). Compared to the effective medium theory, an effective equation can be derived only for the envelope of this first Bloch eigenfunction. 

Defect modes and guided modes can be shown to exist in perturbed subwavelength resonator crystals. We use the  subwavelength band gap  to demonstrate cavities and waveguides of subwavelength dimensions. First, by perturbing the size of a single resonator inside the crystal, we show that this crystal has a localized eigenmode close to the defect resonator. Further, by modifying the sizes of the subwavelength resonators along a line in a crystal, we show  that the line defect acts as a waveguide; waves of certain frequencies will be localized to, and guided along, the line defect. 

Topological properties of periodic lattices of subwavelength resonators are also considered, and we investigate the existence of Dirac cones in honeycomb lattices and topologically protected edge modes in chains of subwavelength resonators.  
We first show the existence of a Dirac dispersion cone in a honeycomb  crystal comprised
of subwavelength  resonators of arbitrary shape. The high-frequency homogenization technique shows that, near the Dirac points, the Bloch eigenfunction is the sum of two eigenmodes. Each eigenmode can be decomposed into two components: one which is slowly varying and satisfies a homogenized equation, while the other is periodic and highly oscillating. The slowly oscillating components of the eigenmodes satisfy a system of Dirac equations. This yields a near-zero effective refractive index near the Dirac points for the plane-wave envelopes of the Bloch eigenfunctions in a subwavelength metamaterial. The opening of a Dirac cone can create topologically protected edge modes, which are stable against geometric errors of the structure.  We study a crystal which consists of a chain of subwavelength resonators arranged as dimers (often known as an SSH chain) and show that it exhibits a topologically non-trivial band gap, leading to robust localization properties at subwavelength scales.  

Finally, we present a bio-inspired system of subwavelength resonators designed to mimic the cochlea. The system is inspired by the graded properties of the cochlear membranes, which are able to perform spatial frequency separation. Using layer potential techniques, the resonant modes of the system are computed and the model's ability to decompose incoming signals is explored. 

This review is organized as follows. In Section~\ref{sec2}, after stating the subwavelength resonance problem and introducing some preliminaries on the layer potential techniques and Gohberg-Sigal theory, we prove the existence of subwavelength resonances for systems of resonators and show a modal decomposition for the associated eigenmodes.  Then, we study in Section~\ref{sec3} the behavior of the coupled subwavelength resonant modes when the subwavelength resonators are brought close together. In Section~\ref{dilutesect} we derive an effective medium theory for dilute systems of subwavelength resonators. Section \ref{sec5} is devoted to the spectral analysis of periodic structures of subwavelength resonators. After recalling some preliminaries on the Floquet theory and quasi-periodic layer potentials, we prove the occurrence of subwavelength band gap opening in square lattices of subwavelength resonators and characterize the behavior of the first Bloch eigenfunction near the critical frequency where a subwavelength band gap of the periodic structure opens. In Section~\ref{sec6}, we consider honeycomb lattices of subwavelength resonators and prove the existence of Dirac cones. We also study a chain of subwavelength resonators which exhibit a topologically non-trivial band gap. Finally, in Section~\ref{sec7}, we present a graded array of subwavelength resonators which is designed to mimic the frequency separation proprieties of the cochlea.  The review ends with some concluding remarks.

\section{Subwavelength resonances} \label{sec2}
We begin by describing the resonance problem and the main mathematical tools we will use to study a finite collection of subwavelength resonators.
\subsection{Problem setting}
We are interested in studying wave propagation in a homogeneous background medium with $N\in\mathbb{N}$ disjoint bounded inclusions, which we label as $D_1,D_2,\dots,D_N\subset\mathbb{R}^3$. We assume that the boundaries are all of class $\mathcal{C}^{1, \eta}$ with $0<\eta <1$ and write $D=D_1\cup\dots\cup D_N$.

We denote the material parameters within the bounded regions $D$ by $\rho_b$ and $\kappa_b$, respectively. The corresponding parameters for the background medium are $\rho$ and $\kappa$ and the wave speeds in $D$ and $\outside$ are given by
$v_b=\sqrt{{\kappa_b}/{\rho_b}}$ and $v=\sqrt{{\kappa}/{\rho}}$. We define the wave numbers as
	\begin{equation} \label{defkv} k = \frac{\omega}{v}, \quad k_b = \frac{\omega}{v_b}.\end{equation} 
We also define the dimensionless contrast parameter
\begin{equation} \label{defdelta}
\delta=\frac{\rho_b}{\rho}.
\end{equation}
We assume that 
\begin{equation} \label{defp}
\delta\ll1 \mbox{ while } v_b =\O(1) \mbox{ and } v =\O(1).\end{equation} 
This high-contrast assumption will give the desired subwavelength behaviour, which we will study through an asymptotic analysis in terms of $\delta$.

For $\omega \in \mathbb{C}$, we study the scattering problem
	\begin{equation} \label{eq:scattering}
	\left\{
	\begin{array} {ll}
	\ds \Delta {u}+ k^2 {u}  = 0 & \text{in } \R^3 \setminus \overline{D}, \\[0.3em]
	\ds \Delta {u}+ k_b^2 {u}  = 0 & \text{in } D, \\
	\nm
	\ds  {u}|_{+} -{u}|_{-}  = 0  & \text{on } \partial D, \\
	\nm
	\ds  \delta \frac{\partial {u}}{\partial \nu} \bigg|_{+} - \frac{\partial {u}}{\partial \nu} \bigg|_{-} = 0 & \text{on } \partial D, \\
	\nm
	\ds u(x) - u^{in}(x) & \text{satisfies the Sommerfeld radiation} \\ & \text{condition as }  |x| \rightarrow \infty,
	\end{array}
	\right.
	\end{equation}
	where $|_+$ and $|_-$ denote the limits from the outside and inside of $D$. Here, $u^{in}$ is the incident field which we assume satisfies $\Delta u^{in} + k^2u^{in} = 0$ in $\R^3$ and $\nabla u^{in}\big|_D = \O(\omega)$. We restrict to frequencies such that $\Re(k) > 0$, whereby the Sommerfeld radiation condition is given by
	\begin{equation}\label{eq:sommerfeld}
	\lim_{|x|\rightarrow \infty}|x| \left(\frac{\p}{\p |x|} -\iu k\right)u = 0,
	\end{equation}
	which corresponds to the case where $u$ radiates energy  outwards (and not inwards).

\begin{defn}[Subwavelength resonant frequency] \label{def:res^}
	We define a {subwavelength resonant frequency} to be $\omega=\omega(\delta)$ such that $\Re(\omega) > 0$ and:
	\begin{itemize}
		\item[(i)] there exists a non-trivial solution to  \eqref{eq:scattering} when $u^{in}=0$,
		\item[(ii)] $\omega$ depends continuously on $\delta$ and is such that $\omega(\delta)\to0$ as $\delta\to0$.
	\end{itemize}
\end{defn}	
	The scattering problem \eqref{eq:scattering} is a model problem for subwavelength resonators with high-contrast materials. It  can be effectively studied using representations in terms of integral operators.

\subsection{Layer potential theory on bounded domains and Gohberg-Sigal theory}
 \label{sec:layerpot}
 The layer potential operators are the main mathematical tools used in the study of the resonance problem described above. These are operator-valued holomorphic functions, and can be studied using Gohberg-Sigal theory.
 
 \subsubsection{Layer potential operators}
	Let $\S_D^k$ be the single layer potential, defined by
	\begin{equation} \label{eq:Sdef}
	\S_D^k[\phi](x) := \int_{\partial D} G^k(x-y)\phi(y) \; \dx \sigma(y), \quad x \in \R^3,
	\end{equation}
	where $G^k(x)$ is the outgoing Helmholtz Green's function, given by
	$$
	G^k(x) := -\frac{e^{\iu k|x|}}{4\pi|x|}, \quad x \in \R^3, \ \Re(k)\geq 0.
	$$
	Here, ``outgoing'' refers to the fact that $G^k$ satisfies the Sommerfeld radiation condition \eqref{eq:sommerfeld}. For $k=0$ we omit the superscript and write the fundamental solution to the Laplacian as $G$. 

	For the single layer potential corresponding to the Laplace equation, $\S_D^0$, we also omit the superscript and write $\S_D$. It is well known that the trace operator $\S_D: L^2(\p D) \rightarrow H^1(\p D)$ is invertible, where $H^1(\p D)$ is the space of functions that are square integrable on $\p D$ and have a weak first derivative that is also square integrable. We denote by $\L(L^2(\partial D),H^1(\partial D))$ the set of bounded linear operators from $L^2(\partial D)$ into $H^1(\partial D)$.
	
	The Neumann-Poincar\'e operator $\K_D^{k,*}: L^2(\partial D) \rightarrow L^2(\partial D)$ is defined by
	\begin{equation*}\label{eq:K_def}
	\K_D^{k,*}[\phi](x) := \int_{\partial D} \frac{\partial }{\partial \nu_x}G^k(x-y) \phi(y) \; \dx \sigma(y), \quad x \in \partial D,
	\end{equation*}
	where $\partial/\partial \nu_x$ denotes the outward normal derivative at $x\in\p D$.  For $k=0$ we omit the superscript and write $\K_D^{*}$. 
	
	The behaviour of $\S_D^k$ on the boundary $\partial D$ is described by the following relations, often known as \emph{jump relations},
	\begin{equation}\label{eq:jump1}
	\S_D^k[\phi]\big|_+ = \S_D^k[\phi]\big|_-,
	\end{equation}
	and
	\begin{equation}\label{eq:jump2}
	\frac{\partial }{\partial \nu}\S_D^k[\phi]\Big|_{\pm}  =  \left(\pm\frac{1}{2} I + \K_D^{k,*}\right) [\phi],
	\end{equation}
	where $|_\pm$ denote the limits from outside and inside $D$. When $k$ is small, the single layer potential satisfies
	\begin{equation} \label{eq:exp_S}
	\S_D^k =\S_D + k\S_{D,1} + k^2\S_{D,2} + k^3 \S_{D,3} +  \O(k^4),
	\end{equation}
	where the error term is with respect to the operator norm $\|.\|_{\L(L^2(\partial D),H^1(\partial D))}$, and the operators $\S_{D,n}: L^2(\partial D) \rightarrow H^1(\partial D)$ for $n=1,2,3$ are given by
	$$
	\S_{D,n}[\phi](x) =  - \frac{\iu^{n}}{4 \pi n!} \int_{\p D}| x - y |^{n-1} \phi(y) \; \de \sigma(y) \qquad x\in \D.
	$$
	Moreover, we have
	\begin{equation} \label{eq:exp_K}
	\K_D^{k,*} = \K_D^{0,*} + k^2\K_{D,2} + k^3\K_{D,3} + \O(k^4),
	\end{equation}
	where the error term is with respect to the operator norm $\|.\|_{\L(L^2(\partial D),L^2(\partial D))}$ and where 
	$$\K_{D,2}[\phi](x) = \frac{1}{8\pi}\int_{\p D}\frac{(x-y)\cdot \nu_x}{|x-y|}\phi(y)\de \sigma (y), \quad \K_{D,3}[\phi](x) = \frac{\iu}{12\pi}\int_{\p D}(x-y)\cdot \nu_x\phi(y)\de \sigma (y),$$ for $x\in \partial D$. 	
	We have the following lemma from \cite{ammari2017double}.
	\begin{lemma} \label{lem:ints}
		Let $N=2$. For any $\varphi\in L^2(\D)$ we have, for $i=1,2$,
		\begin{equation} \label{eq:properties}
		\begin{split}
		\int_{\D_i}\left(-\frac{1}{2}I+\K_D^{*}\right)[\varphi]\de\sigma=0,
		\qquad&\int_{\D_i}\left(\frac{1}{2}I+\K_D^{*}\right)[\varphi]\de\sigma=\int_{\D_i}\varphi\de\sigma,\\
		\int_{\D_i} \K_{D,2}[\varphi]\de\sigma=-\int_{D_i}\S_D[\varphi]\de x, \qquad &\int_{\D_i} \K_{D,3}[\varphi]\de\sigma=\frac{\iu|D_i|}{4\pi}\int_{\D}\varphi\de\sigma.
		\end{split}
		\end{equation}
	\end{lemma}

	A thorough presentation of other properties of the layer potential operators and their use in wave-scattering problems can be found in \emph{e.g.} \cite{ammari2018mathematical}.

\subsubsection{Generalized argument principle and generalized Rouch\'e's theorem}

The Gohberg-Sigal theory refers to the generalization to operator-valued functions of two classical results in complex analysis, the argument principle and Rouch\'e's theorem \cite{gohsig, gohberg2009holomorphic, ammari2018mathematical}.

Let $\mathcal{B}$ and $\mathcal{B}^\prime$ be two Banach spaces and denote by $\mathcal{L }(\mathcal{B},\mathcal{B}^\prime)$ the space of bounded linear operators from $\mathcal{B}$ into $\mathcal{B}^\prime$. A point $z_{0}$ is called a  {\it
characteristic value} of the operator-valued function  $z \mapsto A(z) \in \mathcal{L }(\mathcal{B},\mathcal{B}^\prime)$  if $A(z)$ is holomorphic in some neighborhood of $z_0$, except possibly at
$z_0$ and there
exists a vector-valued function $\phi(z)$ with values in
$\mathcal{B}$ such that
\begin{enumerate}
\item[(i)] $\phi(z)$ is holomorphic at $z_{0}$ and $\phi(z_{0}) \not= 0$,
\item[(ii)]
 $A(z)\phi(z)$ is holomorphic at $z_{0}$
 and vanishes at this point.
\end{enumerate}

Let $V$ be a simply connected bounded domain with rectifiable
boundary $\partial V$. An operator-valued function 
$A(z)$ is normal with respect to $\p V$ if  it is finitely
meromorphic and of Fredholm type in $V$, continuous on $\p {V}$, and
invertible for all $z\in \p V$.

If  $A(z)$ is normal
with respect to the contour $\partial V$  and $z_{j},$ $
j=1,\ldots,\sigma$, are all its characteristic values and poles lying in $V$, the full multiplicity $\mathcal{M}(A;\partial V)$ of $A(z)$ for
$ z \in V$ is the number of characteristic
 values of $A(z)$ for $ z\in V$, counted with their multiplicities, minus the
 number of poles of $A(z)$ in $V$, counted with their multiplicities:
 $$
\mathcal{M}(A;\partial V) := \sum_{j=1}^{\sigma} M ( A(z_{j})),
 $$
with $M ( A(z_{j}))$ being the multiplicity of $z_j$; see \cite[Chap. 1]{ammari2018mathematical}. 

The following results are from \cite{gohsig}. 
\begin{theorem}[Generalized argument principle] \label{thmprincipal}
Suppose that $A(z)$ is an operator-valued function which is normal
with respect to  $\partial V$. Let $f(z)$ be a scalar function which
is holomorphic in $ V$ and continuous in $\overline{ V}$. Then
 \begin{eqnarray*}
 \frac{1}{2\pi \i} \mbox{ {\rm tr} }\int_{\partial
 V} f(z)A(z)^{-1}\frac{d}{dz}A(z)dz =
 \sum_{j=1}^{\sigma} M ( A(z_{j}))f(z_{j}),
 \end{eqnarray*}
where $z_{j},$  $j=1,\ldots,\sigma$, are all the points in $ V$
which are either poles or characteristic values of $A(z)$.
\end{theorem}

A generalization \index{generalized Rouch\'e's theorem} of
Rouch{\'e}'s theorem to operator-valued functions is stated below.

\begin{theorem}[Generalized Rouch\'e's theorem] \label{rouche}
Let $A(z)$ be an operator-valued function which is normal with
respect to $\partial V$. If an operator-valued function $S(z)$
which is finitely meromorphic in $V$ and continuous on $\partial
V$ satisfies the condition
\begin{eqnarray*}
 \|A(z)^{-1}S(z)\|_{ \mathcal{L }(\mathcal{B},\mathcal{B})} < 1, \hspace{0.4cm}
z \in \partial V, \end{eqnarray*} then $A + S $ is also
normal with respect to $\partial V$ and
\begin{eqnarray*} \mathcal{M}(A;\partial V)=
\mathcal{M}(A + S;\partial V).
\end{eqnarray*}
\end{theorem}

\subsection{Capacitance matrix analysis}

\noindent The existence of subwavelength resonant frequencies is stated in the following theorem, which was proved in \cite{davies2019fully, ammari2018minnaert} using Theorem \ref{rouche}. 

\begin{theorem}
	A system of $N$ subwavelength resonators exhibits $N$ subwavelength resonant frequencies with positive real parts, up to multiplicity.
\end{theorem}
\begin{proof} 
The solution $u$ to the scattering problem \eqref{eq:scattering} can be represented as 
	\begin{equation} \label{eq:layer_potential_representation}
	u(x) = \begin{cases}
	u^{in}+\S_{D}^{k}[\psi](x), & x\in\outside,\\
	\nm
	\S_{D}^{k_b}[\phi](x), & x\in D,
	\end{cases}
	\end{equation} 
	for some surface potentials $(\phi,\psi)\in L^2(\D)\times L^2(\D)$, which must be chosen so that $u$ satisfies the transmission conditions across $\D$. Using the jump relation between $\S_{D}^k$ and $\K_{D}^{k,*}$, we see that in order to satisfy the transmission conditions on $\D$ the densities $\phi$ and $\psi$ must satisfy, for $x\in \D$,
	\begin{equation}
	\label{repformula}
	\begin{cases}
	\S_{D}^{k_b}[\phi](x)-\S_{D}^{k}[\psi](x)=u^{in}(x),\\
	\nm
	\left(-\frac{1}{2}I+\K_{D}^{k_b,*}\right)[\phi](x)-\delta\left(\frac{1}{2}I+\K_{D}^{k,*}\right)[\psi](x)=\delta \ddp{u^{in}}{\nu}(x).
	\end{cases}
	\end{equation}
Therefore,  $\phi$ and $\psi$ satisfy the following system of boundary integral equations:
\begin{equation}  \label{eq-boundary}
\mathcal{A}(\omega, \delta)[\Psi] =F,  
\end{equation}
where
\begin{equation} \label{page450}
\mathcal{A}(\omega, \delta) = 
 \begin{pmatrix}
  \mathcal{S}_D^{k_b} &  -\mathcal{S}_D^{k}  \\
  \nm
  -\frac{1}{2}I + (\mathcal{K}_D^{k_b})^*& -\delta( \frac{1}{2}I + (\mathcal{K}_D^{k})^*)
\end{pmatrix}, 
\,\, \Psi= 
\begin{pmatrix}
\phi\\
\psi
\end{pmatrix}, 
\,\,F= 
\begin{pmatrix}
u^{in}\\
\delta \frac{\partial u^{in}}{\partial \nu}
\end{pmatrix}.
\end{equation}

One can show that the scattering problem (\ref{eq:scattering}) is equivalent to the system of boundary integral equations (\ref{eq-boundary}). It is clear that $\mathcal{A}(\omega, \delta)$ is a bounded linear operator from $\mathcal{H}:= L^2(\p D) \times L^2(\p D)$ to $\mathcal{H}_1:=H^1(\p D) \times L^2(\p D)$. As defined in \Cref{def:res^}, the resonant frequencies to the scattering problem (\ref{eq:scattering}) are the complex numbers $\omega$ with positive imaginary part such that there exists a nontrivial solution to the following equation:
\begin{equation}  \label{eq-resonance}
\mathcal{A}(\omega, \delta)[\Psi] =0.
\end{equation}
These can be viewed as the characteristic values of the holomorphic operator-valued  function (with respect to $\omega$)
$\mathcal{A}(\omega, \delta)$. The subwavelength resonant frequencies lie in the right half of a small neighborhood of the origin in the complex plane. 
In what follows, we apply the Gohberg-Sigal theory to find these frequencies. 

We first look at the limiting case when $\delta =\omega=0$.
It is clear that

\begin{equation}  \label{eq-A_0-3d}
\mathcal{A}_0:= \mathcal{A}(0, 0) = 
 \begin{pmatrix}
  \mathcal{S}_D &  -\mathcal{S}_D  \\
  \nm
  -\frac{1}{2}I + \mathcal{K}_D^{*}& 0
\end{pmatrix},
\end{equation}
where $\mathcal{S}_D$ and $\mathcal{K}_D^{*}$ are respectively the single-layer potential and the Neumann--Poincar\'e operator on $\p D$ associated with the Laplacian.
 
Since $\mathcal{S}_D: L^2(\partial D) \rightarrow H^1(\partial D)$ is invertible in dimension three and $\mathrm{Ker } ( -\frac{1}{2}I + \mathcal{K}_D^{*})$ has dimension equal to the number of connected components of $D$, it follows that $\mathrm{Ker } (\mathcal{A}_0)$ is of dimension $N$. 
This shows that $\omega=0$ is a characteristic value for the holomorphic operator-valued  function $\mathcal{A}(\omega, 0)$ of full multiplicity $2N$. 
By the generalized Rouch\'e's theorem, we can conclude that for any $\delta$, sufficiently small, there exist $2N$ characteristic values to the holomorphic operator-valued  function $\mathcal{A}(\omega, \delta)$
such that $\omega_n(0)=0$ and 
$\omega_n$ depends on $\delta$ continuously. $N$ of these characteristic values,
$\omega_n= \omega_n(\delta), n=1, \ldots, N,$ have positive real parts, and these are precisely the subwavelength resonant frequencies of the scattering problem (\ref{eq:scattering}). 
\end{proof}

Our approach to approximate the subwavelength resonant frequencies is to study the \emph{(weighted) capacitance matrix}, which offers a rigorous discrete approximation to the differential problem. The eigenstates of this $N\times N$-matrix characterise, at leading order in $\delta$, the subwavelength resonant modes of the system of $N$ resonators.

In order to introduce the notion of capacitance, we define the functions $\psi_j$, for $j=1,\dots,N$, as 
\begin{equation} \label{defpsi}
\psi_j:=\S_D^{-1}[\chi_{\D_j}],
\end{equation}
where $\chi_A:\mathbb{R}^3\to\{0,1\}$ is used to denote the characteristic function of a set $A\subset\mathbb{R}^3$.	The capacitance matrix $C=(C_{ij})$ is defined, for $i,j=1,\dots,N$, as
\begin{equation} \label{defCap}
C_{ij}:=-\int_{\D_i} \psi_j\de\sigma.
\end{equation}
In order to capture the behaviour of an asymmetric array of resonators we need to introduce the weighted capacitance matrix $\C=(\C_{ij})$, given by
\begin{equation} \label{defCapw}
\C_{ij}:=\frac{1}{|D_i|} C_{ij},
\end{equation}
which accounts for the differently sized resonators (see \emph{e.g.} \cite{ ammari2020exceptional, davies2020close, ammari2017double} for other variants in different settings). 

We define the functions $S_n^\omega$, for $n=1\dots,N$, as
$$S_n^\omega(x) := \begin{cases}
\S_{D}^{k}[\psi_n](x), & x\in\outside,\\
\S_{D}^{k_b}[\psi_n](x), & x\in D.\\
\end{cases}
$$

\begin{lemma} \label{lem:modal}
	The solution to the scattering problem can be written, for $x\in\mathbb{R}^3$, as
	\begin{equation*} 
	u(x)-u^{in}(x) = \sum_{n=1}^N q_nS_n^\omega(x) - \S_D^k\left[\S_D^{-1}[u^{in}]\right](x) + \O(\omega),
	\end{equation*} 
	for coefficients $\underline{q}=(q_1,\dots,q_N)$ which satisfy, up to an error of order $\O(\delta \omega+\omega^3)$,
	\begin{equation*}\label{eq:eval_C}
	\left({\omega^2}-{v_b^2\delta}\,\C\right)\underline{q}
	=
	{v_b^2\delta}\begin{pmatrix} \frac{1}{|D_1|} \int_{\D_1}\S_D^{-1}[u^{in}]\de\sigma \\ \vdots\\
	\frac{1}{|D_N|}\int_{\D_N}\S_D^{-1}[u^{in}]\de\sigma \end{pmatrix}.
	\end{equation*}
\end{lemma}
\begin{proof}
Using the asymptotic expansions \eqref{eq:exp_S} and \eqref{eq:exp_K} for $\S_{D}^k$ and $\K_{D}^{k,*}$ in \eqref{repformula}, we can see that 
	\begin{equation*}\label{eq:psi}
	\psi=\phi-\S_D^{-1}[u^{in}]+ \O(\omega),
	\end{equation*} 
	and, further, that
	\begin{multline}
	\left(-\frac{1}{2}I+\K_D^*+\frac{\omega^2}{{v}_b^2}\K_{D,2}-\delta\left(\frac{1}{2}I+\K_D^*\right)\right)[\phi]=\\-\delta \left(\frac{1}{2}I+\K_D^*\right)\S_D^{-1}[u^{in}]+ \O(\delta\omega+\omega^3). \label{eq:phi}
	\end{multline}
	At leading order, \eqref{eq:phi} says that $\left(-\frac{1}{2}I+\K_D^{*}\right)[\phi]=0$ so, in light of the fact that $\{\psi_1,\dots,\psi_N\}$ forms a basis for $\mathrm{Ker } \left(-\frac{1}{2}I+\K_D^{*}\right)$, the solution can be written as
	\begin{equation} \label{eq:psi_basis}
	\phi=\sum_{n=1}^N q_n\psi_n+ \O(\omega^2+\delta),
	\end{equation}
	for coefficients $\underline{q}=(q_1,\dots,q_N)$.
	
	Finally, integrating \eqref{eq:phi} over $\D_i$, for $1\leq i\leq N$, gives us  that
	\begin{equation*}
	-\omega^2\int_{D_i}\S_D[\phi]\de x -{v}_b^2\delta\int_{\D_i}\phi\de\sigma=-{v_b^2\delta}\int_{\D_i}\S_D^{-1}[u^{in}]\de\sigma,  \label{eq:D}
	\end{equation*}
	up to an error of order $\O(\delta\omega+\omega^3)$. Substituting the expression \eqref{eq:psi_basis} gives the desired result.
\end{proof}

\begin{theorem} \label{thm:res}
	As $\delta \rightarrow 0$, the subwavelength resonant frequencies satisfy the asymptotic formula
	$$\omega_n = \sqrt{v_b^2\lambda_n\delta} -\i\tau_n\delta+ \O(\delta^{3/2}),$$
	for  $n = 1,\dots,N$, where $\lambda_n$ are the eigenvalues of the weighted capacitance matrix $\C$ and $ \tau_n$ are non-negative real numbers that depend on $C$, $v$ and $v_b$.
\end{theorem}
\begin{proof}
	If $u^{in} = 0$, we find from Lemma~\ref{lem:modal} that there is a non-zero solution to the resonance problem when $\omega^2/v_b^2\delta$ is an eigenvalue of $\C$, at leading order.
	
	To find the imaginary part, we adopt the ansatz
	\begin{equation} \label{eq:omega_ansatz}
	\omega_n=\sqrt{v_b^2\lambda_n\delta} -\i\tau_n\delta+ \O(\delta^{3/2}).
	\end{equation}
	Using a few extra terms in the asymptotic expansions for $\S_{D}^k$ and $\K_{D}^{k,*}$, we have that
	\begin{equation*}
	\psi=\phi+\frac{k_b-k}{4\pi\i}\left(\sum_{n=1}^N\psi_n\right)\int_{\D}\phi\de\sigma+ \O(\omega^2),
	\end{equation*} 
	and, hence, that
	\begin{multline*}
	\left(-\frac{1}{2}I+\K_D^*+k_b^2\K_{D,2}+k_b^3\K_{D,3}-\delta\left(\frac{1}{2}I+\K_D^*\right)\right)[\phi]\\-\frac{\delta(k_b-k)}{4\pi\i}\left(\sum_{n=1}^N\psi_n\right)\int_{\D}\phi\de\sigma= \O(\delta\omega^2+\omega^4).
	\end{multline*}
	We then substitute the decomposition \eqref{eq:psi_basis} and integrate over $\D_i$, for $i=1,\dots,N$, to find that, up to an error of order $\O(\delta\omega^2+\omega^4)$, it holds that
	\begin{equation*}
	\bigg(-\frac{\omega^2}{v_b^2}-\frac{\omega^3\i}{4\pi v_b^3}JC+\delta \C+\frac{\delta\omega\i}{4\pi}\bigg(\frac{1}{v_b}-\frac{1}{v}\bigg)\C JC \bigg)
	\underline{q}=0,
	\end{equation*}
	where $J$ is the $N\times N$ matrix of ones (\emph{i.e.} $J_{ij}=1$ for all $i,j=1,\dots,N$). 
	Then, using the ansatz \eqref{eq:omega_ansatz} for $\omega_n$ we see that, if $\underline{v}_n$ is the eigenvector corresponding to $\lambda_n$, it holds that
	\begin{equation}
	\tau_n= \frac{v_b^2}{8\pi v} \frac{\underline{v}_n\cdot C J C\underline{v}_n}{\| \underline{v}_n\|_D^2},
	\end{equation}
	where we use the norm $\| x\|_D:=\big(\sum_{i=1}^N |D_i| x_i^2\big)^{1/2}$ for $x\in\mathbb{R}^N$. Since $C$ is symmetric, we can see that $\tau_n\geq0$.
\end{proof}

It is more illustrative to rephrase Lemma~\ref{lem:modal} in terms of basis functions that are associated with the resonant frequencies. Denote by $\underline{v}_n=(v_{1,n},\dots,v_{N,n})$ the eigenvector of $\C$ with eigenvalue $\lambda_n$. Then, we have a modal decomposition with coefficients that depend on the matrix $V=(v_{i,j})$, assuming the system is such that $V$ is invertible. The following result follows from Lemma~\ref{lem:modal} by diagonalising the matrix $\C$.

\begin{lemma} \label{lem:modal_res}
	Suppose that the resonators' geometry is such that the matrix of eigenvectors $V$ is invertible. Then if $\omega=\O(\sqrt{\delta})$ the solution to the scattering problem can be written, for $x\in\mathbb{R}^3$, as
	\begin{equation*} 
	u^s(x):= u(x)-u^{in}(x) = \sum_{n=1}^N a_n u_n(x) - \S_D\left[\S_D^{-1}[u^{in}]\right](x) + \O(\omega),
	\end{equation*} 
	for coefficients which satisfy, up to an error of order $\O(\omega^3)$, 
	\begin{equation*}
	a_n(\omega^2-\omega_n^2)=-A\nu_n\Re(\omega_n)^2,
	\end{equation*}
	where $\nu_n=\sum_{j=1}^{N} [V^{-1}]_{n,j}$, \emph{i.e.} the sum of the $n$\textsuperscript{th} row of $V^{-1}$.
\end{lemma}

\begin{remark} \label{remark12}
When $N=1$, the subwavelength resonant frequency $\omega_1$ is called the Minnaert resonance. 
By writing an asymptotic expansion of
$\mathcal{A}(\omega, \delta)$ in terms of $\delta$ 
and applying the generalized argument principle (Theorem \ref{thmprincipal}), one can prove that
$\omega_1$  satisfies as $\delta$ goes to zero the asymptotic formula  \cite{ammari2018minnaert} 
\begin{equation} \label{defomegaM}
 \omega_1 = 
 \underbrace{{\sqrt{\frac{\mathrm{Cap}_D}{|D|}} v_b {\sqrt{\delta}}}}_{:= \omega_M}  - {\i} \underbrace{{\left(\frac{\mathrm{Cap}_D^2 v_b^2}{8 \pi  v |D|} {\delta}\right)}}_{:= \tau_M} + \O(\delta^{\frac{3}{2}}),
 \end{equation}
 where  \begin{equation} \label{defcap}
 \mathrm{Cap}_D:= \ds - \int_{\partial D} \mathcal{S}_D^{-1}[\chi_{\partial D}] \; \dx\sigma \end{equation} is the capacity of $\partial D$. Moreover, the following 
 {monopole approximation} of the scattered field for $\omega$ near $\omega_M$  holds \cite{ammari2018minnaert}:
\begin{equation} \label{monopole}
u^s(x) = {g(\omega,\delta, D)}(1+\O(\omega)+\O(\delta)+o(1))u^{in}(0) {G^{k}(x)},
\end{equation}
with the origin $0 \in D$ and the {scattering coefficient} $g$ being given by 
\begin{equation} \label{defg}
g(\omega,\delta, D) = \frac{\mathrm{Cap}_D}{1-{(\frac{\omega_M}{\omega})^2} + \i {\gamma_M}}, 
\end{equation}
where the damping constant $\gamma_M$ is given  by
$$
\gamma_M :=  \frac{(v+v_b) \mathrm{Cap}_D \omega}{8 \pi v v_b} - \frac{(v-v_b)}{v} \frac{v_b \mathrm{Cap}_D^2 \delta}{8 \pi |D| \omega}.
$$
This shows the scattering enhancement near $\omega_M$.

When $N=2$,  there are 
{two subwavelength resonances} with positive real part for the {resonator dimer} $D$. 
Assume that $D_1$ and $D_2$ are symmetric with respect to the origin $0$ and let $C_{ij}$, for $i,j=1,2$, be defined by \eqref{defCap}. Then, as $\delta \rightarrow 0$, by using Lemma \ref{lem:ints} it follows that \cite{ammari2017double}
\begin{align} \label{tau10}
 \omega_1 &=   \underbrace{{\sqrt{ (C_{11}+ C_{12})}}   v_b  {\sqrt{\delta}}}_{:= \omega_{M,1}}   - \i\tau_1 {\delta} +\O(\delta^{3/2}), 
 \end{align}
 \begin{align} \label{tau20}
 \omega_2 &= \underbrace{{\sqrt{(C_{11}-C_{12})}}  v_b  {\sqrt{\delta}}}_{:= \omega_{M,2}} +  {\delta^{3/2}} \hat\eta_1 + \i\delta^2 \hat\eta_2 + \O(\delta^{5/2}),
\end{align}
where $\hat\eta_1$ and $\hat\eta_2$ are real numbers  determined by $D$, $v$, and $v_b$ and  
$$ 
\tau_1 = \frac{v_b^2}{4\pi  v}(C_{11}+C_{12})^2.
$$ 
The resonances $\omega_1$ and $\omega_2$ are called the {hybridized} resonances of the resonator dimmer $D$. 

On the other hand, the resonator dimer can be approximated as a {point scatterer} with {resonant monopole} and {resonant dipole} modes. Assume that $D_1$ and $D_2$ are symmetric with respect to the origin. Then for $\omega=\O(\delta^{1/2})$ and  $\delta \rightarrow 0$, and $|x|$ being sufficiently large, we have \cite{ammari2017double}
\begin{equation} \label{dimerh} \begin{array}{lll}
u^s(x)&=&  \underbrace{g^0(\omega)u^{in}(0)  {G^k(x)}}_{{monopole}} \\
\nm && + \underbrace{\nabla u^{in}(0)\cdot g^1(\omega)  {\nabla G^k(x)}}_{{dipole}} +\O(\delta|x|^{-1}),
\end{array}
\end{equation}
where the
{scattering coefficients} $g^0(\omega)$ and $g^1(\omega)=(g^1_{ij}(\omega))$ are given by
\begin{align}
&g^0(\omega)=\frac{C(1,1)}{{1- \omega_1^2/\omega^2}}(1+\O(\delta^{1/2})), \quad C(1,1):= C_{11} + C_{12} + C_{21} + C_{22},\\
&~g^1_{ij}(\omega)= \int_{\partial D}  \mathcal{S}_D^{-1}[x_i](y) y_j - \frac{\delta v_b^2}{\omega^2 |D|({1- \omega_2^2/\omega^2})} P^2\delta_{i1}\delta_{j1},
\end{align}
with \begin{align}
 P:=  \int_{\partial D} y_1(\psi_1-\psi_2)(y) \; \dx \sigma(y), \label{defP}
\end{align}
$\psi_i$, for $i=1,2,$ being defined by (\ref{defpsi}), and  $\delta_{i1}$ and $\delta_{j1}$ being the Kronecker delta. 

As shown in \eqref{monopole}-\eqref{defg}, around $\omega_M$, a single resonator in free-space scatters waves with a greatly enhanced amplitude. If a second resonator is introduced, coupling interactions will occur giving according to \eqref{dimerh} a system that has both monopole and dipole resonant modes. This pattern continues for larger number $N$ of resonators  \cite{davies2019fully}. 

\end{remark}

\begin{remark}
	The invertibility of $V$ is a subtle issue and depends only on the geometry of the inclusions $D=D_1\cup\dots\cup D_N$. In the case that the resonators are all identical, $V$ is clearly invertible since $\C$ is symmetric. 
\end{remark}

\begin{remark}
	In many cases $\tau_n=0$ for some $n$ (see for instance formula (\ref{tau20})), meaning the imaginary parts exhibit higher-order behaviour in $\delta$. For example, the second (dipole) frequency for a pair of identical resonators is known to be $\O(\delta^{2})$ \cite{ammari2017double}. In any case, the resonant frequencies will have negative imaginary parts, due to the radiation losses.
\end{remark}

\section{Close-to-touching subwavelength resonators} \label{sec3}

In this section, we study the behaviour of the coupled subwavelength resonant modes when two subwavelength resonators are brought close together. We consider the case of a pair of spherical resonators and use bispherical coordinates to derive explicit representations for the capacitance coefficients which, as shown in Theorem \ref{thm:res}, capture the system's resonant behaviour at leading order. 
We derive estimates for the rate at which the gradient of the scattered wave blows up as the resonators are brought together.

For simplicity, we study the effect of scattering by a pair of spherical inclusions, $D_1$ and $D_2$, with the same radius, which we denote by $r$,  and separation distance $\epsilon$ (so that their centres are separated by $2r +\epsilon$). We refer to \cite{davies2020close} for the case of arbitrary sized spheres.

We choose the separation distance $\epsilon$ as a function of $\delta$ and will perform an asymptotic analysis in terms of $\delta$. We choose $\epsilon$ to be such that, for some $0<\beta<1$,
\begin{equation} \label{assum_epsilon}
\epsilon\sim e^{-1/\delta^{1-\beta}} \text{ as } \delta\to0.
\end{equation}
As we will see shortly, with $\epsilon$ chosen to be in this regime the subwavelength resonant frequencies are both well behaved (\emph{i.e.} $\omega=\omega(\delta)\to0$ as $\delta\to0$) and we can compute asymptotic expansions in terms of $\delta$.

From Theorem \ref{thm:res} (see also Remark \ref{remark12}), 
	there exist two subwavelength resonant modes, $u_1$ and $u_2$, with associated resonant frequencies $\omega_1$ and $\omega_2$ with positive real part, labelled such that $\Re(\omega_1)<\Re(\omega_2)$.

\begin{figure}
    \centering
    \captionsetup{width=.6\linewidth}
	\includegraphics[width=.7\linewidth]{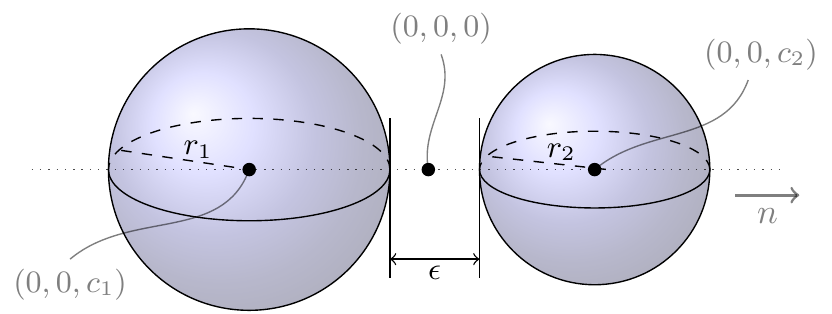}
    \caption{Two close-to-touching spheres, annotated with the bispherical coordinate system outlined in \Cref{sec:coordinates}.}
    \label{fig:coordinates}
\end{figure}

\subsection{Coordinate system} \label{sec:coordinates}

The Helmholtz problem \eqref{eq:scattering} is invariant under translations and rotations so we are free to choose the coordinate axes. Let $R_j$ be the reflection with respect to $\partial D_j$ and let $p_1$ and $p_2$ be the unique fixed points of the combined reflections $R_1\circ R_2$ and $R_2\circ R_1$, respectively. Let $n$ be the unit vector in the direction of $p_2-p_1$. We will make use of the Cartesian coordinate system $(x_1,x_2,x_3)$ defined to be such that $p=(p_1+p_2)/2$ is the origin and the $x_3$-axis is parallel to the unit vector $n$. Then one can see that \cite{kang2019quantitative}
\begin{equation}
p_1 = (0,0,-\alpha)\quad \mbox{and} \quad p_2 = (0,0,\alpha),
\end{equation}
where
\begin{equation}
\alpha:=\frac{\sqrt{\epsilon (4r +\epsilon)}}{2}.
\end{equation}
Moreover, the sphere $D_i$ is centered at $(0,0,c_i)$ where
\begin{equation} \label{eq:centres}
c_i=(-1)^i \sqrt{r^2+\alpha^2}.
\end{equation}

We then introduce a bispherical coordinate system $(\xi,\theta,\varphi)$ which is related to the Cartesian coordinate system $(x_1,x_2,x_3)$ by
\begin{equation} \label{def:bispherical_coordinates}
x_1=\frac{\alpha\sin\theta\cos\varphi}{\cosh\xi-\cos\theta}\,,\quad 
x_2=\frac{\alpha\sin\theta\sin\varphi}{\cosh\xi-\cos\theta}\,,\quad 
x_3=\frac{\alpha\sinh\xi}{\cosh\xi-\cos\theta}\,,
\end{equation}
and is chosen to satisfy $-\infty<\xi<\infty$, $0\leq\theta<\pi$ and $0\leq\varphi<2\pi$. The reason for this choice of coordinate system is that $\D_1$ and $\D_2$ are given by the level sets
\begin{equation}
\D_1=\big\{\xi=- \sinh^{-1}\left(\frac{\alpha}{r}\right) \big\},\qquad \D_2= \big\{\xi= \sinh^{-1}\left(\frac{\alpha}{r}\right) \big\}.
\end{equation}
This is depicted in \Cref{fig:coordinates} (for arbitrary sized spheres). The Cartesian coordinate system is chosen so that we can define a bispherical coordinate system \eqref{def:bispherical_coordinates} such that the boundaries of the two resonators are convenient level sets.

\subsection{Resonant frequency hybridization and gradient blow-up}

Firstly, the resonant frequencies are given, in terms of the capacitance coefficients, by (see \eqref{tau10} and \eqref{tau20})
\begin{equation} \label{eq:sym_resonances}
\begin{split}
	\omega_1&=\sqrt{\delta \frac{3v_b^2}{4\pi r^3}(C_{11}+C_{12})}+\O(\delta),\\
	\omega_2&=\sqrt{\delta \frac{3v_b^2}{4\pi r^3}(C_{11}-C_{12})}+\O(\delta).
\end{split}
\end{equation}
Further to this, the capacitance coefficients are given by
\begin{equation} \label{eq:sym_cap}
\begin{split}
C_{11}&=C_{22}=8\pi\widetilde{\alpha}\sum_{n=0}^{\infty} \frac{e^{(2n+1)\xi_0}}{e^{2(2n+1)\xi_0}-1}, \\
C_{12}&=C_{21}=-8\pi\widetilde{\alpha}\sum_{n=0}^{\infty}\frac{1}{e^{2(2n+1)\xi_0}-1},
\end{split}
\end{equation}
where
\begin{equation*}
\widetilde{\alpha}:=\sqrt{\epsilon(r+\epsilon/4)}, \qquad
\xi_0:=\sinh^{-1}\left(\frac{\widetilde{\alpha}}{r}\right).
\end{equation*}
From \cite{lekner2011near}, we know the asymptotic behaviour of the series in \eqref{eq:sym_cap} as $\xi_0\to0$, from which we can see that as $\epsilon\to0$,
\begin{equation} \label{eq:cap_sym_asym}
\begin{split}
C_{11}&=2\pi \frac{\widetilde{\alpha}}{\xi_0}\left[\gamma+2\ln 2+\ln \left(\sqrt{r}\right)-\ln \left(\sqrt{\epsilon}\right)\right]+\O(\epsilon),\\
C_{12}&=-2\pi \frac{\widetilde{\alpha}}{\xi_0}\left[\gamma+\ln \left(\sqrt{r}\right)-\ln \left(\sqrt{\epsilon}\right)\right]+\O(\epsilon),
\end{split}
\end{equation}
where $\gamma\approx0.5772\dots$ is the Euler–Mascheroni constant.

Combining \eqref{eq:sym_resonances} and \eqref{eq:cap_sym_asym} we reach the fact that the resonant frequencies are given, as $\delta\to0$, by
\begin{equation} \label{eq:res_identical}
\begin{split}
\omega_1&=\sqrt{\delta \frac{3v_b^2\ln 2}{r^2}}+ \O\left(\delta\right),\\
\omega_2&=\sqrt{\delta \frac{3v_b^2}{2r^2} \left(\ln \left(\frac{r}{\epsilon}\right)+2\gamma+2\ln 2\right)}+ \O\left(\sqrt{\delta}\right).
\end{split}
\end{equation}
Thus, the choice of $\epsilon\sim e^{-1/\delta^{1-\beta}}$, where $0<\beta<1$, means that as $\delta\to0$ we have that  $\omega_1\sim\sqrt{\delta}$ and $\omega_2\sim\delta^{\beta/2}$.

The two resonant modes, $u_1$ and $u_2$, correspond to the two resonators oscillating in phase and in antiphase with one another, respectively. Since the eigenmode $u_2$ has different signs on the two resonators, $\nabla u_2$ will blow up as the two resonators are brought together. Conversely, $u_1$ takes the same value on the two resonators so there will not be a singularity in the gradient. In particular, if we normalise the eigenmodes so that for any $x\in\D$
\begin{equation}
\lim_{\delta\to0} |u_1(x)|\sim1,\qquad \lim_{\delta\to0} |u_2(x)|\sim1,
\end{equation}
then the choice of $\epsilon$ to satisfy the regime $\epsilon\sim e^{-1/\delta^{1-\beta}}$ means that the maximal gradient of each eigenmode has the asymptotic behaviour, as $\delta\to0$,
\begin{equation} \label{eq:sym_mode_blowup}
\max_{x\in\outside}|\nabla u_1(x)|\sim 1, \qquad
\max_{x\in\outside}|\nabla u_2(x)|\sim \frac{1}{\epsilon}.
\end{equation}
By decomposing the scattered field into the two resonant modes, we can use \eqref{eq:sym_mode_blowup} to understand the singular behaviour exhibited by the acoustic pressure. The solution $u$ to the scattering problem \eqref{eq:scattering} with incoming plane wave $u^{in}$ with frequency $\omega=\O(\delta^{1/2})$ is given, for $x\in\outside$, by
\begin{equation}
u(x)=u^{in}(x) +au_1(x)+bu_2(x),
\end{equation}
where the coefficients $a$ and $b$ satisfy, as $\delta\to0$, the equations
\begin{align*}
a(\omega^2-\omega_1^2)&= \frac{\delta v_b^2}{|D|}\int_{\D} \S_D^{-1}[u^{in}] \de\sigma +\O(\delta^{\hat{\beta}}),\\
b(\omega^2-\omega_2^2)&= -\frac{\delta v_b^2}{|D|}\left(\int_{\D_1} \S_D^{-1}[u^{in}] \de \sigma-\frac{|D_1|}{|D_2|}\int_{\D_2} \S_D^{-1}[u^{in}] \de\sigma \right)+\O(\delta^{\hat{\beta}}),
\end{align*}
with $\hat{\beta}:=\min(2-\beta,3/2)$ and $|D|$ being the volume of $D=D_1\cup D_2$.

\section{Effective medium theory for subwavelength resonators} \label{dilutesect}

\subsection{High refractive index effective media} 
We  consider a domain $\Omega$ which contains a large number of small, identical resonators. If $D_0$ is a fixed domain, then for some $r>0$ the $N$ resonators are given, for $1\leq j\leq N$, by
	\begin{equation*}
	D_{0,j}^{r,N}=r D_0+z_j^N,
	\end{equation*}
	for positions $z_j^N$. We always assume that $r$ is sufficiently small such that the resonators are not overlapping and that $D_{0}^{r,N}=\bigcup_{j=1}^N D_{0,j}^{r,N} \Subset \Omega$. 
	
We find the effective equation in the specific case that the frequency $\omega =\O(1)$ and satisfies 
\begin{equation} \label{ass:frequency}
1- (\frac{\omega_M}{\omega})^2=  \beta_0 r^{\epsilon_0},
\end{equation}	
for some fixed $ 0<\epsilon_0<1$ and constant $\beta_0$. 
We note that there are two cases depending on whether $\omega > \omega_M$ or $\omega < \omega_M$. In the former case, $\beta_0 >0$ while in the latter case we have $\beta_0<0$. 

Moreover, we assume that there exists some positive number $\Lambda$ independent of $N$ such that
	\begin{equation} \label{ass:number}
	r^{1-\epsilon_0}N=\Lambda \quad \mbox{and } \Lambda \mbox{ is large}.
	\end{equation} 

Since the resonators are small, we can use the point-scatter approximation from \eqref{monopole} to describe how they interact with incoming waves. To do so, we must make some extra assumptions on the regularity of the distribution $\{z_j^N:1\leq j\leq N\}$ so that the system is well behaved as $N\to\infty$ (under the assumption \eqref{ass:number}). In particular, we assume that there exists some constant $\eta$ such that for any $N$ it holds that
	\begin{equation} \label{ass:dist}
	\min_{i\neq j} |z_i^N-z_j^N| \geq \frac{\eta}{N^{1/3}},
	\end{equation}
	and, further, there exists some $0<\epsilon_1<1$ and constants $C_1,C_2>0$ such that for all $h\geq 2\eta N^{-1/3}$,
	\begin{align}
	\sum_{|x-z_j^N|\geq h} \frac{1}{|x-y_j^N|^2}\leq C_1 N |h|^{-\epsilon_1}, \qquad&\text{uniformly for all } x\in\Omega,\label{ass:reg1}\\
	\sum_{2\eta N^{-1/3}\leq|x-z_j^N|\leq 3h} \frac{1}{|x-y_j^N|}\leq C_2 N |h|, \qquad&\text{uniformly for all } x\in\Omega. \label{ass:reg2}
	\end{align}
	Finally, we  also need that 
	\begin{equation} \label{ass:epsilon}
	\epsilon_2:=\frac{\epsilon_0}{1-\epsilon_0}-\frac{\epsilon_1}{3}>0.
	\end{equation}
	If we represent the field that is scattered by the collection of resonators $D_0^{r,N}=\bigcup_{j=1}^N D_{0,j}^{r,N}$ as
	\begin{equation*} 
	u^N(x)= \begin{cases}
	u^{in}(x)+\S_{D_0^{r,N}}^{k}[\psi^N](x), & x\in\R^3 \setminus \overline{D_0^{r,N}},\\
	\S_{D_0^{r,N}}^{k_0}[\phi^N](x), & x\in D_0^{r,N},
	\end{cases}
	\end{equation*} 
	for some $\psi^N, \phi^N\in L^2(\p D_0^{r,N})$,	then we have the following lemma, which follows from \cite[Proposition~3.1]{ammari2017effective}. This justifies using a point-scatter approximation to describe the total incident field acting on the resonator $D_{0,j}^{r,N}$ and the scattered field due to $D_{0,j}^{r,N}$, defined respectively as
	\begin{equation*}
	u_j^{in,N}=u^{in}+\sum_{i\neq j} \S_{D_{0,i}^{r,N}}^k[\psi^N] \qquad \text{and} \qquad 	u_j^{s,N}=\S_{D_{0,j}^{r,N}}^k[\psi^N].
	\end{equation*}

	\begin{lemma} \label{lem:points}
		Under the assumptions \eqref{ass:frequency}--\eqref{ass:epsilon}, it holds that the total incident field acting on the resonator $D_{0,j}^{r,N}$ is given, at $z_j^N$, by
		\begin{equation*}
		u_j^{in,N}(z_j^N)=u^{in}(z_j^N)+\sum_{i\neq j} \frac{r\mathrm{Cap}_{D_0}}{1-(\frac{\omega_M}{\omega})^2} G^k(z_j^N-z_i^N)u^{in}(z_j^N),
		\end{equation*}
		up to an error of order $\O(N^{-\epsilon_2})$. Similarly, it holds that the scattered field due to the resonator $D_{0,j}^{r,N}$ is given, at $x$ such that $|x-z_j^N|\gg r$, by
		\begin{equation*}
		u_j^{s,N}(x)=  \frac{r\mathrm{Cap}_{D_0}}{1-(\frac{\omega_M}{\omega})^2} G^k(x-z_j^N) u_j^{in,N}(z_j^N) ,
		\end{equation*}
		up to an error of order $\O(N^{-\epsilon_2}+r|x-z_j^N|^{-1})$.
	\end{lemma}
	
	In order for the sums in \ref{lem:points} to be well behaved as $N\to\infty$, we make one additional assumption on the regularity of the distribution: that there exists a real-valued function $\widetilde{V}\in \mathcal{C}^1(\overline{\Omega})$ such that for any $f\in \mathcal{C}^{0,\alpha}(\Omega)$, with $0<\alpha\leq1$, there is a constant $C_3$ such that 
	\begin{equation} \label{ass:integral}
	\max_{1\leq j\leq N} \left| \frac{1}{N}\sum_{i\neq j} G^k(z_j^N-z_i^N)f(z_i^N)-\int_\Omega G^k(z_j^N-y)\widetilde{V}(y)f(y) \de y \right| \leq C_3\frac{1}{N^{\alpha/3}}\|f\|_{\mathcal{C}^{0,\alpha}(\Omega)}.
	\end{equation}
	\begin{remark}
		It  holds that $\widetilde{V}\geq0$. If the resonators' centres $\{z_j^N:j=1,\dots,N\}$ are uniformly distributed, then $\widetilde V$ will be a positive constant, $\widetilde V = {1}/{|\Omega|}.$
	\end{remark}

	Under all these assumptions, we are able to derive effective equations for the system with an arbitrarily large number of small resonators. If we let $\epsilon_3\in(0,\tfrac{1}{3})$, then we will seek effective equations on the set given by
	\begin{equation} \label{defomegaN}
	Y_{\epsilon_3}^N:=\left\{x\in\mathbb{R}^3:|x-z_j^N|\geq N^{\epsilon_3-1} \text{ for all }1\leq j\leq N\right\},
	\end{equation}
	which is the set of points that are sufficiently far from the resonators, avoiding the singularities of the Green's function. The following result from \cite{ammari2017effective} holds. 
	
	\begin{theorem} \label{thm:homogenized}
		Let $\omega < \omega_M$. Under the assumptions \eqref{ass:number}--\eqref{ass:integral}, the solution $u^N$ to the scattering problem \eqref{eq:scattering} with the system of resonators $D_0^{r,N}=\bigcup_{j=1}^N D_{0,j}^{r,N}$ converges to the solution of
		\begin{equation*}
		\begin{cases}
		\left(\Delta+k^2-\frac{\Lambda \mathrm{Cap}_{D_0}}{\beta_0}\widetilde{V}(x)\right)u(x)=0, &  x\in\Omega, \\
		\left(\Delta+k^2\right)u(x)=0, & x\in\mathbb{R}^3\setminus \overline{\Omega},\\
		u|_+ = u|_- & \mbox{on } \partial \Omega, 	
		\end{cases}
		\end{equation*}
		as $N\to\infty$, together with a radiation condition governing the behaviour in the far field, which says that uniformly for all $x\in Y_{\epsilon_3}^N$ it holds that
		\begin{equation*}
		|u^N(x)-u(x)|\leq C N^{-\min\left\{\frac{1-\epsilon_0}{6},\epsilon_2,\epsilon_3,\frac{1-\epsilon_3}{3}\right\}}.
		\end{equation*}
	\end{theorem}
	
	By our assumption, $k=\O(1)$, $\widetilde{V}=\O(1)$, and $\beta_0 <0$. When $- \Lambda \mathrm{Cap}_{D_0}/{\beta_0} \gg 1$, we see that we have an effective high refractive index medium. As a consequence, this together with \cite{ammari2015mathematical} gives a rigorous mathematical theory for the super-focusing experiment in \cite{fink}. Similarly to Theorem \ref{thm:homogenized}, if $\omega > \omega_M$, to the scattering problem \eqref{eq:scattering} with the system of resonators $D_0^{r,N}$ converges to the solution of the following dissipative equation
		\begin{equation*}
		\begin{cases}
		\left(\Delta+k^2-\frac{\Lambda \mathrm{Cap}_{D_0}}{\beta_0}\widetilde{V}(x)\right)u(x)=0, &  x\in\Omega, \\
		\left(\Delta+k^2\right)u(x)=0, & x\in\mathbb{R}^3\setminus \overline{\Omega},\\
		u|_+ = u|_- & \mbox{on } \partial \Omega, 	
		\end{cases}
		\end{equation*}
		as $N\to\infty$, together with a radiation condition governing the behaviour in the far field, which says that uniformly for all $x\in Y_{\epsilon_3}^N$ it holds that
		\begin{equation*}
		|u^N(x)-u(x)|\leq C N^{-\min\left\{\frac{1-\epsilon_0}{6},\epsilon_2,\epsilon_3,\frac{1-\epsilon_3}{3}\right\}}.
		\end{equation*}

\begin{remark}
At the resonant frequency $\omega= \omega_M$, the scattering coefficient $g$ defined by (\ref{defg}) is of order one. 
Thus each scatterer is  a point source with magnitude one.  As a consequence, the addition or removal of one resonator from the medium affects the total field by a magnitude of the same order as the incident field. Therefore, we cannot expect any effective medium theory for the  medium at this resonant frequency.  
\end{remark}

\subsection{Double-negative metamaterials}

In this subsection, we show  that, using dimers of identical subwavelength resonators, the effective material parameters of dilute  system of dimers can both be negative over a non empty range of frequencies \cite{ammari2017double}. 
As shown in \eqref{dimerh}, a dimer of identical subwavelength resonators can be approximated as a point scatterer with monopole and dipole modes.  It features two slightly different subwavelength resonances, called the hybridized  resonances. The hybridized resonances are fundamentally different modes. The first mode is, as in the case of a single resonator,  a monopole mode, while the second mode is a dipole mode. The resonance associated with the dipole mode is usually referred to as the anti-resonance. 

For an appropriate volume fraction, when the excitation frequency is close to the anti-resonance, a double-negative effective $\rho$ and $\kappa$ for media consisting of a large number of  dimers with certain conditions on their distribution can be obtained. The dipole modes in the background medium contribute to the effective $\rho$ while the monopole modes contribute to the effective $\kappa$.

Here we consider the scattering of an incident plane wave $u^{in}$ by $N$ identical dimers with different orientations distributed in a homogeneous medium in $\R^3$.  The $N$ identical dimers are generated by scaling the normalized  dimer $D$ by a factor $r$, and then rotating the orientation and translating the center. More precisely,  the dimers occupy the domain
$$
D^N:=\cup_{1\leq j \leq N}D_j^N,
$$ 
where $D_j^N=z_j^N + r R_{d_j^N}D$ for $1\leq j \leq N$, with $z_j^N$ being the center of the dimer $D_j^N$, $r$ being the characteristic size, and $R_{d_j^N}$ being the rotation in $\R^3$ which aligns the dimer $D_j^N$ in the direction $d_j^N$.  Here, $d_j^N$ is a vector of unit length in $\R^3$. For simplicity, we suppose that $D$ is made of two identical spherical resonators. 
We also assume that $0< r \ll 1$, $N\gg 1$ and that $\{z_j^N: 1\leq j \leq N\} \subset \Omega$ where $\Omega$ is a bounded domain.  See Figure \ref{doublef}.

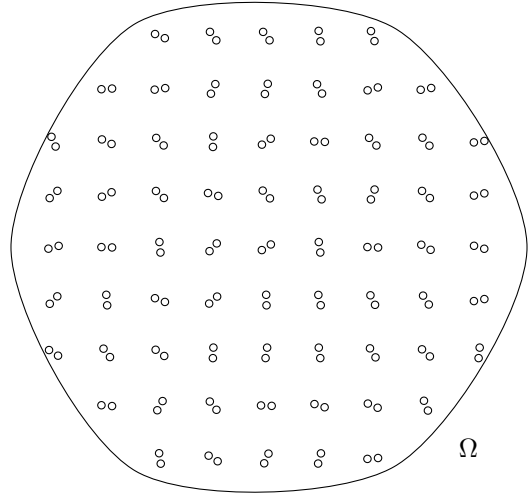
\begin{figure}
\begin{center}
\begin{tikzpicture}[scale=0.7]
	\pgfmathsetmacro{\rad}{0.07pt}	
	\pgfmathsetmacro{\sep}{0.1pt}
	
	\foreach \x in {-3,...,3}
	\foreach \y in {-3,...,3}
	{
	\begin{scope}[xshift = \x cm, yshift=\y cm, rotate=rand*360]
	\draw (-\sep,0) circle (\rad);
	\draw (\sep,0) circle (\rad);
	\end{scope}
	}	
	\foreach \c in {-2,...,2}
	{
		\begin{scope}[xshift = \c cm, yshift=4 cm, rotate=rand*360]
		\draw (-\sep,0) circle (\rad);
		\draw (\sep,0) circle (\rad);
		\end{scope}
		\begin{scope}[xshift = \c cm, yshift=-4 cm, rotate=rand*360]
		\draw (-\sep,0) circle (\rad);
		\draw (\sep,0) circle (\rad);
		\end{scope}
		\begin{scope}[xshift = 4 cm, yshift=\c cm, rotate=rand*360]
		\draw (-\sep,0) circle (\rad);
		\draw (\sep,0) circle (\rad);
		\end{scope}
		\begin{scope}[xshift = -4 cm, yshift=\c cm, rotate=rand*360]
		\draw (-\sep,0) circle (\rad);
		\draw (\sep,0) circle (\rad);
		\end{scope}
	}	

	\draw[name path = O] plot [smooth cycle] coordinates {(0:4.9) (60:4.8) (120:4.9) (180:4.8) (240:4.9) (300:4.8) };
	\draw (3.8,-3.8) node{$\Omega$};
\end{tikzpicture}
\end{center} \caption{Illustration of the dilute system of subwavelength dimers.} \label{doublef}
\end{figure}

The scattering of  waves by the dimers can be modeled by the following system of equations: 
\begin{equation} \label{eq-scattering2}
\left\{
\begin{array} {ll}
&\ds \nabla \cdot \frac{1}{\rho} \nabla  u^N+ \frac{\omega^2}{\kappa} u^N  = 0 \quad \mbox{in } \R^3 \backslash \overline{D^N}, \\
\nm
&\ds \nabla \cdot \frac{1}{\rho_b} \nabla  u^N+ \frac{\omega^2}{\kappa_b} u^N  = 0 \quad \mbox{in } D^N, \\
\nm
& \ds u^N_{+} -u^N_{-}  =0   \quad \mbox{on } \partial D^N, \\
\nm
&  \ds \frac{1}{\rho} \frac{\p u^N}{\p \nu} \bigg|_{+} - \frac{1}{\rho_b} \frac{\p u^N}{\p \nu} \bigg|_{-} =0 \quad \mbox{on } \partial D^N,\\
\nm
&  \ds u^N- u^{in}  \,\,\,  \mbox{satisfies the Sommerfeld radiation condition}, 
  \end{array}
 \right.
\end{equation}
where $u^N$ is the total field.

We make the following assumptions:
\begin{itemize}
\item[(i)] 
$\delta = \mu^2 r^2$ for some positive number $\mu >0$;
\item[(ii)] $\omega = \omega_{M,2} + a r^2$ for some real number $a < \mu^3 
\hat \eta_1 $, where $\omega_{M,2}$ is defined in (\ref{tau20}); 
\item[(iii)] $r N = \Lambda $ for some positive number $\Lambda >0$; 
\item[(iv)] The dimers are regularly distributed in the sense that
 \[
\min_{i \neq j } |z^N_i -z^N_j|  \geq  \eta N^{-\frac{1}{3}},
\]
for some constant $\eta$ independent of $N$; 
\item[(v)] 
There exists a function $\widetilde V \in \mathcal{C}^1(\bar{\Omega})$ such that for any $f \in \mathcal{C}^{0, \alpha}(\Omega)$ with $0 < \alpha \leq 1$, \eqref{ass:integral} holds 
for some constant $C$ independent of $N$;
\item[(vi)] 
There exists a matrix valued function $\widetilde B \in \mathcal{C}^1(\bar{\Omega})$ such that
for $f \in (\mathcal{C}^{0, \alpha}(\Omega))^3$ with $0 < \alpha \leq 1$, 
$$ \begin{array}{l} \ds
\max_{1\leq j \leq N} |\frac{1}{N} \sum_{i\neq j} (f(z_i^N )\cdot d_i^N)( d_i^N \cdot \nabla G^k(z_i^N - z_j^N)) -
\int_{\Omega} f(y) \widetilde B \nabla_y G^k(y-z_j^N) \; \dx y| \\
\nm \qquad \ds \leq C \frac{1}{N^{\frac{\alpha}{3}}}\|f\|_{\mathcal{C}^{0, \alpha}(\Omega)}  
\end{array} $$
for some constant $C$ independent of $N$;
\item[(vii)] 
There exists a constant $C>0$ such that
$$
\max_{1\leq j \leq N}  \frac{1}{N} \sum_{i \neq j} \frac{1}{|z_j^N-z_i^N|} 
  \leq C,\quad \max_{1\leq j \leq N}  \frac{1}{N} \sum_{i \neq j} \frac{1}{|z_j^N-z_i^N|^2} 
  \leq C,
$$
for all $1\leq j \leq N$. 
\end{itemize}

We introduce the two constants  
$$
\widetilde g^0 = \frac{2(C_{11}+C_{12})}{1- \omega_{M,1}^2/\omega_{M,2}^2}, \quad \widetilde g^1 = 
\frac{\mu^2  v_b^2 }{2|D|\omega_{M,2} (\mu^3 \hat \eta_1 -a)} P^2, 
$$
where $P$ is defined by (\ref{defP}), $\omega_{M,1}$ and $\omega_{M,2}$ are the leading orders  in  the asymptotic expansions (\ref{tau10}) and (\ref{tau20}) of the hybridized resonant frequencies  as $\delta\rightarrow 0$, and the two functions $$
M_1 = \begin{cases}
I &\quad \mbox{in } \mathbb{R}^3\setminus {\Omega},
 \\
I- \Lambda \widetilde g^1 \widetilde B &\quad \mbox{in }   \Omega,
 \end{cases}
 $$
 and
 $$
 M_2 = \begin{cases}
 k^2 &\quad \mbox{in }  \mathbb{R}^3\setminus {\Omega},
 \\
 k^2- \Lambda\widetilde g^0 \widetilde V &\quad \mbox{in } \Omega.
 \end{cases}
$$
 The following result from \cite{ammari2017double} holds. 
\begin{theorem}  \label{lem-pointapprox1}
Suppose that there exists a unique solution $u$ to
\begin{equation} \label{eq-pde}
\nabla \cdot M_1(x) \nabla u(x) + M_2(x)u(x) =0 \quad \mbox{in } \mathbb{R}^3,
\end{equation}
such that $u-u^{in}$ satisfies the Sommerfeld radiation condition at infinity. Then, under  assumptions [(i)]--[(vii)], we have $u^N(x) \rightarrow u(x)$ uniformly for $x\in\mathbb{R}^3$ such that $|x-z_j^N|\gg N^{-1} \text{ for all }1\leq j\leq N$. 
\end{theorem}

Note that from (\ref{eq:sym_cap}), it follows that $\omega_{M,2} > \omega_{M,1}$. 
Therefore, for large $\Lambda$, both the matrix $M_1$ and the scalar function $M_2$ are negative in $\Omega$. See Figure \ref{effectivef}.

\begin{figure} 
\centering
{\includegraphics[scale=0.7]{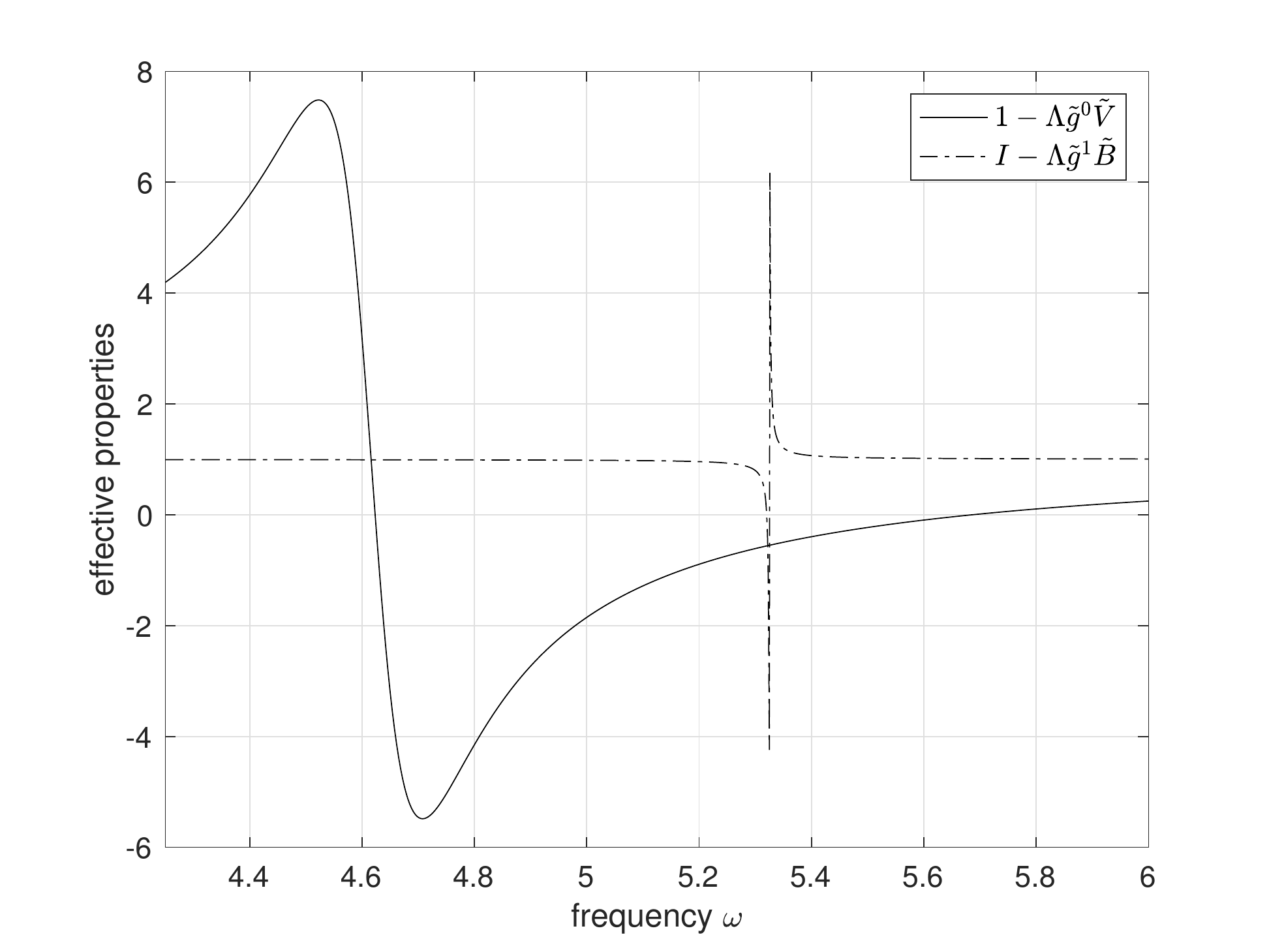}} \caption{Effective properties of the homogenized medium.} \label{effectivef}
\end{figure}

\section{Periodic structures of subwavelength resonators} \label{sec5}

In this section we investigate whether  there is a possibility of  subwavelength 
band gap opening in subwavelength resonator crystals. 
We first formulate the spectral problem for a subwavelength resonator crystal. Then we derive an asymptotic formula for the quasi-periodic  resonances in terms of the contrast between the densities outside and inside the resonators. We prove the existence of a subwavelength band gap and estimate its width.

\subsection{Floquet theory} \label{sec:floquet}
Let $f(x)$ for $x \in \R^d, d=1,2,3,$  be a function decaying sufficiently fast. We let $l_1,..., l_d$ be linearly independent lattice vectors, and define the unit cell $Y$ and the lattice $\Lambda$ as 
$$Y = \left\{\sum_{n=1}^d s_n l_n \ \Big| \ 0 < s_n < 1 \right\}, \qquad \Lambda = \left\{\sum_{n=1}^d q_n l_n \ \Big| \ q_n \in \N \right\}.$$
The Floquet transform of $f$  is defined as:
\begin{equation} \label{floquettransform} \mathcal{U}[f](x,\alpha) = \sum_{n
\in \Lambda} f(x-n) e^{\i \alpha \cdot n}. 
\end{equation}
This transform is an analogue of the Fourier transform for the periodic case. The parameter $\alpha$ is called
the quasi-periodicity, and it is an analogue
of the dual variable in the Fourier transform. If we shift $x$ by
a period $m \in \Lambda$, we get the 
Floquet condition (or quasi-periodic condition)
\begin{equation} \label{floquetcondition}
\mathcal{U}[f](x+m,\alpha) = e^{\i \alpha \cdot m}
\mathcal{U}[f](x,\alpha),\end{equation} which shows that it suffices to know
the function $\mathcal{U}[f](x,\alpha)$ on the unit cell
$Y$ in order to recover it completely as a function of
the $x$-variable. Moreover, $\mathcal{U}[f](x,\alpha)$ is periodic
with respect to  $\alpha$: 
\begin{equation}
\label{floquetperiodic} \mathcal{U}[f](x,\alpha + q) =
 \mathcal{U}[f](x,\alpha), \quad q \in \Lambda^*.
 \end{equation}
Here, $\Lambda^*$ is the dual lattice, generated by the dual lattice vectors $\alpha_1,...,\alpha_d$ defined through the relation
$$\l_i \alpha_j = 2\pi \delta_{ij}, \qquad 0 \leq i,j \leq d.$$
Therefore, $\alpha$ can be considered as an element of the torus
$\R^d/\Lambda^*$. Another way of saying this is that all
information about $\mathcal{U}[f](x,\alpha)$ is contained in its
values for $\alpha$ in the fundamental domain $Y^*$ of the dual
lattice $\Lambda^*$. This domain is referred to as the
 Brillouin zone and is depicted in Figure \ref{bzone} for a square lattice in two dimensions.
 
 \begin{figure}[h]
 \begin{center}
 \includegraphics[height=7cm]{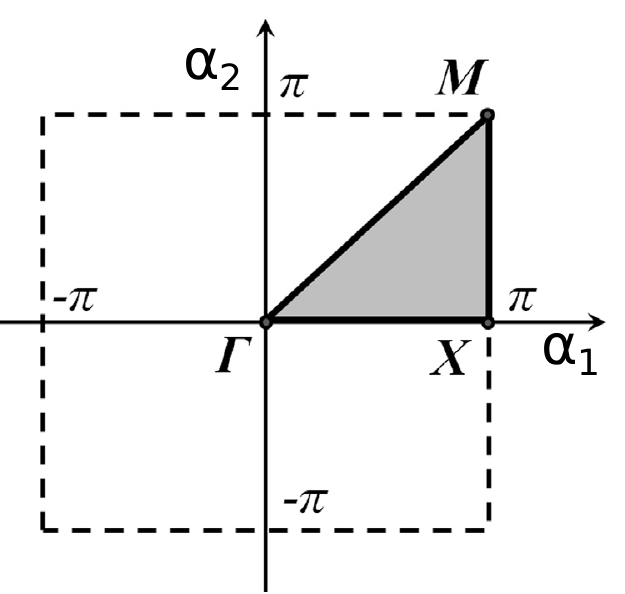}
 \caption{(First) Brillouin {zone} for a square lattice in two dimensions, with the symmetry points $\Gamma, X$ and $M$. The highlighted triangle is known as the reduced Brillouin zone. \label{bzone}}
 \end{center}
 \end{figure}

The following result is an analogue of the  Plancherel theorem when one uses the Fourier transform.
Suppose that the measures $\dx\alpha$ and the Brillouin zone
$Y^*$ are normalized. The following theorem holds \cite{kuchment2}.
\begin{theorem}[Plancherel-type theorem]
The transform $$\mathcal{U} : L^2(\R^d) \rightarrow
L^2(Y^* , L^2(Y))$$ is isometric. Its inverse is
given by
$$\mathcal{U}^{-1}[g](x) = \int_{Y^*} g(x,\alpha)
\dx\alpha,$$ where the function $g(x,\alpha) \in L^2(Y\times Y^*)$ is extended from $Y$ to all $x \in \R^d$
according to the Floquet condition \eqref{floquetcondition}.
\end{theorem}

Consider now a linear partial differential operator $L(x,
\partial_x)$, whose coefficients are periodic with respect to
$\Lambda$.
Due to periodicity, the operator commutes with the Floquet transform
$$ \mathcal{U}[Lf](x,\alpha) = L(x, \partial_x) \mathcal{U}[f](x,\alpha).$$
For each $\alpha$, the operator $L(x,
\partial_x)$ now acts on functions satisfying the corresponding
Floquet condition \eqref{floquetcondition}. Denoting this operator by
$L(\alpha)$, we see that the Floquet transform $\mathcal{U}$
expands the periodic partial differential operator $L$ in
$L^2(\R^d)$ into the direct integral of
operators \begin{equation} \label{directintegral} \int^{\oplus}_{Y^*} L(\alpha) \dx \alpha.
\end{equation} 
If  $L$ is a self-adjoint operator, one can prove the main
spectral result: \begin{equation} \label{spectralstatement} \sigma(L) = \ds
\bigcup_{\alpha \in Y^*} \sigma(L(\alpha)),\end{equation} where $\sigma$
denotes the spectrum.

If $L$ is  elliptic, the operators $L(\alpha)$ have compact
resolvents and hence discrete spectra. If $L$ is bounded from
below, the spectrum of $L(\alpha)$ accumulates only at $+\infty$.
Denote by $\mu_n(\alpha)$ the $n$th eigenvalue of $L(\alpha)$
(counted in increasing order with their multiplicity). The
function $\alpha \mapsto \mu_n(\alpha)$ is continuous in $Y^*$. It
is one branch of the dispersion relations  and is called a \emph{band  function}. We
conclude that the spectrum $\sigma(L)$ consists of the closed
intervals (called the spectral bands)
$$\bigg[\min_\alpha \mu_n(\alpha), \max_\alpha \mu_n(\alpha)\bigg],$$ where
$\min_\alpha \mu_n(\alpha) \rightarrow +\infty$ when $n
\rightarrow +\infty$.

\subsection{Quasi-periodic layer potentials} \label{sectquasiGH}
We introduce a quasi-periodic version of the layer potentials. Again, we let $Y$ and $Y^*$ be the unit cell and dual unit cell, respectively. For $\alpha \in Y^*$, the function $G^{\alpha, k}$ is defined to satisfy
$$ (\Delta_x + k^2) G^{\alpha, k} (x,y) = \sum_{m\in \Lambda} \delta(x-y-n) e^{\i m\cdot \alpha},$$
where $\delta$ is the Dirac delta function and  $G^{\alpha, k} $ is $\alpha$-quasi-periodic, \emph{i.e.}, $e^{- \i \alpha\cdot x} G^{\alpha, k}(x,y)$ is periodic in $x$ with respect to $Y$.  It is known that $G^{\alpha, k} $ can be written as
$$ G^{\alpha, k}(x,y) = \sum_{q\in \Lambda^*} \frac{e^{\i (\alpha+q)\cdot (x-y)}}{k^2- |\alpha + q|^2},$$
if $k \ne |\alpha + q|$ for any $q \in Y^*$. We remark that 
\begin{align} 
G^{\alpha,k}(x, y)=   G^{\alpha,0}(x,y) -  G_l^{\alpha,\#}(x -y) := G^{\alpha,0}(x,y) - \sum_{l=1}^\infty k^{2 l}\sum_{q\in \Lambda ^*} \frac{e^{\i (\alpha+q)\cdot (x-y)}}{|\alpha+q|^{2(l+1)}} \label{eq:defGk2}
\end{align}
when $\alpha \neq 0$, and $k \rightarrow 0$.

We let $D$ be as in \Cref{sec-2} and additionally assume $D\Subset Y$. Then the quasi-periodic single layer potential $\mathcal{S}_D^{\alpha,k}$ is defined by
\begin{equation} \label{singlealpha} \mathcal{S}_D^{\alpha,k}[\phi](x) = \int_{\partial D} G^{\alpha,k} (x,y) \phi(y) \dx \sigma(y),\quad x\in \mathbb{R}^3. \end{equation}
It satisfies the following jump formulas:
\begin{equation*}
\S_D^{\alpha,k}[\phi]\big|_+ = \S_D^{\alpha,k}[\phi]\big|_-,
\end{equation*}
and
$$ \frac{\p}{\p\nu} \mathcal{S}_D^{\alpha,k}[\phi] \Big|_{\pm}  = \left( \pm \frac{1}{2} I +( \mathcal{K}_D^{-\alpha,k} )^*\right)[\phi]\quad \mbox{on}~ \p D,$$
where $(\mathcal{K}_D^{-\alpha,k})^*$ is the operator given by
$$ (\mathcal{K}_D^{-\alpha, k} )^*[\phi](x)= \int_{\p D} \frac{\p}{\p\nu_x} G^{\alpha,k}(x,y) \phi(y) \dx \sigma(y).$$
We remark that  $\mathcal{S}_D^{\alpha,0} : L^2(\p D) \rightarrow H^1(\p D)$ is invertible for $\alpha \ne 0$ \cite{ammari2018mathematical}. Moreover, the following decomposition holds for the layer potential $\mathcal{S}_D^{\alpha,k}$: 
\begin{equation} \label{series-s2}
\mathcal{S}_{D}^{\alpha,k} =  \mathcal{S}_D^{\alpha, 0} + k^{2}\mathcal{S}_{D,1}^{\alpha} + \O(k^4) \quad \text{with} \quad \mathcal{S}_{D,1}^\alpha[\psi] := \int_{\p D} G_1^{\alpha,\#}(x -y) \psi(y) \dx \sigma(y),
\end{equation}
where the error term is with respect to the operator norm $\|.\|_{\mathcal{L}(L^2(\partial D), H^1(\partial D))}$. Furthermore, analogously to (\ref{eq:exp_K}), we have
\begin{equation} \label{series-k2}
(\mathcal{K}_{D}^{-\alpha,k})^* =  (\mathcal{K}_{D}^{-\alpha,k})^*  + k^2 \mathcal{K}_{D,1}^\alpha + \O(k^3),
\end{equation}
where the error term is with respect to the operator norm $\|.\|_{\mathcal{L}(L^2(\partial D), L^2(\partial D))}$.

Finally, we introduce the $\alpha$-quasi capacity of $D$, denoted by $\mathrm{Cap}_{D,\alpha}$, 
$$
\mathrm{Cap}_{D,\alpha}: = \int_{Y\setminus \overline{D}} |\nabla u|^2\; \dx y, 
$$
where $u$ is the $\alpha$-quasi-periodic harmonic function in $Y\setminus \overline{D}$ with $u=1$ on $\p D$. For $\alpha\neq 0$,  we have $u(x) =\mathcal{S}_D^{\alpha,0} \left(\mathcal{S}_D^{\alpha,0}\right)^{-1}[\chi_{\partial D}](x)$ for $ x \in Y \setminus \overline{D}$ and
\begin{equation} \label{alphacap}
\mathrm{Cap}_{D,\alpha} :=-\int_{\p D}  \left( \mathcal{S}_D^{\alpha,0}\right)^{-1}[\chi_{\partial D}](y)\; \dx \sigma(y). 
\end{equation}
Moreover, we have a variational definition of $\mathrm{Cap}_{D,\alpha}$. 
Indeed, let $\mathcal{C}_{\alpha}^{\infty}(Y)$ be the set of $\mathcal{C}^{\infty}$ functions in $Y$ which can be extended to $\mathcal{C}^{\infty}$ $\alpha$-quasi-periodic functions in $\R^3$. Let $\mathcal{H}_{\alpha}$ be the closure of the set $\mathcal{C}_{\alpha}^{\infty}(Y)$ in $H^1(Y)$, and let
$\mathcal{V}_\alpha:= \{  v\in \mathcal{H}_{\alpha} : v=1~\mbox{on }\p D\}$. Then we can show that
\begin{equation} \label{vcalpha}
\mathrm{Cap}_{D,\alpha}= \min_{v\in\mathcal{V}_\alpha}\int_{Y\setminus \overline{D}} |\nabla v|^2 \dx y.
\end{equation}

\subsection{Square lattice subwavelength resonator crystal}
We first describe the  crystal  under consideration. 
Assume that the resonators occupy  $\cup_{n\in \mathbb{Z}^d} (D+n)$ for a bounded and simply connected domain $D \Subset Y$ with $\p D \in \mathcal{C}^{1, \eta}$ with $0<\eta <1$. See Figure \ref{figy}.
As before, we denote by $\rho_b$ and $\kappa_b$ the material parameters inside the resonators and by $\rho$ and $\kappa$ the corresponding parameters for the background media and let $v, v_b, k,$ and  $k_b$ be defined by (\ref{defkv}). We also let the  dimensionless contrast parameter $\delta$  be defined by (\ref{defdelta}) and assume for simplicity that $v_b/v=1$. 

\begin{figure}[h]
	\centering
	\begin{tikzpicture}[scale=1.2]
	\begin{scope}[xshift=-4cm,scale=1]
	\coordinate (a) at (1,0);		
	\coordinate (b) at (0,1);	
	
	\draw (-0.5,-0.5) -- (0.5,-0.5) -- (0.5,0.5) -- (-0.5,0.5) -- cycle; 
	\draw (0,0) circle(6pt);
	\draw[opacity=0.2] (-0.5,0) -- (0,0)
	(0.5,0) -- (0,0)
	(0,-0.5) -- (0,0)
	(0,0.5) -- (0,0);
	
	\begin{scope}[shift = (a)]
	\draw (0,0) circle(6pt);
	\draw[opacity=0.2] (-0.5,0) -- (0,0)
	(0.5,0) -- (0,0)
	(0,-0.5) -- (0,0)
	(0,0.5) -- (0,0);
	\end{scope}
	\begin{scope}[shift = (b)]
	\draw (0,0) circle(6pt);
	\draw[opacity=0.2] (-0.5,0) -- (0,0)
	(0.5,0) -- (0,0)
	(0,-0.5) -- (0,0)
	(0,0.5) -- (0,0);
	\end{scope}
	\begin{scope}[shift = ($-1*(a)$)]
	\draw (0,0) circle(6pt);
	\draw[opacity=0.2] (-0.5,0) -- (0,0)
	(0.5,0) -- (0,0)
	(0,-0.5) -- (0,0)
	(0,0.5) -- (0,0);
	\end{scope}
	\begin{scope}[shift = ($-1*(b)$)]
	\draw (0,0) circle(6pt);
	\draw[opacity=0.2] (-0.5,0) -- (0,0)
	(0.5,0) -- (0,0)
	(0,-0.5) -- (0,0)
	(0,0.5) -- (0,0);
	\end{scope}
	\begin{scope}[shift = ($(a)+(b)$)]
	\draw (0,0) circle(6pt);
	\draw[opacity=0.2] (-0.5,0) -- (0,0)
	(0.5,0) -- (0,0)
	(0,-0.5) -- (0,0)
	(0,0.5) -- (0,0);
	\end{scope}
	\begin{scope}[shift = ($-1*(a)-(b)$)]
	\draw (0,0) circle(6pt);
	\draw[opacity=0.2] (-0.5,0) -- (0,0)
	(0.5,0) -- (0,0)
	(0,-0.5) -- (0,0)
	(0,0.5) -- (0,0);
	\end{scope}
	\begin{scope}[shift = ($(a)-(b)$)]
	\draw (0,0) circle(6pt);
	\draw[opacity=0.2] (-0.5,0) -- (0,0)
	(0.5,0) -- (0,0)
	(0,-0.5) -- (0,0)
	(0,0.5) -- (0,0);
	\end{scope}
	\begin{scope}[shift = ($-1*(a)+(b)$)]
	\draw (0,0) circle(6pt);
	\draw[opacity=0.2] (-0.5,0) -- (0,0)
	(0.5,0) -- (0,0)
	(0,-0.5) -- (0,0)
	(0,0.5) -- (0,0);
	\end{scope}
	\end{scope}
	
	\draw[dashed,opacity=0.5,->] (-3.9,0.65) .. controls(-2.9,1.8) .. (0.5,0.7);
	\begin{scope}[xshift=2cm,scale=2.8]	
	\coordinate (a) at (1,{1/sqrt(3)});		
	\coordinate (b) at (1,{-1/sqrt(3)});	
	\coordinate (Y) at (1.8,0.45);
	\coordinate (c) at (2,0);
	\coordinate (x1) at ({2/3},0);
	\coordinate (x0) at (1,0);
	\coordinate (x2) at ({4/3},0);

	\pgfmathsetmacro{\rb}{0.25pt}
	\pgfmathsetmacro{\rs}{0.2pt}\
	
	\draw[->] (-0.5,-0.5) -- (-0.5,0.5) node[left]{}; 
	\draw[->] (-0.5,-0.5) -- (0.5,-0.5) node[below]{}; 
	\draw (0.5,-0.5) -- (0.5,0.5) -- (-0.5,0.5);
	\draw (0,0) circle(6pt);
	\draw (0.3,0) node{$D$};
	
	\draw (0.5,0.5) node[right]{$Y$};
	\end{scope}
	\end{tikzpicture}
	\caption{Illustration of the square lattice crystal and quantities in $Y$.}\label{figy}
\end{figure}
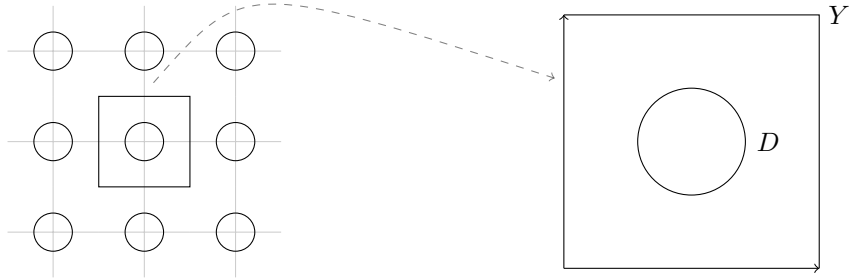

To investigate the phononic gap of the  crystal we consider the following $\alpha-$quasi-periodic equation in the unit cell $Y=[-1/2,1/2)^3$:
\begin{equation} \label{eq-scatteringB}
\left\{
\begin{array} {ll}
&\ds \nabla \cdot \frac{1}{\rho} \nabla  u+ \frac{\omega^2}{\kappa} u  = 0 \quad \text{in} \quad Y \backslash \overline{D}, \\
\nm
&\ds \nabla \cdot \frac{1}{\rho_b} \nabla  u+ \frac{\omega^2}{\kappa_b} u  = 0 \quad \text{in} \quad D, \\
\nm
&\ds  u |_{+} -u |_{-}  =0   \quad \text{on} \quad \partial D, \\
\nm
& \ds  \frac{1}{\rho} \frac{\p u}{\p \nu} \bigg|_{+} - \frac{1}{\rho_b} \frac{\p u}{\p \nu} \bigg|_{-} =0 \quad \text{on} \quad \partial D,\\
\nm
&  e^{-\i \alpha \cdot x} u  \,\,\,  \mbox{is periodic.}
  \end{array}
 \right.
\end{equation}

By choosing proper physical units, we may assume that the resonator size is of order one.  We assume  that the wave speeds outside and inside the resonators are comparable to each other and that condition (\ref{defp}) holds.

\subsection{Subwavelength band gaps and Bloch modes} \label{sec-2}
As described in \Cref{sec:floquet}, the problem \eqref{eq-scatteringB} has nontrivial solutions for discrete values of $\omega$ such as 
$$  0 \le \omega_1^\alpha \le \omega_2^\alpha \le \cdots$$
and we have the following band structure of propagating frequencies for the given periodic structure:
$$ [0, \max_\alpha \omega_1^\alpha] \cup [ \min_\alpha \omega_2^\alpha, \max_\alpha \omega_2^\alpha] \cup  [ \min_\alpha \omega_3^\alpha, \max_\alpha \omega_3^\alpha] \cup \cdots. $$

A non-trivial solution to this problem and its corresponding frequency is called a Bloch eigenfunction and a Bloch eigenfrequency. The Bloch eigenfrequencies $\omega_i^\alpha, \ i=1,2,\ldots$ with positive real part, seen as functions of $\alpha$, are the {band functions}.

We use the quasi-periodic  single-layer potential introduced in (\ref{singlealpha}) to represent the solution to the scattering problem (\ref{eq-scatteringB}) in $Y\setminus \overline{D}$.
We look for a solution $u$ of~\eqref{eq-scatteringB} of the form:
\begin{equation} \label{Helm-solutionB}
u =
\begin{cases}
\mathcal{S}_{D}^{\alpha,k} [\psi]\quad & \text{in} ~ Y \setminus \overline{D},\\
\nm
 \mathcal{S}_{D}^{k_b} [\psi_b]   &\text{in} ~   {D},
\end{cases}
\end{equation}
for some surface potentials $\psi, \psi_b \in  L^2(\p D)$. 
Using the jump relations for the single-layer potentials, one can show that~\eqref{eq-scatteringB} is equivalent to the boundary integral equation
\begin{equation}  \label{eq-boundaryB}
\mathcal{A}(\omega, \delta)[\Psi] =0,  
\end{equation}
where
\[
\mathcal{A}(\omega, \delta) = 
 \begin{pmatrix}
  \mathcal{S}_D^{k_b} &  -\mathcal{S}_D^{\alpha,k}  \\
  \nm
  -\frac{1}{2} I + \mathcal{K}_D^{k_b, *}& -\delta( \frac{1}{2} I + (\mathcal{K}_D^{ -\alpha,k})^*)
\end{pmatrix}, 
\,\, \Psi= 
\begin{pmatrix}
\psi_b\\
\psi
\end{pmatrix}.
\]

As before, we denote by $$\mathcal{H} = L^2(\p D) \times L^2(\p D) \quad \mbox{ and } \quad \mathcal{H}_1 = H^1(\p D) \times L^2(\p D).$$  It is clear that $\mathcal{A}(\omega, \delta)$ is a bounded linear operator from $\mathcal{H}$ to $\mathcal{H}_1$, \emph{i.e.}
$\mathcal{A}(\omega, \delta) \in \mathcal{L}(\mathcal{H}, \mathcal{H}_1)$. We first look at the limiting case when $\delta =0$. The operator $\mathcal{A}(\omega, \delta)$ is a perturbation of
\begin{equation}  \label{eq-A_0-3dB}
 \mathcal{A}(\omega, 0) = 
 \begin{pmatrix}
  \mathcal{S}_D^{k_b} &  -\mathcal{S}_D^{\alpha,k}  \\
  \nm
  -\frac{1}{2} I + \mathcal{K}_D^{k_b, *} & 0
\end{pmatrix}.
\end{equation}
We see that $\omega_0$ is a characteristic value of $\mathcal{A}(\omega,0)$ if and only if $(\omega_0 v^{-1}_b)^2$ is a Neumann eigenvalue of $D$ or $(\omega_0 v^{-1})^2$ is a Dirichlet eigenvalue of $Y\backslash \overline{D}$ with $\alpha$-quasi-periodicity on $\partial Y$.  Since 
zero is a Neumann eigenvalue of $D$, 
$0$ is a characteristic value for the holomorphic operator-valued  function $\mathcal{A}(\omega,0)$. {By noting that there is a positive lower bound for the other Neumann eigenvalues of $D$ and all the Dirichlet eigenvalues of $Y\backslash \overline{D}$ with $\alpha$-quasi-periodicity on $\partial Y$,}  we can conclude the following result by the Gohberg-Sigal theory.
\begin{lemma}
For any $\delta$ sufficiently small, there exists  {one and only one} characteristic value 
$\omega_0= \omega_0(\delta)$ in a neighborhood of the origin in the complex plane to the holomorphic operator-valued  function 
$\mathcal{A}(\omega, \delta)$.
Moreover,  
$\omega_0(0)=0$ and $\omega_0$ depends on $\delta$ continuously.
\end{lemma}

\subsubsection{Asymptotic behavior of the first Bloch eigenfrequency $\omega_1^\alpha$}
In this section we assume $\alpha \ne 0$. 
We define
\begin{equation} \label{defA0}
\mathcal{A}_0 :=\mathcal{A}(0,0)= 
 \begin{pmatrix}
  \mathcal{S}_D&  -\mathcal{S}_D^{\alpha,0}  \\
  \nm
  -\frac{1}{2} I + \mathcal{K}_D^{*}& 0
\end{pmatrix},
\end{equation} 
and let  $\mathcal{A}_0^* : \cH_1 \to \cH$ be the adjoint of $\mathcal{A}_0$. We choose an element $\psi_0\in L^2(\p D)$ such that
$$ \big( -\frac{1}{2} I + \mathcal{K}_D^*  \big)[\psi_0] =0,\quad  \int_{\partial D} \psi_0 = 1.$$
We recall the definition (\ref{defcap}) of   the capacity of the set $D$,  $\mathrm{Cap}_D$, which is equivalent to
\begin{equation}\label{capacityB} 
\mathcal{S}_D [\psi_0] = - \frac{1}{\mathrm{Cap}_D} \quad \mbox{on } {\p D}.
\end{equation}

Then we can easily check that $\mathrm{Ker } (\mathcal{A}_0)$ and $ \mathrm{Ker } (\mathcal{A}_0^*)$ are spanned respectively by
\[
\Psi_0 = \begin{pmatrix}
    \psi_0\\
  \widetilde\psi_0\end{pmatrix}
  \quad \text{and} \quad
  \Phi_0 = \begin{pmatrix}
    0\\
  1 \end{pmatrix},
\]
where  $\widetilde \psi_0 =( \mathcal{S}_D^{\alpha,0})^{-1} \mathcal{S}_D[\psi_0]$.
We now perturb  $ \mathcal{A}_0 $ by a rank-one operator $\mathcal{P}_0$ from $\mathcal{H}$ to $\mathcal{H}_1$
given by 
$
\mathcal{P}_0[\Psi]:= (\Psi, \Psi_0)\Phi_0,
$
and denote it by
$
\widetilde{\mathcal{A}_0}= \mathcal{A}_0 + \mathcal{P}_0
$.
Then the followings hold:
\begin{enumerate}
\item[(i)] $\widetilde{\mathcal{A}_0}[\Psi_0]= \| \Psi_0 \|^2 \Phi_0 $, $\widetilde{\mathcal{A}_0}^*[\Phi_0] = \| \Phi_0 \|^2\Psi_0$. 

\item[(ii)] The operator $\widetilde{\mathcal{A}_0}$ and its adjoint $\widetilde{\mathcal{A}_0}^*$ are invertible in
$\mathcal{L}(\mathcal{H}, \mathcal{H}_1)$ and  $\mathcal{L}(\mathcal{H}_1, \mathcal{H})$, respectively. 

\end{enumerate}

Using (\ref{eq:exp_S}), (\ref{eq:exp_K}), (\ref{series-s2}), and (\ref{series-k2}), we can expand $\mathcal{A}(\omega,\delta)$ as
\begin{equation} \label{expdA}
\begin{array}{l}
\mathcal{A}(\omega, \delta):=\mathcal{A}_0 + \mathcal{B}(\omega, \delta)
= \mathcal{A}_0 + \omega \mathcal{A}_{1, 0}+ \omega^2 \mathcal{A}_{2, 0}
+ \omega^3 \mathcal{A}_{3, 0} + \delta \mathcal{A}_{0, 1}+ \delta \omega^2\mathcal{A}_{2, 1}\\
\nm \qquad \qquad \ds  + \O(| \omega| ^4 + |\delta \omega^3|) \end{array}
\end{equation}
where 
\[
\mathcal{A}_{1,0} = \begin{pmatrix}
  v_b^{-1} \mathcal{S}_{D,1} & 0 \\
  \nm
  0& 0
\end{pmatrix},
\,\, \mathcal{A}_{2,0}= 
\begin{pmatrix}
  v_b^{-2} \mathcal{S}_{D,2} &  -v^{-2} \mathcal{S}_{D,1}^\alpha  \\
  \nm
   v_b^{-2} \mathcal{K}_{D,2} & 0
\end{pmatrix},
\,\, \mathcal{A}_{3,0}= 
\begin{pmatrix}
  v_b^{-3} \mathcal{S}_{D,3} & 0  \\
  \nm
 v_b^{-3}  \mathcal{K}_{D,3} & 0
\end{pmatrix},
\]
\[
\mathcal{A}_{0, 1}=
\begin{pmatrix}
0& 0\\
\nm
0 &  -(\frac{1}{2}+ (\mathcal{K}_{D}^{-\alpha,0})^*)
\end{pmatrix},
\,\,  \mathcal{A}_{2, 1}=
\begin{pmatrix}
0& 0\\
\nm
0 &  -v^{-2} \mathcal{K}^{\alpha}_{D,1}
\end{pmatrix}.
\]
From the above expansion, it follows that
\begin{equation} \label{eq:Aomegadelta}
\begin{array}{lll}
A(\omega, \delta) &= & \ds - \omega^2 \frac{ v_b^{-2} | D |}{\mathrm{Cap}_D} - \omega^3 v_b^{-3} \frac{\i c_1 | D |}{4 \pi } + c_2\delta + \omega \delta \frac{\i c_1c_2  v_b^{-1} \mathrm{Cap}_D}{4 \pi}  \\
\nm &&\ds  + \O( | \omega |^4 + | \delta | \, |\omega|^2  + | \delta |^2),
\end{array} 
\end{equation}
where \begin{equation}
c_1:=\frac{\|\psi_0\|^2}{\|\psi_0\|^2+\|\widetilde\psi_0\|^2},\end{equation}
and \begin{equation}  \label{defc2B}
c_2:= \ds \int_{\partial D} \widetilde\psi_0 \; \big(1/2 + \mathcal{K}_D^{-\alpha,0}\big)[\chi_{\partial D}]\; \dx \sigma.\end{equation}

We now solve $A(\omega, \delta) =0$. 
It is clear that $\delta = \O(\omega^2)$ and thus $\omega_0(\delta) = \O(\sqrt{\delta})$. 
We write 
$$
\omega_0(\delta) = a_1 \delta^{\frac{1}{2}} + a_2 \delta + \O(\delta^{\frac{3}{2}}),
$$
and get
\begin{align*}
	&  -  \frac{ v_b^{-2} | D |}{ \mathrm{Cap}_D} \left( a_1 \delta^{\frac{1}{2}} + a_2 \delta + \O(\delta^{\frac{3}{2}}) \right)^2 
	 - \frac{\i c_1  v_b^{-3} | D |}{4 \pi}\left( a_1 \delta^{\frac{1}{2}} + a_2 \delta + \O(\delta^{\frac{3}{2}}) \right)^3 \\
	 & \qquad
	 +c_2 \delta + \frac{\i c_1c_2 v_b^{-1} \mathrm{Cap}_D}{4 \pi }  \left( a_1 \delta^{\frac{3}{2}} + a_2 \delta^2 + \O(\delta^{\frac{5}{2}}) \right) + \O(\delta^2) = 0.
\end{align*}
From the coefficients of the $\delta$ and $\delta^{\frac{3}{2}}$ terms, we obtain 
\[
	- a_1^2  \frac{v_b^{-2} | D |}{ \mathrm{Cap}_D} + c_2 = 0
	 \]
	 and
	 \[
	 2 a_1 a_2 \frac{- v_b^{-2} | D |}{ \mathrm{Cap}_D} - a_1^3 \frac{\i c_1 v_b^{-3} | D |}{4\pi }  + a_1 \frac{\i c_1c_2 v_b^{-1} \mathrm{Cap}_D}{4 \pi } = 0, 
\]
which yields 
\[
	a_1 =  \pm \sqrt{ \frac{ c_2 \mathrm{Cap}_D}{| D |} } v_b
	 \quad \text{and} \quad 
	 a_2 = 0.
\]
From the definition (\ref{alphacap}) of the $\alpha$-quasi-periodic capacity, it follows that
\begin{align*} c_2 &= \frac{\mathrm{Cap}_{D,\alpha}}{\mathrm{Cap}_D}.\end{align*}
Therefore, the following result from \cite{bandgap} holds. 
\begin{theorem}\label{approx_thm} For $\alpha \ne 0$ and sufficiently small $\delta$, we have
\begin{align}
\omega_1^\alpha= \omega_M \sqrt{\frac{\mathrm{Cap}_{D,\alpha}}{\mathrm{Cap}_D}} + \O(\delta^{3/2}), \label{o_1_alpha}
\end{align}
where $\omega_M$ is defined in (\ref{defomegaM}) by $$\ds \omega_M= \sqrt{ \frac{\delta  \mathrm{Cap}_D}{|D|} } v_b.$$
\end{theorem}

Now from \eqref{o_1_alpha}, we can see that $$\omega_{M,\alpha}:= \omega_M \sqrt{\frac{\mathrm{Cap}_{D,\alpha}}{\mathrm{Cap}_D}} \rightarrow 0$$ as $\alpha\to 0$   because 
 $$ \big( (1/2) I +( \mathcal{K}_D^{-\alpha,0})^*\big)(\mathcal{S}_D^{\alpha,0})^{-1} [\chi_{\partial D}] \rightarrow 0,$$  and  so $\mathrm{Cap}_{D,\alpha}  \rightarrow 0$ as $\alpha\rightarrow0$. Moreover,  it is clear that $\omega_{M,\alpha}$ lies in a small neighborhood of zero. 

We define $\omega^1_*:= \max_{\alpha} \omega_{M,\alpha}$. Then we deduce the following
result regarding a subwavelength band gap opening.
\begin{theorem}\label{main}
For every $\epsilon>0$, there exists $\delta_0>0$  and {$\widetilde \omega>  \omega^1_*+\epsilon$} such that 
\begin{equation}
 [ \omega^1_*+\epsilon, \widetilde\omega ] \subset [\max_\alpha \omega_1^\alpha, \min_\alpha \omega_2^\alpha]
 \end{equation} 
 for $\delta<\delta_0$.
\end{theorem}
\begin{proof}
 Using  $\omega_1^0=0$ and the continuity of $\omega_1^\alpha$ in $\alpha$ and $\delta$, we get $\alpha_0$ and $\delta_1$ such that $\omega_1^\alpha < \omega^1_*$ for every $| \alpha|<\alpha_0$ and $\delta<\delta_1$. Following the derivation of \eqref{o_1_alpha},  we can check that it is valid uniformly in $\alpha$ as far as $|\alpha| \ge \alpha_0$. Thus there exists $\delta_0 < \delta_1$ such $\omega_1^\alpha \le \omega^1_* +\epsilon$ for $|\alpha| \ge \alpha_0$.  We have shown that $ \max_\alpha \omega_1^\alpha \le \omega^1_*+\epsilon$ for sufficiently small $\delta$. To have  $\min_\alpha \omega_2^\alpha > \omega^1_* +\epsilon$ for small $\delta$, it is enough to check that $\mathcal{A}(\omega,\delta)$ has no small characteristic value other than $\omega_1^\alpha$. For $\alpha$ away from $0$, we can see that it is true following the proof of Theorem \ref{approx_thm}. If $\alpha=0$, we have
\begin{equation}
\mathcal{A}(\omega,\delta)=\mathcal{A}(\omega,0) + \O(\delta),\end{equation}
near $\omega_2^0$ with $\delta=0$. Since  $\omega_2^0\ne 0$, we have $\omega_2^0(\delta) > \omega^1_* + \epsilon$ for sufficiently small $\delta$. Finally, using the continuity of $\omega_2^\alpha$ in $\alpha$, we obtain  $\min_\alpha \omega_2^\alpha > \omega^1_* +\epsilon$ for small $\delta$.  This completes the proof.
\end{proof}

As shown in \cite{highfrequency}, the first Bloch eigenvalue $\omega_1^\alpha$ attains its maximum $\omega^1_*$ at $\alpha_*=(\pi,\pi,\pi)$ (\textit{i.e.} the corner $M$ of the Brillouin zone). The proof relies on the variational characterization (\ref{vcalpha}) of the quasi-periodic capacity. 

\begin{theorem} \label{main2} Assume that $D$ is symmetric with respect to $\{ x_j=0\}$ for $j=1,2,3$. Then both $\mathrm{Cap}_{D,\alpha}$ and $\omega_1^\alpha$ attain their maxima at $\alpha_*=(\pi,\pi,\pi)$. 
\end{theorem}

The results of Theorems \ref{main} and \ref{main2} are illustrated in Figure \ref{bloch}.
\begin{figure}[h]
\centering
  \includegraphics[height=8cm]{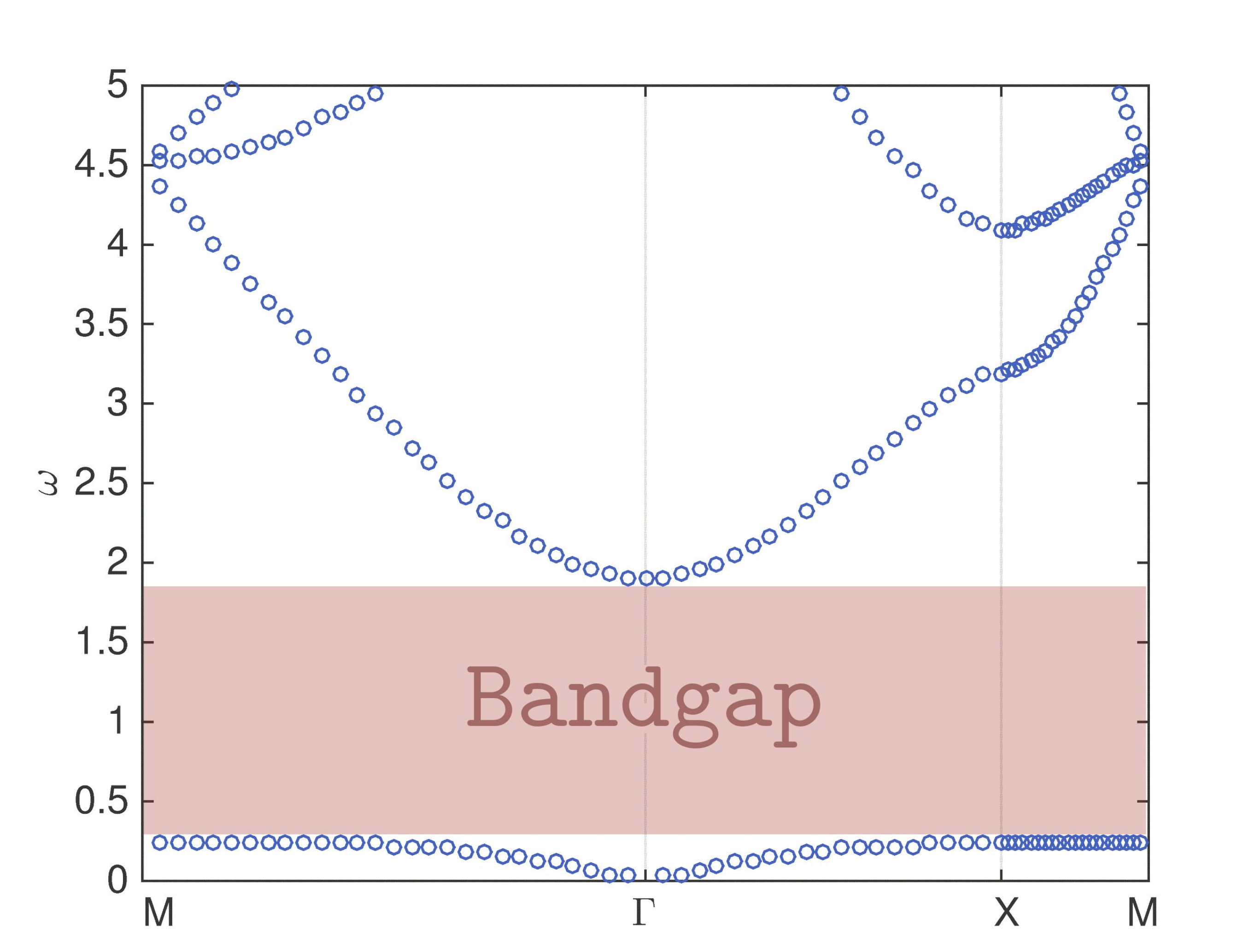}
\caption{Subwavelength band gap opening.} \label{bloch}
\end{figure}

Next, we consider the behavior of the first Bloch eigenfunction. 
In \cite{highfrequency}  a high-frequency  homogenization approach for subwavelength resonators has been developed. An asymptotic expansion of the Bloch eigenfunction near the critical frequency  has been  computed. It is proved that the eigenfunction can be decomposed into two parts: one is slowly varying and satisfies a homogenized equation, while the other is periodic and varying at the microscopic scale. The microscopic oscillations explain why these structures can be used to achieve super-focusing, while the exponential decay of the slowly varying part proves the band gap opening above the critical frequency. 

We need the following lemma from \cite{highfrequency}. 

\begin{lemma} For $\epsilon>0$ small enough, 
\begin{align*}
\mathrm{Cap}_{D, \alpha_*+\epsilon\widetilde\alpha}= \mathrm{Cap}_{D, \alpha_*} +  \epsilon^2 {\Lambda_D^{\widetilde\alpha}} + \O(\epsilon^4), 
\end{align*}
where $\Lambda_D^{\widetilde\alpha}$ is a {negative semi-definite quadratic function} of $\widetilde \alpha$:
$$
\frac{v_b^2}{|D|} \Lambda_D^{\widetilde\alpha} = - \sum_{1\leq i,j\leq 3} \lambda_{ij} \widetilde\alpha_i \widetilde\alpha_j
$$
with $(\lambda_{ij})$ being symmetric and positive semi-definite. 
\end{lemma}

Assume that the resonators are arranged with period  $r>0$ and {$\delta = \O(r^2)$}.
Then, by a scaling argument, the critical frequency $\omega_*^r = (1/r) \omega_*^1 = \O(1)$ as {$r\rightarrow 0$}.

\begin{theorem} \label{superf}
For $\omega$ near the critical frequency $\omega_*^r$: 
$(\omega_*^r)^2  -\omega^2  = \O(r^2)$, 
the following asymptotic of the first {Bloch eigenfunction} $u_{1,r}^{\alpha_*/r + \widetilde\alpha}$ holds:
$$
u_{1,r}^{\alpha_*/r + \widetilde\alpha}(x) = \underbrace{e^{\i \widetilde\alpha \cdot x}}_{{\mbox{macroscopic behavior}}} \underbrace{S \left(\frac{x}{r}\right)}_{{\mbox{microscpic behavior}}} + \; \O(r).
$$
The macroscopic {plane wave} $e^{\i \widetilde\alpha \cdot x}$ satisfies:
$$
{ \sum_{1\leq i, j \leq 3}\lambda_{ij} \partial_i \partial_j \widetilde{ u} (x)+ \frac{(\omega_*^r)^2 -\omega^2}{\delta}\widetilde{ u}(x)= 0}.
$$
\end{theorem}

If we write
 $(\omega_*^r)^2 - \omega^2   = \beta \delta$, then 
$$\sum_{1\leq i, j \leq 3}\lambda_{ij} \widetilde \alpha_i \widetilde \alpha_j = \beta + \O(r^2).$$ Moreover, for $\beta >0$, the {plane wave Bloch eigenfunction} satisfies the homogenized equation for the  crystal while the microscopic field is periodic and varies on the scale of $r$. If $\beta <0$, then the Bloch eigenfunction is {exponentially growing or decaying} which is another way to see that a {band gap opening} occurs above the critical frequency. 

Theorem \ref{superf} shows that the super-focusing property  at subwavelength scales near the critical frequency $\omega_*^r$ holds true. Here, the mechanism is not due to effective (high-contrast below $\omega_*^r$ and negative above $\omega_*^r$) properties of the medium. The effective medium theory described in Section~\ref{dilutesect} is no longer valid in the nondilute case. 

Figure \ref{figbloch}  shows a one-dimensional plot along the $x_1$-axis of the real part of the {Bloch} {eigenfunction} of the {square lattice} over many unit cells. 

\begin{figure}[h] \begin{center}  \includegraphics[scale=0.7]{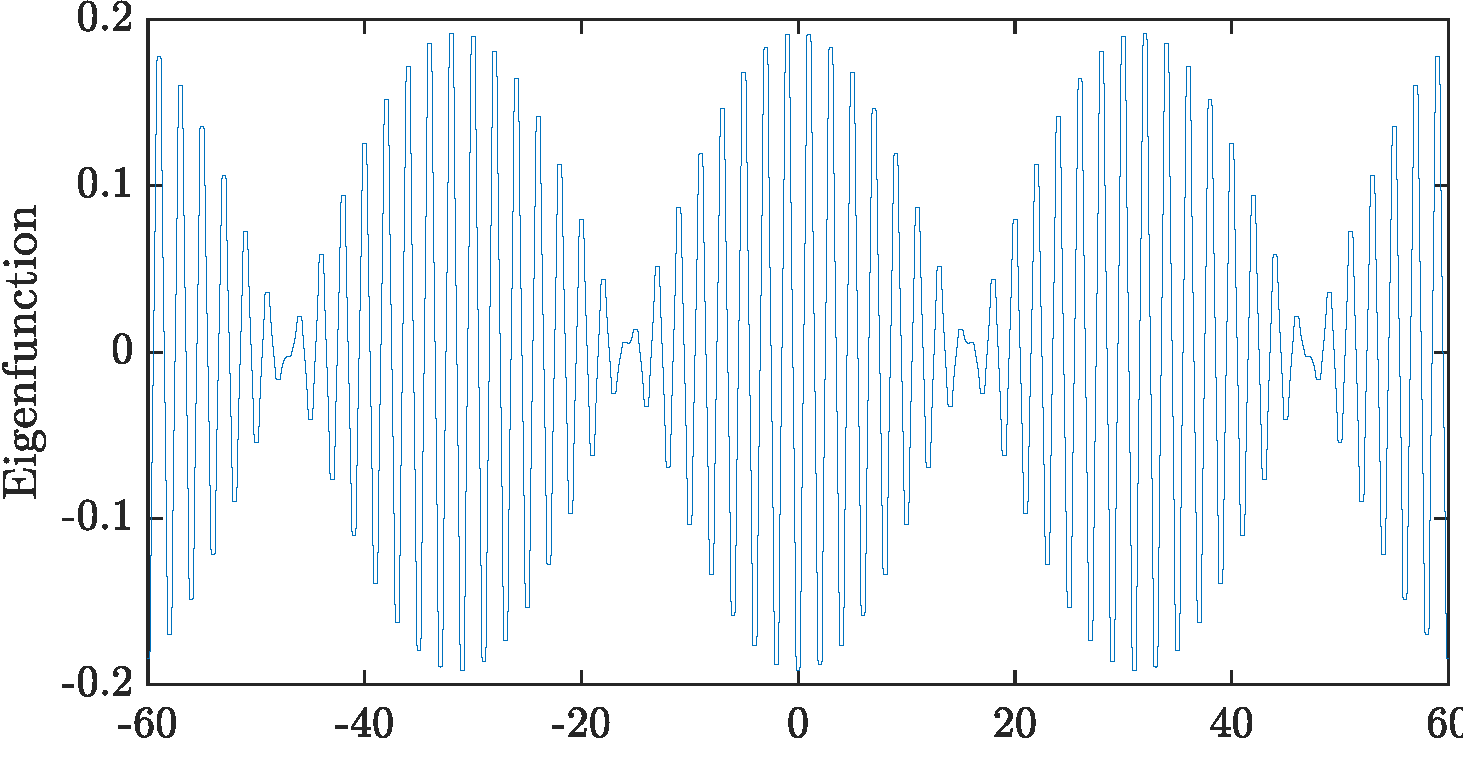} 
\caption{Real part of the {Bloch} {eigenfunction} of the {square lattice} shown over many unit cells. }  \label{figbloch} \end{center}
\end{figure}

\section{Topological metamaterials} \label{sec6}
We begin this section by studying existence and consequences of a Dirac cone singularity in a honeycomb structure. Dirac singularities are intimately connected with topologically protected edge modes, and we then study such modes in an array of subwavelength resonators.

\subsection{Dirac singularity}
The classical example of a structure with a Dirac singularity is graphene, where this singularity is responsible for many peculiar electronic properties. Graphene consists of a single layer of carbon atoms in an honeycomb lattice, and in this section we study a similar structure with subwavelength resonators.

In the homogenization theory of metamaterials, the goal is to map the metamaterial to a homogeneous material with some effective parameters. It has  been  demonstrated in the previous section  that this approach does not apply in the case of  crystals at ``high'' frequencies, \textit{i.e.}, away from the centre $\Gamma$ (corresponding to $\alpha=(0,0,0)$) of the Brillouin zone.  In Theorem \ref{superf}, it is shown that around the symmetry point $M$ (corresponding to $\alpha=(\pi,\pi,\pi)$) in the Brillouin zone of a crystal with a square lattice, the Bloch eigenmodes display oscillatory behaviour on two distinct scales: small scale oscillations on the order of the size of individual
resonators, while simultaneously the plane-wave envelope oscillates at a much larger scale and satisfies a homogenized equation.

In this section we prove the near-zero effective index property in  honeycomb crystal at the deep subwavelength scale.  We develop a homogenization theory that captures both the macroscopic behaviour of the eigenmodes and the oscillations in the microscopic scale. The near-zero effective refractive index at the macroscale is a consequence of the existence of a Dirac dispersion cone.

We consider a two-dimensional infinite honeycomb crystal in two dimensions depicted in Figure \ref{fig:honeycomb}. Define the lattice $\Lambda$ generated by the lattice vectors
$$ l_1 = L\left( \frac{\sqrt{3}}{2}, \frac{1}{2} \right),~~l_2 = L\left( \frac{\sqrt{3}}{2}, -\frac{1}{2}\right),$$
where $L$ is the lattice constant. Denote by $Y$ a fundamental domain of the given lattice. Here, we take 
$$ Y:= \left\{ s l_1+ t l_2 ~|~ 0 \le s,t \le 1 \right\}. $$
Define the three points $x_0, x_1,$ and $x_2$ as
$$x_0 = \frac{l_1 + l_2}{2}, \quad x_1 = \frac{l_1+l_2}{3}, \quad x_2 = \frac{2(l_1 + l_2)}{3} .$$

\begin{figure}[tb]
	\centering
	\begin{tikzpicture}
	\begin{scope}[xshift=-5cm,scale=1.2]
	\coordinate (a) at (1,{1/sqrt(3)});		
	\coordinate (b) at (1,{-1/sqrt(3)});	
	\pgfmathsetmacro{\rb}{0.25pt}
	\pgfmathsetmacro{\rs}{0.2pt}
	
	\draw (0,0) -- (1,{1/sqrt(3)}) -- (2,0) -- (1,{-1/sqrt(3)}) -- cycle; 
	\begin{scope}[xshift = 1.33333cm]
	\draw plot [smooth cycle] coordinates {(0:\rb) (60:\rs) (120:\rb) (180:\rs) (240:\rb) (300:\rs) };
	\end{scope}
	\begin{scope}[xshift = 0.666667cm, rotate=60]
	\draw plot [smooth cycle] coordinates {(0:\rb) (60:\rs) (120:\rb) (180:\rs) (240:\rb) (300:\rs) };
	\end{scope}
	
	\draw[opacity=0.2] ({2/3},0) -- ({4/3},0)
	($0.5*(1,{1/sqrt(3)})$) -- ({2/3},0)
	($0.5*(1,{-1/sqrt(3)})$) -- ({2/3},0)
	($(1,{1/sqrt(3)})+0.5*(1,{-1/sqrt(3)})$) -- ({4/3},0)
	($0.5*(1,{1/sqrt(3)})+(1,{-1/sqrt(3)})$) -- ({4/3},0);
	
	\begin{scope}[shift = (a)]
	\begin{scope}[xshift = 1.33333cm]
	\draw plot [smooth cycle] coordinates {(0:\rb) (60:\rs) (120:\rb) (180:\rs) (240:\rb) (300:\rs) };
	\end{scope}
	\begin{scope}[xshift = 0.666667cm, rotate=60]
	\draw plot [smooth cycle] coordinates {(0:\rb) (60:\rs) (120:\rb) (180:\rs) (240:\rb) (300:\rs) };
	\end{scope}	
	\draw[opacity=0.2] ({2/3},0) -- ({4/3},0)
	($0.5*(1,{1/sqrt(3)})$) -- ({2/3},0)
	($0.5*(1,{-1/sqrt(3)})$) -- ({2/3},0)
	($(1,{1/sqrt(3)})+0.5*(1,{-1/sqrt(3)})$) -- ({4/3},0)
	($0.5*(1,{1/sqrt(3)})+(1,{-1/sqrt(3)})$) -- ({4/3},0);
	\end{scope}
	\begin{scope}[shift = (b)]
	\begin{scope}[xshift = 1.33333cm]
	\draw plot [smooth cycle] coordinates {(0:\rb) (60:\rs) (120:\rb) (180:\rs) (240:\rb) (300:\rs) };
	\end{scope}
	\begin{scope}[xshift = 0.666667cm, rotate=60]
	\draw plot [smooth cycle] coordinates {(0:\rb) (60:\rs) (120:\rb) (180:\rs) (240:\rb) (300:\rs) };
	\end{scope}
	\draw[opacity=0.2] ({2/3},0) -- ({4/3},0)
	($0.5*(1,{1/sqrt(3)})$) -- ({2/3},0)
	($0.5*(1,{-1/sqrt(3)})$) -- ({2/3},0)
	($(1,{1/sqrt(3)})+0.5*(1,{-1/sqrt(3)})$) -- ({4/3},0)
	($0.5*(1,{1/sqrt(3)})+(1,{-1/sqrt(3)})$) -- ({4/3},0);
	\end{scope}
	\begin{scope}[shift = ($-1*(a)$)]
	\begin{scope}[xshift = 1.33333cm]
	\draw plot [smooth cycle] coordinates {(0:\rb) (60:\rs) (120:\rb) (180:\rs) (240:\rb) (300:\rs) };
	\end{scope}
	\begin{scope}[xshift = 0.666667cm, rotate=60]
	\draw plot [smooth cycle] coordinates {(0:\rb) (60:\rs) (120:\rb) (180:\rs) (240:\rb) (300:\rs) };
	\end{scope}
	\draw[opacity=0.2] ({2/3},0) -- ({4/3},0)
	($0.5*(1,{1/sqrt(3)})$) -- ({2/3},0)
	($0.5*(1,{-1/sqrt(3)})$) -- ({2/3},0)
	($(1,{1/sqrt(3)})+0.5*(1,{-1/sqrt(3)})$) -- ({4/3},0)
	($0.5*(1,{1/sqrt(3)})+(1,{-1/sqrt(3)})$) -- ({4/3},0);
	\end{scope}
	\begin{scope}[shift = ($-1*(b)$)]
	\begin{scope}[xshift = 1.33333cm]
	\draw plot [smooth cycle] coordinates {(0:\rb) (60:\rs) (120:\rb) (180:\rs) (240:\rb) (300:\rs) };
	\end{scope}
	\begin{scope}[xshift = 0.666667cm, rotate=60]
	\draw plot [smooth cycle] coordinates {(0:\rb) (60:\rs) (120:\rb) (180:\rs) (240:\rb) (300:\rs) };
	\end{scope}
	\draw[opacity=0.2] ({2/3},0) -- ({4/3},0)
	($0.5*(1,{1/sqrt(3)})$) -- ({2/3},0)
	($0.5*(1,{-1/sqrt(3)})$) -- ({2/3},0)
	($(1,{1/sqrt(3)})+0.5*(1,{-1/sqrt(3)})$) -- ({4/3},0)
	($0.5*(1,{1/sqrt(3)})+(1,{-1/sqrt(3)})$) -- ({4/3},0);
	\end{scope}
	\begin{scope}[shift = ($(a)+(b)$)]
	\begin{scope}[xshift = 1.33333cm]
	\draw plot [smooth cycle] coordinates {(0:\rb) (60:\rs) (120:\rb) (180:\rs) (240:\rb) (300:\rs) };
	\end{scope}
	\begin{scope}[xshift = 0.666667cm, rotate=60]
	\draw plot [smooth cycle] coordinates {(0:\rb) (60:\rs) (120:\rb) (180:\rs) (240:\rb) (300:\rs) };
	\end{scope}
	\draw[opacity=0.2] ({2/3},0) -- ({4/3},0)
	($0.5*(1,{1/sqrt(3)})$) -- ({2/3},0)
	($0.5*(1,{-1/sqrt(3)})$) -- ({2/3},0)
	($(1,{1/sqrt(3)})+0.5*(1,{-1/sqrt(3)})$) -- ({4/3},0)
	($0.5*(1,{1/sqrt(3)})+(1,{-1/sqrt(3)})$) -- ({4/3},0);
	\end{scope}
	\begin{scope}[shift = ($-1*(a)-(b)$)]
	\begin{scope}[xshift = 1.33333cm]
	\draw plot [smooth cycle] coordinates {(0:\rb) (60:\rs) (120:\rb) (180:\rs) (240:\rb) (300:\rs) };
	\end{scope}
	\begin{scope}[xshift = 0.666667cm, rotate=60]
	\draw plot [smooth cycle] coordinates {(0:\rb) (60:\rs) (120:\rb) (180:\rs) (240:\rb) (300:\rs) };
	\end{scope}
	\draw[opacity=0.2] ({2/3},0) -- ({4/3},0)
	($0.5*(1,{1/sqrt(3)})$) -- ({2/3},0)
	($0.5*(1,{-1/sqrt(3)})$) -- ({2/3},0)
	($(1,{1/sqrt(3)})+0.5*(1,{-1/sqrt(3)})$) -- ({4/3},0)
	($0.5*(1,{1/sqrt(3)})+(1,{-1/sqrt(3)})$) -- ({4/3},0);
	\end{scope}
	\begin{scope}[shift = ($(a)-(b)$)]
	\begin{scope}[xshift = 1.33333cm]
	\draw plot [smooth cycle] coordinates {(0:\rb) (60:\rs) (120:\rb) (180:\rs) (240:\rb) (300:\rs) };
	\end{scope}
	\begin{scope}[xshift = 0.666667cm, rotate=60]
	\draw plot [smooth cycle] coordinates {(0:\rb) (60:\rs) (120:\rb) (180:\rs) (240:\rb) (300:\rs) };
	\end{scope}
	\draw[opacity=0.2] ({2/3},0) -- ({4/3},0)
	($0.5*(1,{1/sqrt(3)})$) -- ({2/3},0)
	($0.5*(1,{-1/sqrt(3)})$) -- ({2/3},0)
	($(1,{1/sqrt(3)})+0.5*(1,{-1/sqrt(3)})$) -- ({4/3},0)
	($0.5*(1,{1/sqrt(3)})+(1,{-1/sqrt(3)})$) -- ({4/3},0);
	\end{scope}
	\begin{scope}[shift = ($-1*(a)+(b)$)]
	\begin{scope}[xshift = 1.33333cm]
	\draw plot [smooth cycle] coordinates {(0:\rb) (60:\rs) (120:\rb) (180:\rs) (240:\rb) (300:\rs) };
	\end{scope}
	\begin{scope}[xshift = 0.666667cm, rotate=60]
	\draw plot [smooth cycle] coordinates {(0:\rb) (60:\rs) (120:\rb) (180:\rs) (240:\rb) (300:\rs) };
	\end{scope}
	\draw[opacity=0.2] ({2/3},0) -- ({4/3},0)
	($0.5*(1,{1/sqrt(3)})$) -- ({2/3},0)
	($0.5*(1,{-1/sqrt(3)})$) -- ({2/3},0)
	($(1,{1/sqrt(3)})+0.5*(1,{-1/sqrt(3)})$) -- ({4/3},0)
	($0.5*(1,{1/sqrt(3)})+(1,{-1/sqrt(3)})$) -- ({4/3},0);
	\end{scope}
	\end{scope}

	\draw[dashed,opacity=0.5,->] (-3.9,0.65) .. controls(-1.8,1.5) .. (1,0.7);
	\begin{scope}[scale=2.8]	
	\coordinate (a) at (1,{1/sqrt(3)});		
	\coordinate (b) at (1,{-1/sqrt(3)});	
	\coordinate (Y) at (1.8,0.45);
	\coordinate (c) at (2,0);
	\coordinate (x1) at ({2/3},0);
	\coordinate (x0) at (1,0);
	\coordinate (x2) at ({4/3},0);

	\pgfmathsetmacro{\rb}{0.25pt}
	\pgfmathsetmacro{\rs}{0.2pt}
	
	\begin{scope}[xshift = 1.33333cm]
	\draw plot [smooth cycle] coordinates {(0:\rb) (60:\rs) (120:\rb) (180:\rs) (240:\rb) (300:\rs) };
	\draw (0:\rb) node[xshift=7pt] {$D_2$};
	\end{scope}
	\begin{scope}[xshift = 0.666667cm, rotate=60]
	\draw plot [smooth cycle] coordinates {(0:\rb) (60:\rs) (120:\rb) (180:\rs) (240:\rb) (300:\rs) };
	\end{scope}
	\draw ({0.6666667-\rb},0) node[xshift=-7pt] {$D_1$};
	
	\draw (Y) node{$Y$};
	\draw[->] (0,0) -- (a) node[above,pos=0.7]{$l_1$};
	\draw[->] (0,0) -- (b) node[below,pos=0.7]{$l_2$};
	\draw (a) -- (c) -- (b);
	\draw[fill] (x1) circle(0.5pt) node[xshift=6pt,yshift=-6pt]{$x_1$}; 
	\draw[fill] (x0) circle(0.5pt) node[yshift=4pt, xshift=6pt]{$x_0$}; 
	\draw[fill] (x2) circle(0.5pt) node[xshift=6pt,yshift=4pt]{$x_2$}; 
	\draw[dashed] (1,0.7) node[right]{$p$} -- (1,-0.7);
	\end{scope}
	\end{tikzpicture}
	\caption{Illustration of the honeycomb crystal and quantities in $Y$.} \label{fig:honeycomb}
\end{figure}
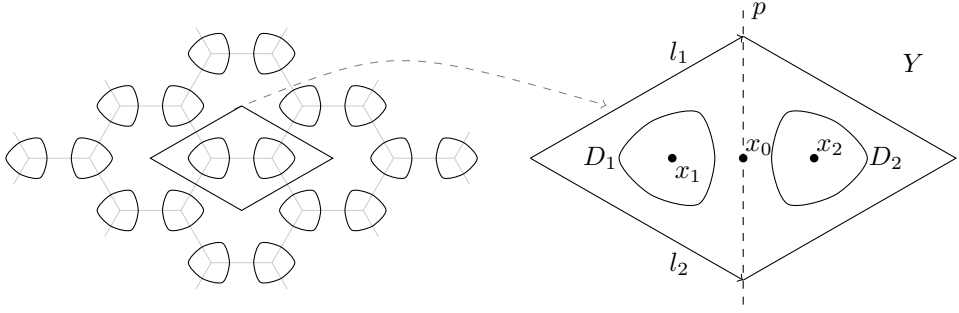

We will consider a general shape of the subwavelength resonators, under certain symmetry assumptions. Let $R_0$ be the rotation around $x_0$ by $\pi$, and let $R_1$ and $R_2$ be the rotations by $-\frac{2\pi}{3}$ around $x_1$ and $x_2$, respectively. These rotations can be written as
$$ R_1 x = Rx+l_1, \quad R_2 x = Rx + 2l_1, \quad R_0 x = 2x_0 - x , $$
where $R$ is the rotation by  $-\frac{2\pi}{3}$ around the origin. Moreover, let $R_3$ be the reflection across the line $p = x_0 + \R e_2$, where $e_2$ is the second standard basis element. Assume that the unit cell contains two subwavelength resonators $D_j$, $j=1,2$, each centred at $x_j$ such that
$$R_0 D_1 = D_2, \quad R_1 D_1 = D_1,\quad R_2 D_2 = D_2, \quad R_3D_1 = D_2.$$
We denote the pair of subwavelength resonators by $D=D_1 \cup D_2$. The dual lattice of $\Lambda$, denoted $\Lambda^*$, is generated by $\alpha_1$ and $\alpha_2$ given by
$$ \alpha_1 = \frac{2\pi}{L}\left( \frac{1}{\sqrt{3}}, 1\right),~~\alpha_2 = \frac{2\pi}{L}\left(\frac{1}{\sqrt{3}}, -1 \right).$$
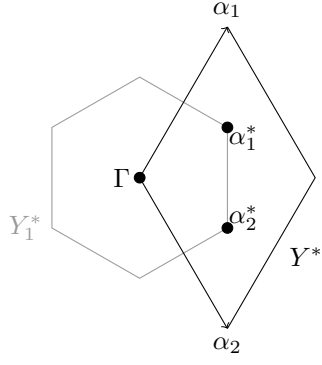
\begin{figure}[h]
	\centering
	\begin{tikzpicture}[scale=2]	
	\coordinate (a) at ({1/sqrt(3)},1);		
	\coordinate (b) at ({1/sqrt(3)},-1);
	\coordinate (c) at ({2/sqrt(3)},0);
	\coordinate (K1) at ({1/sqrt(3)},{1/3});
	\coordinate (K2) at ({1/sqrt(3)},{-1/3});
	\coordinate (K3) at (0,{-2/3});
	\coordinate (K4) at ({-1/sqrt(3)},{-1/3});
	\coordinate (K5) at ({-1/sqrt(3)},{1/3});
	\coordinate (K6) at (0,{2/3});
	
	\draw[->] (0,0) -- (a) node[above]{$\alpha_1$};
	\draw[->] (0,0) -- (b) node[below]{$\alpha_2$};
	\draw (a) -- (c) -- (b) node[pos=0.4,below right]{$Y^*$};
	\draw[fill] (K1) circle(1pt) node[xshift=6pt,yshift=-4pt]{$\alpha_1^*$}; 
	\draw[fill] (K2) circle(1pt) node[xshift=6pt,yshift=4pt]{$\alpha_2^*$}; 
	\draw[fill] (0,0) circle(1pt) node[left]{$\Gamma$}; 
	
	\draw[opacity=0.4] (K1) -- (K2) -- (K3) -- (K4) node[left]{$Y_1^*$} -- (K5) -- (K6) -- cycle; 
	\end{tikzpicture}
\caption{Illustration of the dual lattice and the Brillouin zone $Y^*$.}
\label{fig:bz}
\end{figure}
The Brillouin zone $Y^*:= {\R^2}/{\Lambda^*}$ can be represented either as  
$$Y^* \simeq \left\{ s \alpha_1+ t \alpha_2 ~|~ 0 \le s,t \le 1 \right\}, $$
or as the first Brillouin zone $Y_1^*$, which is a hexagon illustrated in Figure \ref{fig:bz}.
The points $$\alpha_1^*= \frac{2\alpha_1+\alpha_2}{3}, \quad \alpha^*_2 = \frac{\alpha_1+2\alpha_2}{3},$$ in the Brillouin zone are called \emph{Dirac points}. For simplicity, we only consider the analysis around the Dirac point $\alpha_* := \alpha_1^*$, the main difference around $\alpha_2^*$ is summarized in Remark \ref{rmk:alpha2}. 

Wave propagation in the  honeycomb lattices of subwavelength resonators is described by the following $\alpha$-quasi-periodic Helmholtz problem in $Y$:
\begin{equation}  \label{HP1}
\left\{
\begin{array} {lll}
&\ds \nabla \cdot \frac{1}{\rho} \nabla  u+ \frac{\omega^2}{\kappa} u  = 0 \quad &\text{in} \ Y \backslash \overline{D}, \\
\nm
&\ds \nabla \cdot \frac{1}{\rho_b} \nabla  u+ \frac{\omega^2}{\kappa_b} u  = 0 \quad &\text{in} \ D, \\
\nm
&\ds  u |_{+} -u |_{-}  =0   \quad &\text{on} \ \partial D, \\
\nm
& \ds  \frac{1}{\rho} \frac{\partial u}{\partial \nu} \bigg|_{+} - \frac{1}{\rho_b} \frac{\partial u}{\partial \nu} \bigg|_{-} =0 \quad &\text{on} \ \partial D, \\ 
\nm
& u(x+l)= e^{\i \alpha\cdot l} u(x) \quad & \text{for all} \ l\in \Lambda.
\end{array}
\right.
\end{equation}

Let $\psi_j^{\alpha}\in L^2(\p D)$  be given by
\begin{align} \label{psi_def}
\mathcal{S}_D^{\alpha,0}[\psi_j^{\alpha}] = \chi_{\partial D_j}\quad \mbox{on}~\partial D,\quad j=1, 2.
\end{align}
Define the capacitance  matrix $C^\alpha=(C_{ij}^\alpha)$ by
\begin{equation} \label{defcapal} C_{ij}^\alpha := - \int_{\partial D_i} \psi_j^\alpha \;\dx \sigma,\quad i,j=1, 2.\end{equation}
Using the symmetry of the honeycomb structure, it can be shown that the capacitance coefficients satisfy \cite{ammari2018honeycomb}
$$c_1^\alpha := C_{11}^\alpha = C_{22}^{\alpha}, \quad c_2^\alpha := C_{12}^{\alpha} = \overline{C_{21}^{\alpha}},$$ 
and
\begin{equation}\label{c1c2_deri}
\nabla_\alpha c_1^\alpha \Big|_{\alpha=\alpha^*} = 0,
\quad
\nabla_\alpha c_2^\alpha \Big|_{\alpha=\alpha^*} = c\begin{pmatrix} 1\\-\iu\end{pmatrix},
\end{equation}
where we denote $$c:=\frac{\partial c_2^{\alpha}}{\partial \alpha_1}\Big|_{\alpha=\alpha^*} \neq 0,$$
as proved in \cite[Lemma 3.4]{ammari2018honeycomb}.

It is shown in \cite{ammari2018honeycomb} that the first two band functions $\omega_1^\alpha$ and $\omega_2^\alpha$ form a conical dispersion relation near the Dirac point $\alpha_*$. Such a conical dispersion is referred to as a {\it Dirac cone}. More specifically, the following results which hold in the  subwavelength regime are proved in  \cite{ammari2018honeycomb}. 
\begin{theorem}\label{thm:honeycomb}
	For small $\delta$, the first two band functions $\omega_j^\alpha,j=1,2$, satisfy 
	\begin{equation} \label{eq:wasymp}
	\omega_j^\alpha = \sqrt{\frac{\delta\lambda_j^\alpha }{|D_1|}}v_b + \O(\delta),
	\end{equation}
	uniformly for $\alpha$ in a neighbourhood of $\alpha_*$, where $\lambda_j^\alpha, j=1,2,$ are the two eigenvalues of $C^\alpha$ and $|D_1|$ denotes the area of one of the subwavelength resonators. Moreover, for $\alpha$ close to $\alpha_*$ and $\delta$ small enough, the first two band functions form a Dirac cone, \ie{},
	\begin{equation} \label{eq:dirac}
	\begin{matrix}
	\ds \omega_1^\alpha = \omega_*- \lambda|\alpha - \alpha_*| \big[ 1+ \O(|\alpha-\alpha_*|) \big], \\[0.5em]
	\ds \omega_2^\alpha = \omega_*+ \lambda|\alpha - \alpha_*| \big[ 1+ \O(|\alpha-\alpha_*|) \big],
	\end{matrix}
	\end{equation}
	where $\omega_*$ and $\lambda$ are independent of $\alpha$ and satisfy
	$$\omega_*= \sqrt{\frac{\delta c_1^{\alpha_*}}{|D_1|}}v_b + \O(\delta) \quad \text{and} \quad \lambda =  |c|\sqrt\delta\lambda_0 + \O(\delta),  \quad \lambda_0=\frac{1}{2}\sqrt{\frac{v_b^2 }{|D_1|c_1^{\alpha_*}}}$$
	as $\delta \rightarrow 0$. Moreover, the error term $\O(|\alpha-\alpha_*|)$ in (\ref{eq:dirac}) is uniform in $\delta$.
\end{theorem}

The results in Theorem \ref{thm:honeycomb} are illustrated in Figure \ref{honeycombfz}. The figure shows the first three bands. Observe that the first two bands cross at the symmetry point $K$ (corresponding to $\alpha_*$) such that the dispersion relation is linear.  Figure \ref{squarefz} shows the band gap structure in the subwavelength region for a rectangular array of subwavelength dimers. For such arrays, the two first bands cannot cross each other.

\begin{figure}[!h]
\begin{center}
 \includegraphics[height=5.cm]{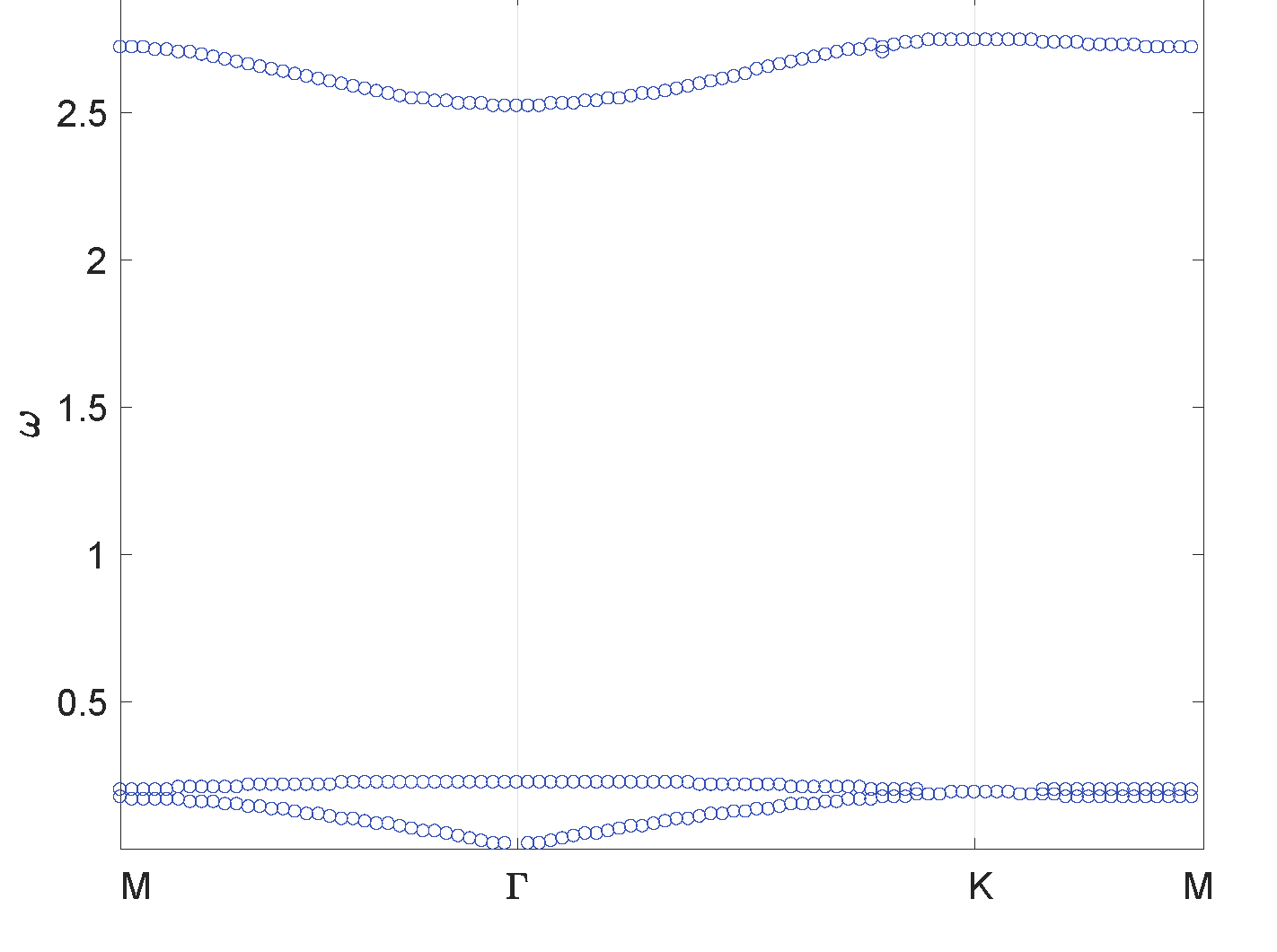}
 \hspace{0.25cm}
  \includegraphics[height=5.cm]{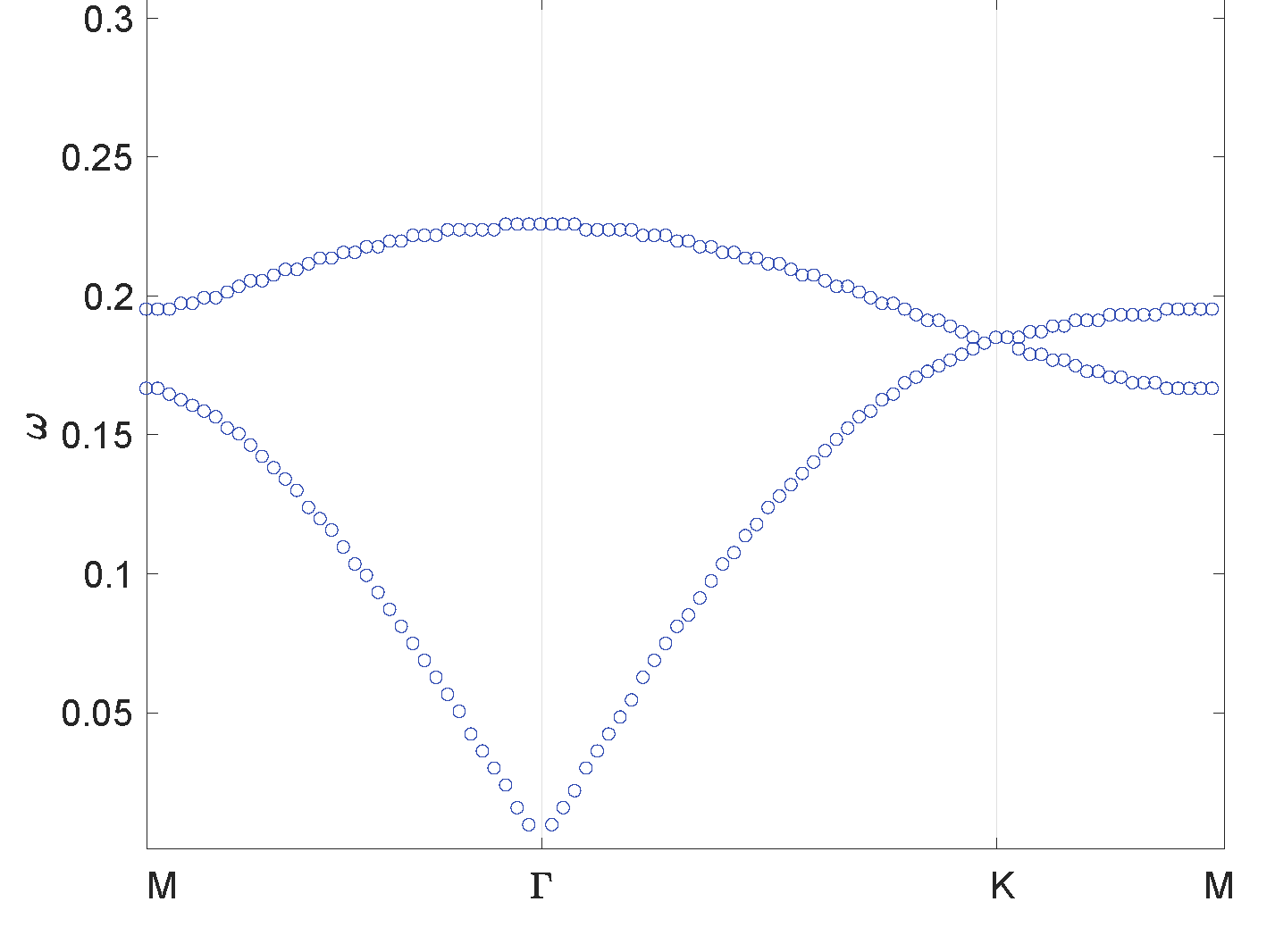}
  \caption{Band gap structure upon zooming in the subwavelength region for a honeycomb lattice of subwavelength resonators.}  \label{honeycombfz}
\end{center}
\end{figure}

\begin{figure}[!h]
\begin{center}
 \includegraphics[height=5.cm]{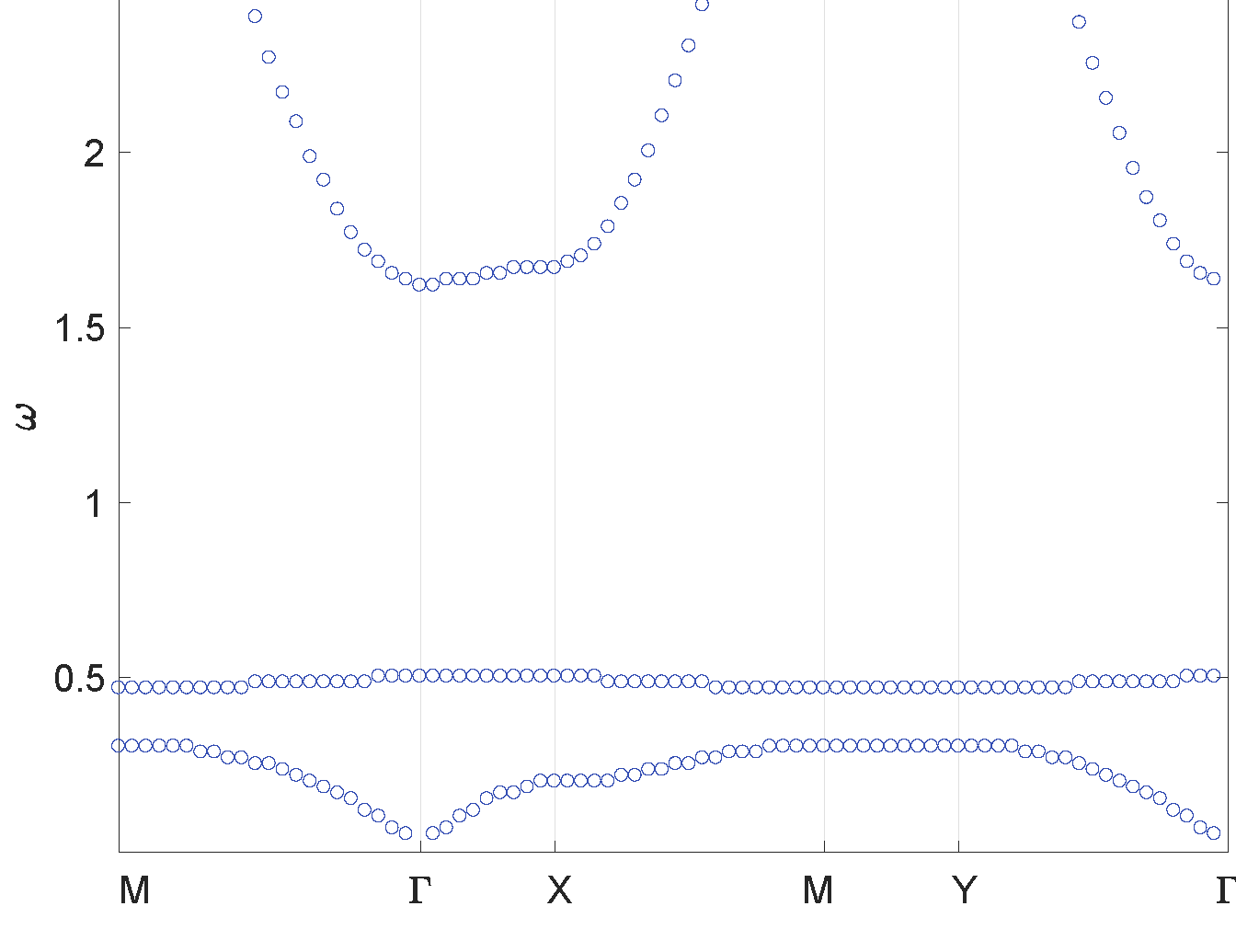}
 \hspace{0.25cm}
  \includegraphics[height=5.cm]{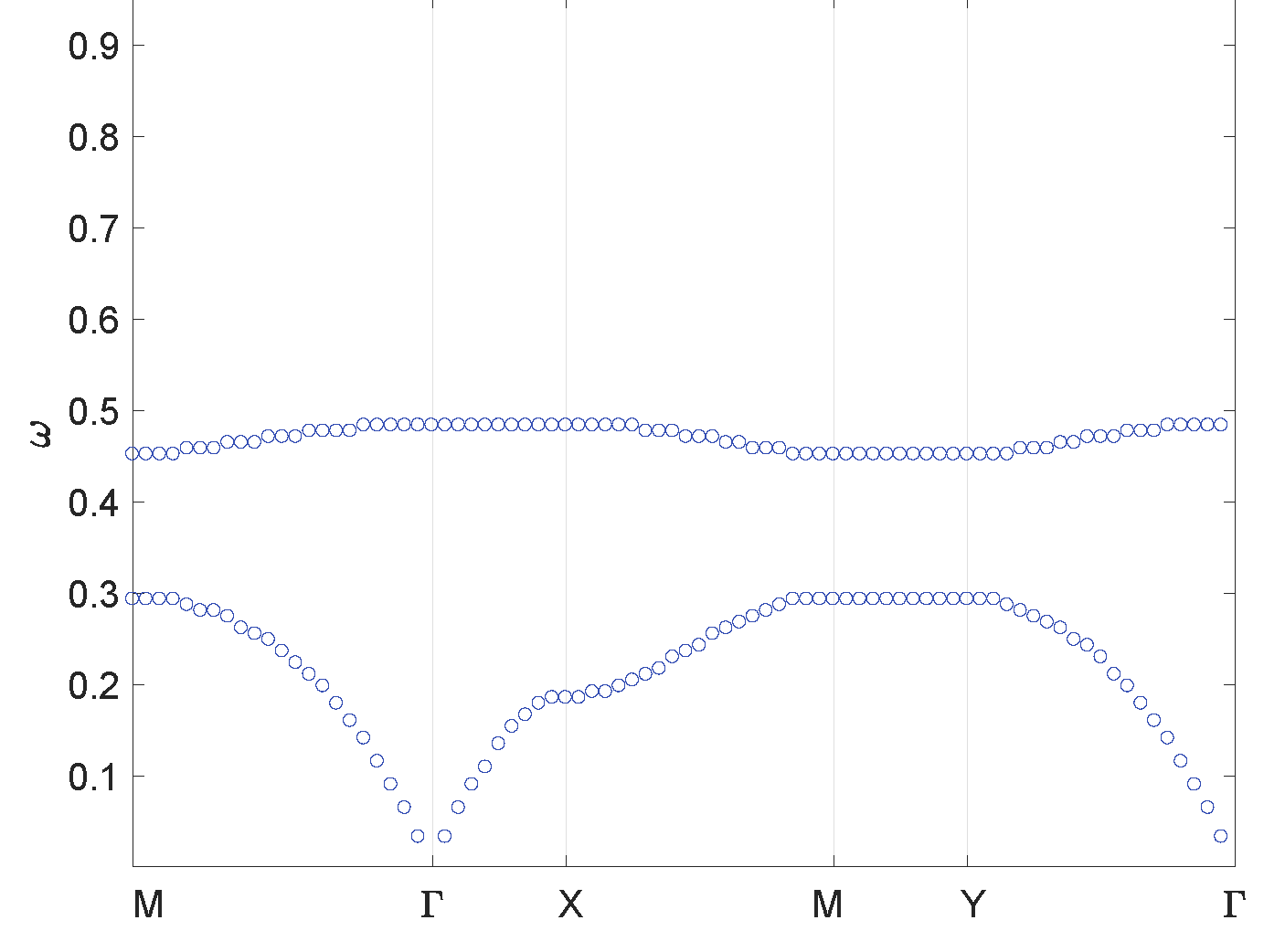}
   \caption{The band gap structure upon zooming in the subwavelength region for a rectangular array of subwavelength dimers.}  \label{squarefz}
\end{center}
\end{figure}

Next we investigate the asymptotic behaviour of the  Bloch eigenfunctions  near the Dirac points. Then we show that the envelopes of the Bloch eigenfunctions satisfy a Helmholtz equation with near-zero effective refractive index and  derive a two-dimensional homogenized equation of Dirac-type for the honeycomb crystal. These results are from \cite{nearzero}. 

We consider the rescaled  honeycomb crystal by replacing the lattice constant $L$ with $r L$ where $r>0$ is a small positive parameter. 
Let $\omega_j^\alpha,j=1,2$,  be the first two eigenvalues  and $u_j^\alpha$ be the associated Bloch eigenfunctions for the honeycomb crystal with lattice constant $L$. Then, by a scaling argument, the honeycomb crystal with lattice constant $r L$ has the first two Bloch eigenvalues
$$
\omega_{\pm,r}^{\alpha/r} = \frac{1}{r} \omega_\pm^\alpha,
$$
and the corresponding eigenfunctions are 
$$
u^{\alpha/r}_{\pm,r}(x) = u_\pm^\alpha\left( \frac{x}{r}\right).
$$
This shows that the Dirac cone is located at the point $\alpha_*/r$. 
We denote the Dirac frequency by
$$
\omega^r_* = \frac{1}{r}\omega_*.
$$

We have the following result for the Bloch eigenfunctions $u_{j,r}^{\alpha/r},j=1,2,$ for $\alpha/r$ near the Dirac points $\alpha_*/r$ \cite{nearzero}.

\begin{lemma}\label{lem:Bloch_scale}
We have
$$
u_{\pm,r}^{\alpha_*/r+\tilde{\alpha}}(x) = A_\pm e^{\i \widetilde\alpha \cdot x} S_{1} \left(\frac{x}{r}\right) + B_\pm e^{\i \widetilde\alpha \cdot x} S_{2} \left(\frac{x}{r}\right) + \O(\delta + r),
$$
where
$$
S_j (x) = \mathcal{S}_D^{\alpha_*,0} [\psi_j^{\alpha_*}](x), \quad j=1,2.
$$
The functions $S_1$ and $S_2$ describe the microscopic behaviour of the Bloch eigenfunction $u_{\pm,r}^{\alpha_*/r+\widetilde{\alpha}}$ while $A_\pm e^{\i \widetilde\alpha\cdot x}$ and  $B_\pm e^{\iu\tilde\alpha\cdot x}$ describe the macroscopic behaviour.
\end{lemma}

Now, we derive a homogenized equation near the Dirac frequency $\omega^r_*$. Recall that the Dirac frequency of the unscaled honeycomb crystal satisfies $\omega_* = \O(\sqrt\delta) $.
As in Theorem \ref{main2}, in order to make the order of $\omega^r_*$ fixed when $r$ tends to zero, we assume that
$\delta = \mu r^2$
for some fixed $\mu>0$. 
Then we have
$$
\omega^r_* = \frac{1}{r} \omega_*= \O(1) \quad \mbox{as } r\rightarrow 0.
$$
So, in what follows, we omit the subscript $r$ in $\omega^r_*$, namely, $\omega_*:=\omega^r_*$.
Suppose the frequency $\omega$ is close to $\omega_*$, \textit{i.e.},
$$
\omega-\omega_* = \beta \sqrt\delta \quad \mbox{for some constant } \beta.
$$
We need to find the Bloch eigenfunctions or $\widetilde{\alpha}$  such that
$$
\omega = \omega_{\pm,r}^{\alpha_*/r + \widetilde\alpha}.
$$ 
We have  that the corresponding $\widetilde{\alpha}$ satisfies
\begin{align*}
\lambda_0
\begin{bmatrix}
 0 &  c( \widetilde{\alpha}_1 - \i \widetilde{\alpha}_2)
 \\
 \overline{c}( \widetilde{\alpha}_1 + \i \widetilde{\alpha}_2) & 0
\end{bmatrix}
\begin{bmatrix}
A_\pm
\\
B_\pm
\end{bmatrix}
= \beta
\begin{bmatrix}
A_\pm
\\
B_\pm 
\end{bmatrix} + \O(s).
\end{align*}
So, it is immediate to see that the macroscopic field $
 [\tilde{u}_{1}, \tilde{u}_{2}]^T:=[A_\pm e^{\iu\tilde{\alpha}\cdot x}, B_\pm e^{\iu\tilde{\alpha}\cdot x}]^T$ satisfies the system of Dirac equations as follows:
$$
\lambda_0
\begin{bmatrix}
 0 & (-c\iu)( \p_1 - \iu \p_2)
 \\
 (-\overline{c}\iu)( \p_1 + \iu \p_2) & 0
\end{bmatrix}
\begin{bmatrix}
\tilde{u}_{1}
\\
\tilde{u}_{2}
\end{bmatrix}
= \beta
\begin{bmatrix}
\tilde{u}_{1}
\\
\tilde{u}_{2}
\end{bmatrix}.
$$ 
Here, the superscript $T$ denotes the transpose and $\partial_i$ is the partial derivative with respect to the $i$th variable. 
Note that the each component $\tilde{u}_j,j=1,2$, of the macroscopic field satisfies the Helmholtz equation
\begin{equation} \label{eq:hom}
\Delta \tilde{u}_j + \frac{\beta^2}{|c|^2\lambda_0^2} \tilde{u}_j = 0.
\end{equation}
Observe, in particular, that \eqref{eq:hom} describes a near-zero refractive index when $\beta$ is small.

The following is the main result on the homogenization theory for honeycomb  lattices of subwavelength resonators \cite{nearzero}.

\begin{theorem} \label{thm:main}
For frequencies $\omega$ close to the Dirac frequency $\omega_*$, namely, $\omega-\omega_* = \beta \sqrt\delta$, the following asymptotic behaviour of the Bloch eigenfunction $u^{\alpha_*/r + \tilde{\alpha}}_r$ holds:
$$
u_{r}^{\alpha_*/r+\tilde{\alpha}}(x) = A e^{\iu\tilde\alpha \cdot x} S_{1} \left(\frac{x}{s}\right) + B e^{\iu\tilde\alpha \cdot x} S_{2} \left(\frac{x}{s}\right) + \O(s),
$$  
where the macroscopic field $[\tilde{u}_{1}, \tilde{u}_{2}]^T:=[A e^{\iu\tilde{\alpha}\cdot x}, B e^{\iu\tilde{\alpha}\cdot x}]^T$ satisfies the two-dimensional Dirac equation
$$
\lambda_0
\begin{bmatrix}
 0 & (-c\iu)( \p_1 - \iu \p_2)
 \\
 (-\overline{c}\iu)( \p_1 + \iu \p_2) & 0
\end{bmatrix}
\begin{bmatrix}
\tilde{u}_{1}
\\
\tilde{u}_{2}
\end{bmatrix}
= \frac{\omega-\omega_*}{\sqrt\delta}
\begin{bmatrix}
\tilde{u}_{1}
\\
\tilde{u}_{2}
\end{bmatrix},
$$ 
which can be considered as a homogenized equation for the honeycomb lattice of subwavelength resonators while the microscopic fields $S_1$ and $S_2$ vary on the scale of $r$.
\end{theorem}

Figure \ref{honeycomb} shows a one-dimensional plot along the $x_1-$axis of the real part of the {Bloch} {eigenfunction} of the honeycomb lattice shown over many unit cells. 

\begin{figure}[h]
\begin{center} \includegraphics[scale=0.6]{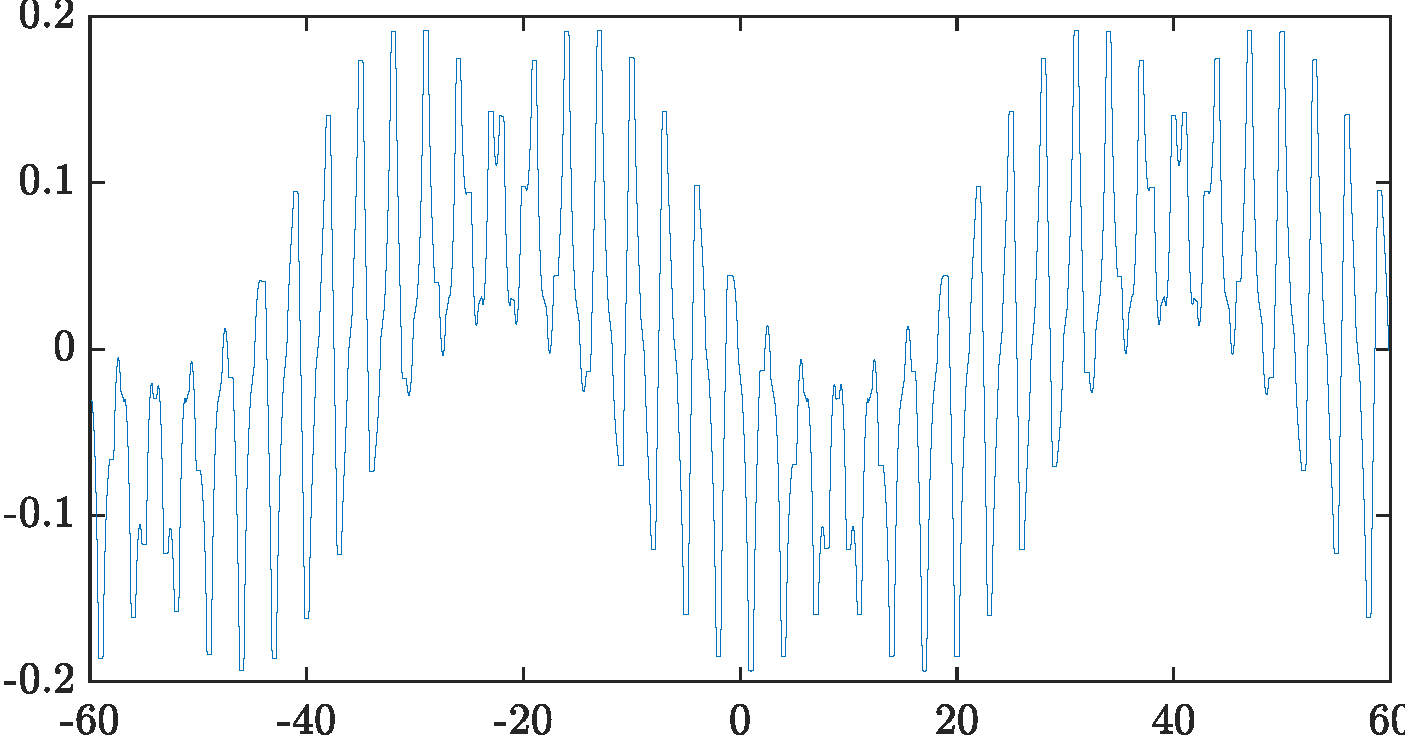}  
\caption{Real part of the first Bloch eigenfunction of the honeycomb lattice shown over many unit cells. } \label{honeycomb}  \end{center}
\end{figure}

\begin{remark} \label{rmk:alpha2}
	Theorem \ref{thm:main} is valid around the Dirac point $\alpha_* = \alpha_1^*$.	Around the other Dirac point, analogous arguments show that Theorem \ref{thm:main} is valid with all quantities instead defined using $\alpha_*=\alpha_2^*$ and the macroscopic field now satisfying 
	$$
	\lambda_0
	\begin{bmatrix}
	0 & (-c\iu)( \p_1 + \iu \p_2)
	\\
	(-\overline{c}\iu)( \p_1 - \iu \p_2) & 0
	\end{bmatrix}
	\begin{bmatrix}
	\tilde{u}_{1}
	\\
	\tilde{u}_{2}
	\end{bmatrix}
	= \frac{\omega-\omega_*}{\sqrt\delta}
	\begin{bmatrix}
	\tilde{u}_{1}
	\\
	\tilde{u}_{2}
	\end{bmatrix}.
	$$ 
\end{remark}

\subsection{Topologically protected edge modes}
A typical way to enable localized modes is to create a cavity inside a band gap structure. The idea is to make the frequency of the cavity mode fall within the band gap, whereby the mode will be localized to the cavity. However, localized modes created this way are highly sensitive to imperfections of the structure.

The principle that underpins the design of robust structures is that one is able to define topological invariants which capture the crystal's wave propagation properties. Then, if part of a crystalline structure is replaced with an arrangement that is associated with a different value of this invariant, not only will certain frequencies be localized to the interface but this behaviour will be stable with respect to imperfections. These eigenmodes are known as \emph{edge modes} and we say that they are \emph{topologically protected} to refer to their robustness. 

\subsubsection{Sensitivity to geometric imperfections} 
Subwavelength metamaterials can be used to achieve cavities of subwavelength dimensions. 
The key idea is to perturb the size of a single subwavelength resonator inside the crystal, thus creating a defect mode. Observe that if we remove one resonator inside the bubbly crystal, we cannot create a defect mode. The defect created in this fashion is actually too small to support a resonant mode. In  \cite{defectSIAM},  it is proved that by perturbing the radius of one resonator (see Figure \ref{fig:defect} where $D_\epsilon$ is the defect resonator) we create a detuned resonator with a resonant frequency that fall within the subwavelength band gap. Moreover, it is shown that the way to shift the frequency into the band gap depends on the crystal: in the dilute regime we have to decrease the defect resonator size while in the non-dilute regime we have to increase the size. 

\begin{figure}[tb]
\centering
    \begin{tikzpicture}[scale=1.6]
    \centering
    \draw[dashed] (0,0) circle (8pt) node{$D$};
    \draw (0,0) circle (12pt) node[yshift=-20pt, xshift=15pt ]{$D_\epsilon$};
    \draw (1,0) circle (8pt);
    \draw (0,1) circle (8pt);
    \draw (1,1) circle (8pt);
    \draw (-1,0) circle (8pt);
    \draw (0,-1) circle (8pt);
    \draw (1,-1) circle (8pt);
    \draw (-1,1) circle (8pt);
    \draw (-1,-1) circle (8pt);
    \draw (1.5,0) node{$\cdots$};
    \draw (-1.5,0) node{$\cdots$};
    \draw (0,1.5) node{$\vdots$};
    \draw (0,-1.5) node{$\vdots$};
    \end{tikzpicture}
    \hspace{0.5cm}
    {\includegraphics[width=7cm]{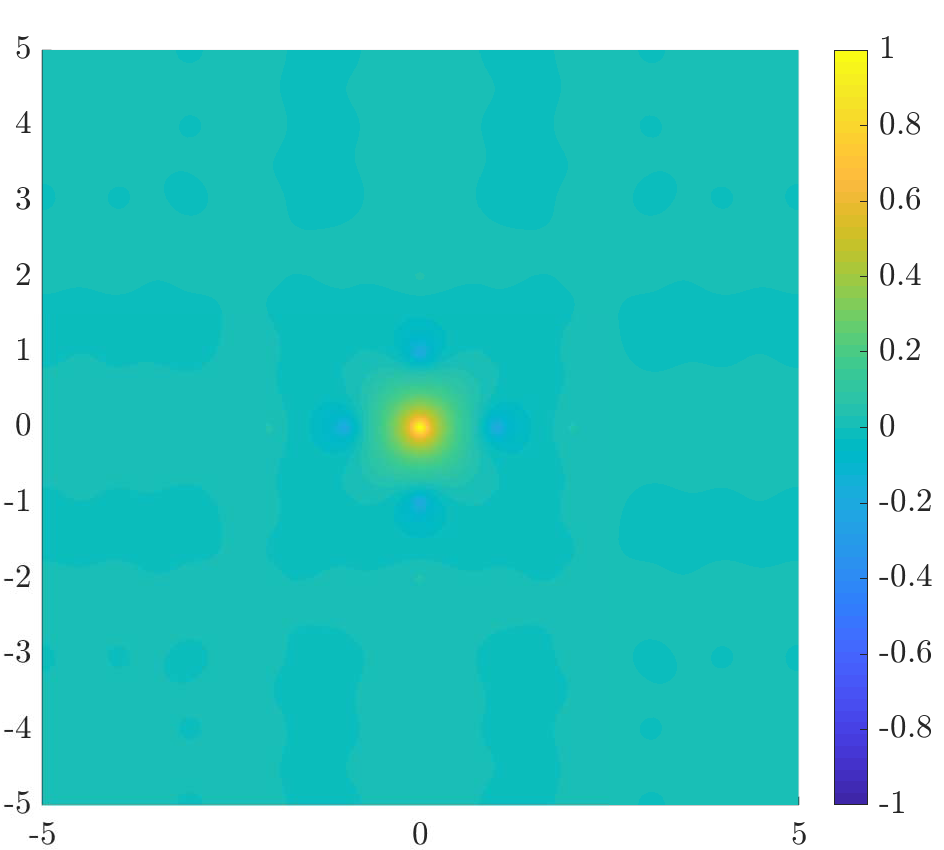}} 
   \caption{Illustration of the defect crystal and the defect mode.} \label{fig:defect}
\end{figure}
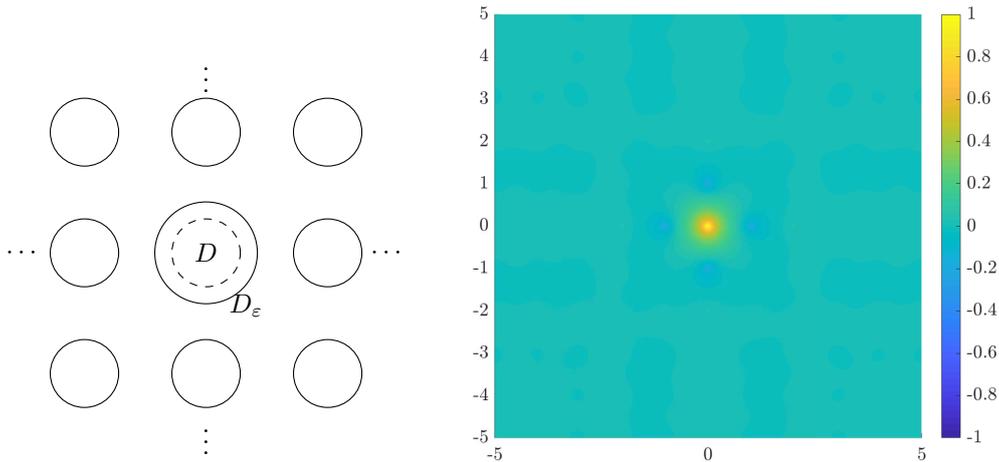

In \cite{linedefect},  a waveguide is created by modifying the sizes of the resonators along a line in a dilute two-dimensional crystal, thereby creating a line defect.
It is proved that the line defect indeed acts as a waveguide; waves of certain frequencies will be localized to, and guided along, the line defect. This is depicted in Figure \ref{linedefectf}.

In wave localization due to a point defect, if the defect size is small the band structure of the defect problem will be a small perturbation of the band structure of the original problem. This way, it is possible to shift the defect band upwards, and a part of the defect band will fall into the subwavelength band gap. In \cite{linedefect} it is shown that for arbitrarily small defects, a part of the defect band will lie inside the band gap. Moreover, it is shown that for suitably large perturbation sizes, the entire defect band will fall into the band gap, and the size of the perturbation needed in order to achieve this can be explicitly quantified. 
In order to have \textit{guided} waves along the line defect, the defect mode must not only be localized to the line, but also propagating along the line. In other words, we must exclude the case of standing waves in the line defect, \textit{i.e.}, modes which are localized in the direction of the line. Such modes are associated to a point spectrum of the perturbed operator which appears as a flat band in the dispersion relation. In \cite{linedefect}, it is shown that the defect band is nowhere flat, and hence does not correspond to bound modes in the direction of the line.

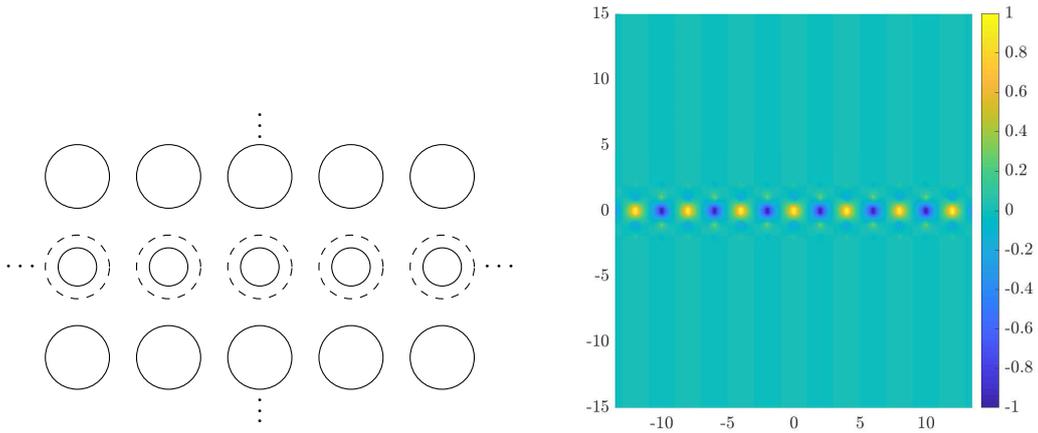
\begin{figure}[h]
    \centering
    \begin{tikzpicture}[scale=1.2]
    \draw[dashed] (0,0) circle (10pt);
    \draw (0,0) circle (6pt);
    \draw[dashed] (-2,0) circle (10pt);
    \draw[dashed] (-1,0) circle (10pt);
    \draw[dashed] (2,0) circle (10pt);
    \draw[dashed] (1,0) circle (10pt);
    \draw (-2,0) circle (6pt);
	\draw (-1,0) circle (6pt);
	\draw (2,0) circle (6pt);
	\draw (1,0) circle (6pt);    
    
    \draw (0,1) circle (10pt);
    \draw (1,1) circle (10pt);
    \draw (0,-1) circle (10pt);
    \draw (1,-1) circle (10pt);
    \draw (-1,1) circle (10pt);
    \draw (-1,-1) circle (10pt);
    \draw (2,1) circle (10pt);
    \draw (2,-1) circle (10pt);
    \draw (-2,1) circle (10pt);
    \draw (-2,-1) circle (10pt);
    \draw (2.65,0) node{$\cdots$};
    \draw (-2.6,0) node{$\cdots$};
    \draw (0,1.65) node{$\vdots$};
    \draw (0,-1.5) node{$\vdots$};
%
    \end{tikzpicture}
    \hspace{0.5cm}
 {\includegraphics[width=6.cm]{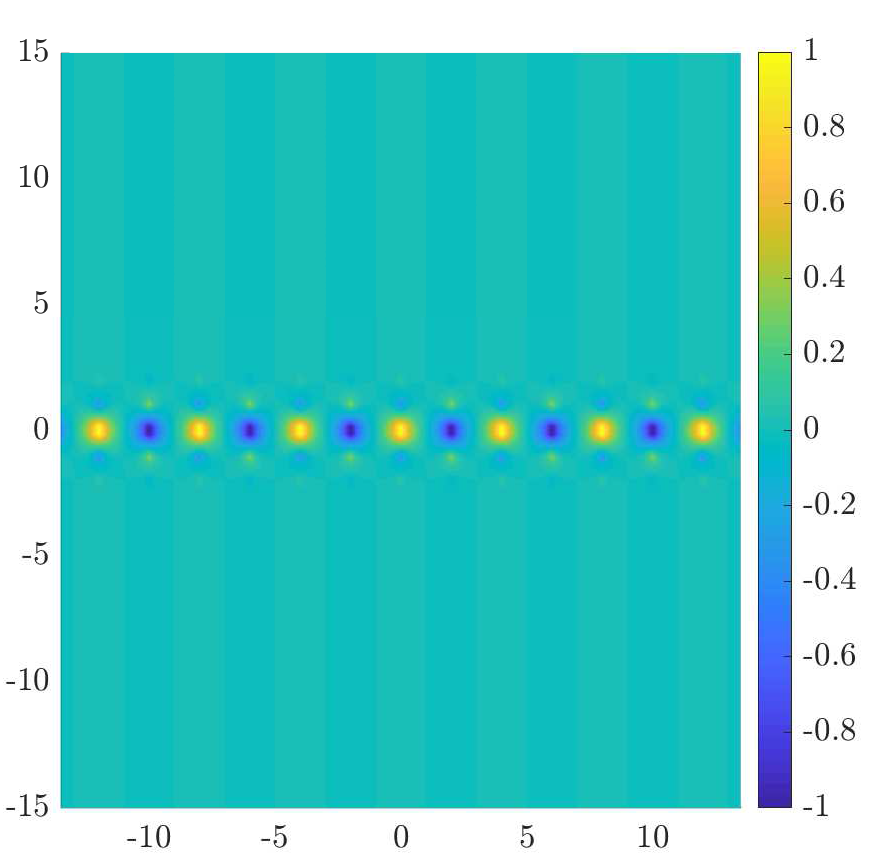}}
 \caption{Illustration of the line defect and the guided mode.} \label{linedefectf}
\end{figure}

One fundamental limitation of the above designs of subwavelength cavities and waveguides is that their properties are often very sensitive to imperfections in the crystal's structure. This is due, as illustrated in Figure \ref{defectotstable}, to the fact that the frequencies of the defect modes and guided waves are very close to the original band.    In order to be able to feasibly manufacture wave-guiding devices, it is important that we are able to design subwavelength crystals that exhibit stability with respect to geometric errors.

\begin{figure}
\centering
\includegraphics[height=4.9cm]{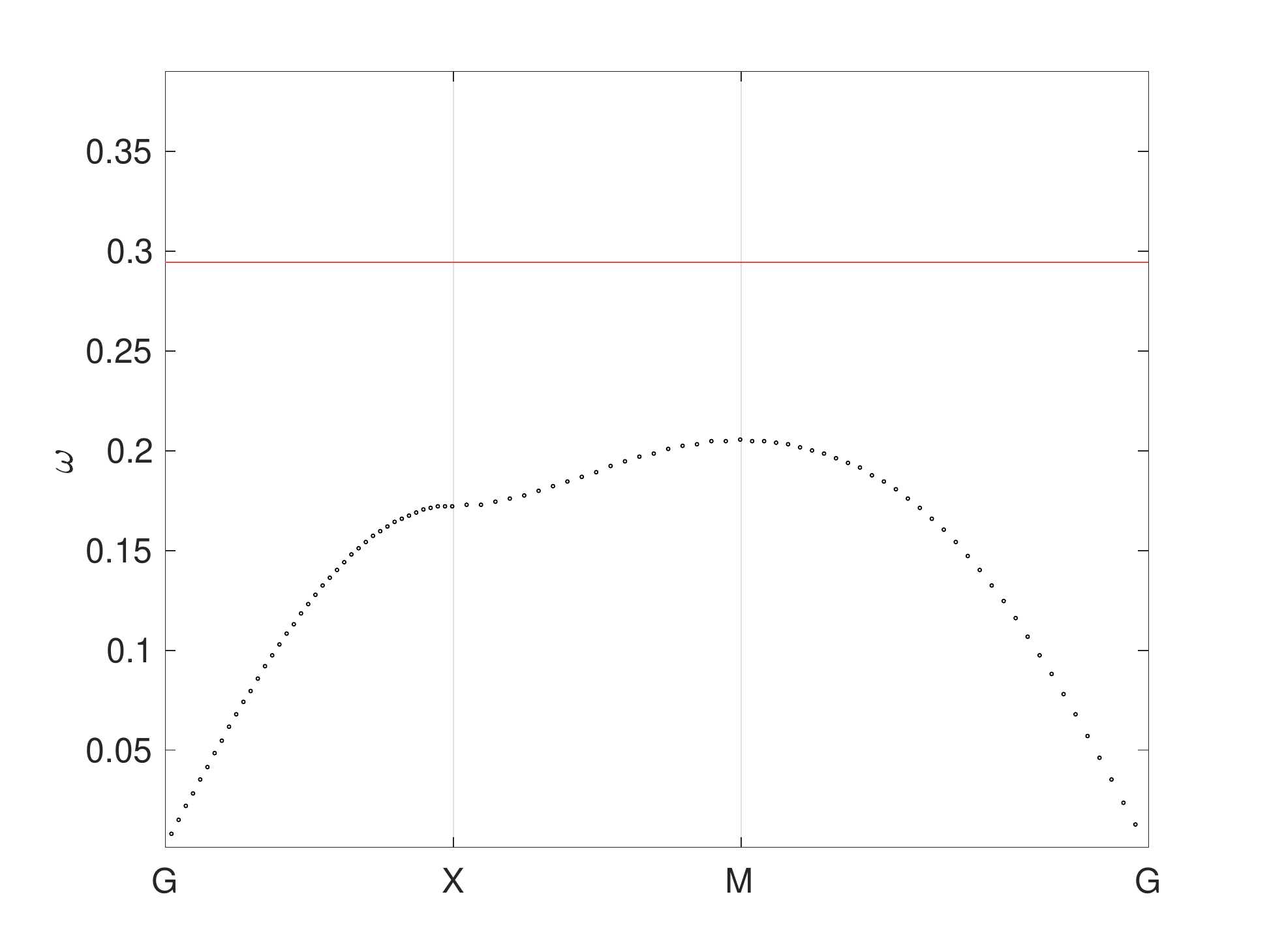} \hspace{0.25cm} \includegraphics[scale=0.38] {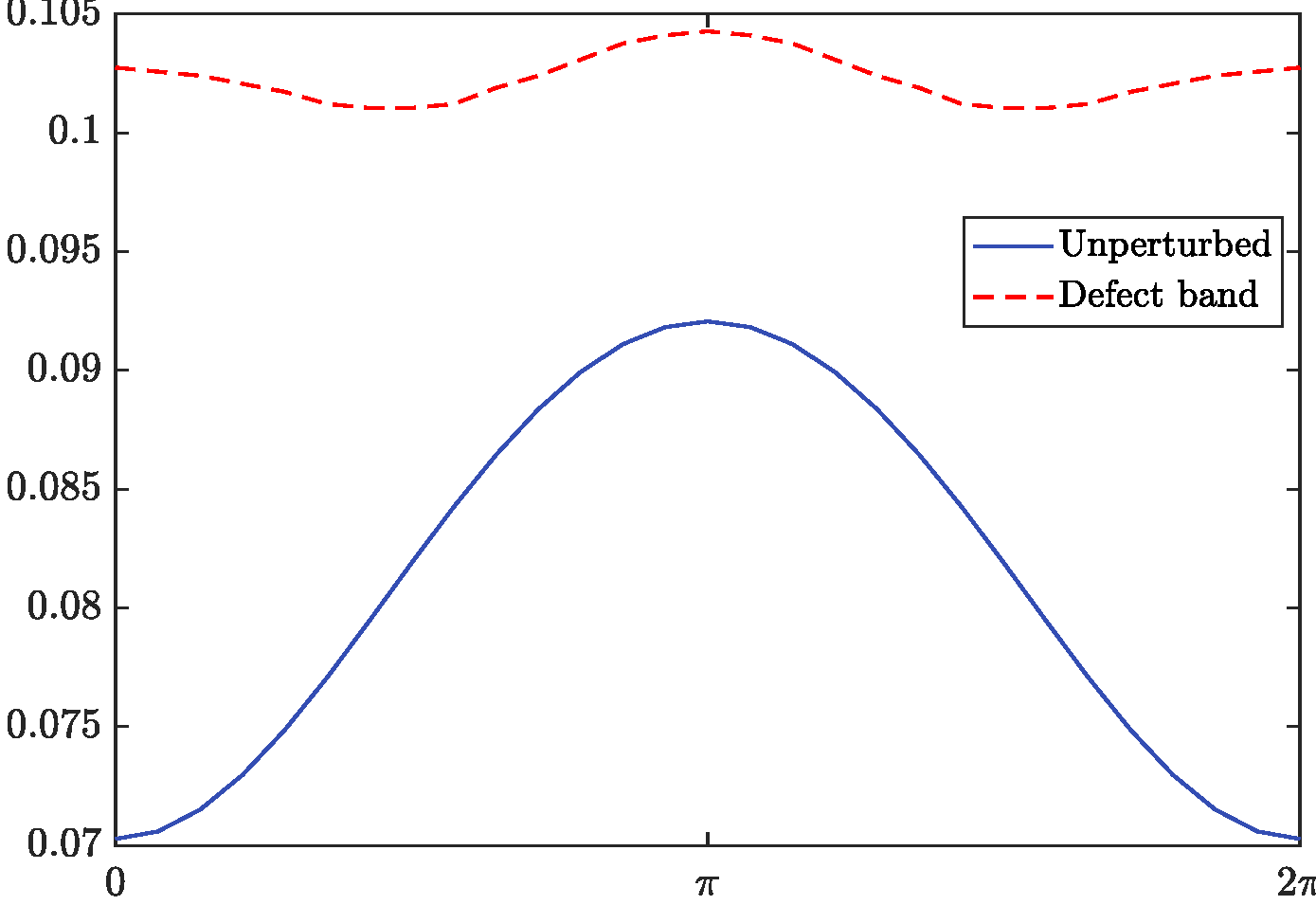} \caption{Frequencies of the defect modes and guided waves.} \label{defectotstable} \end{figure}

\subsubsection{Robustness properties of one-dimensional chains of subwavelength resonators with respect to imperfections}	
	
	In the case of one-dimensional crystals such as a chain of subwavelength resonators, the natural choice of topological invariant is the Zak phase \cite{zak}. Qualitatively, a non-zero Zak phase means that the crystal has undergone \emph{band inversion}, meaning that at some point in the Brillouin zone the monopole/dipole nature of the first/second Bloch eigenmodes has swapped. In this way, the Zak phase captures the crystal's wave propagation properties. If one takes two chains of subwavelength resonators with different Zak phases and joins half of one chain to half of the other to form a new crystal, this crystal will exhibit a topologically protected edge mode at the interface, as illustrated in Figure \ref{zakf}.

\begin{figure}[h]
\begin{center}
	\begin{tikzpicture}
	\pgfmathsetmacro{\cubex}{4}
	\pgfmathsetmacro{\cubey}{0.6}
	\pgfmathsetmacro{\cubez}{0.6}
	\draw [draw=blue, every edge/.append style={draw=blue, densely dashed, opacity=.5}, fill=blue!30!white]
	(0,0,0) coordinate (o) -- ++(-\cubex,0,0) coordinate (a) -- ++(0,-\cubey,0) coordinate (b) edge coordinate [pos=1] (g) ++(0,0,-\cubez)  -- ++(\cubex,0,0) coordinate (c) -- cycle
	(o) -- ++(0,0,-\cubez) coordinate (d) -- ++(0,-\cubey,0) coordinate (e) edge (g) -- (c) -- cycle
	(o) -- (a) -- ++(0,0,-\cubez) coordinate (f) edge (g) -- (d) -- cycle;
	\begin{scope}[xshift=4cm]
	\pgfmathsetmacro{\cubex}{4}
	\pgfmathsetmacro{\cubey}{0.6}
	\pgfmathsetmacro{\cubez}{0.6}
	\draw [draw=red, every edge/.append style={draw=red, densely dashed, opacity=.5}, fill=red!30!white]
	(0,0,0) coordinate (o) -- ++(-\cubex,0,0) coordinate (a) -- ++(0,-\cubey,0) coordinate (b) edge coordinate [pos=1] (g) ++(0,0,-\cubez)  -- ++(\cubex,0,0) coordinate (c) -- cycle
	(o) -- ++(0,0,-\cubez) coordinate (d) -- ++(0,-\cubey,0) coordinate (e) edge (g) -- (c) -- cycle
	(o) -- (a) -- ++(0,0,-\cubez) coordinate (f) edge (g) -- (d) -- cycle;
	\end{scope}
	\draw[<-,out=80,in=170] (0.2,0.4) to (1,0.6);
	\node at (1.65,0.6) {\scriptsize Interface};
	\end{tikzpicture}
\end{center}
\caption{When two crystals with different values of the topological invariant are joined together, a protected edge mode exists at the interface.} \label{zakf}	
\end{figure}
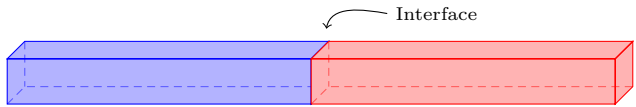
	
	In \cite{ammari2019topological}, the bulk properties of an infinitely periodic chain of subwavelength resonator dimers are studied. Using Floquet-Bloch theory, the resonant frequencies and associated eigenmodes of this crystal are derived, and further a non-trivial band gap is proved. The analogous Zak phase takes different values for different geometries and in the \emph{dilute regime} (that is, when the distance between the resonators is an order of magnitude greater than their size) explicit expressions for its value are given. Guided by this knowledge of how the infinite (bulk) chains behave, a finite chain of resonator dimers that has a topologically protected edge mode is designed. This configuration takes inspiration from the bulk-boundary correspondence in the well-known Su-Schrieffer-Heeger (SSH) model \cite{SSH}  by introducing an interface, on either side of which the resonator dimers can be associated with different Zak phases thus creating a topologically protected edge mode.

In order to present the main results obtained in \cite{ammari2019topological},  we first briefly review the topological nature of the Bloch eigenbundle. Observe that the Brillouin zone $Y^*$ has the topology of a circle. A natural question to ask, when considering the topological properties of a crystal, is whether properties are preserved after parallel transport around $Y^*$. In particular, a powerful quantity to study is the \textit{Berry-Simon connection} $A_n$, defined as
	$$A_n(\alpha) := \iu \int_D  u_n^\alpha \frac{\p}{\p \alpha} \overline{u_n^\alpha}\; \dx x.$$
	For any $\alpha_1,\alpha_2\in Y^*$, the parallel transport from $\alpha_1$ to $\alpha_2$ is $u_n^{\alpha_1}\mapsto e^{\iu \theta}u_n^{\alpha_2}$, where $\theta$ is given by
	\begin{equation*}
	\theta = \int_{\alpha_1}^{\alpha_2} A_n(\alpha) \; \dx \alpha.
	\end{equation*}
	Thus, it is enlightening to introduce the so-called \textit{Zak phase}, $\varphi_n^z$, defined as
	$$\varphi_n^{z} := \iu \int_{Y^*} \int_D u_n^\alpha \frac{\p }{\p \alpha} \overline{u_n^\alpha} \; \dx x \, \dx \alpha,$$
	which corresponds to parallel transport around the whole of $Y^*$. When $\varphi_n^z$ takes a value that is not a multiple of $2\pi$, we see that the eigenmode has gained a non-zero phase after parallel transport around the circular domain $Y^*$. In this way, the Zak phase captures topological properties of the crystal. For crystals with inversion symmetry, the Zak phase is known to only attain the values $0$ or $\pi$ \cite{zak}.

	Next, we study a periodic arrangement of subwavelength resonator dimers. This is an analogue of the SSH model. The goal is to derive a topological invariant which characterises the crystal's wave propagation properties and indicates when it supports topologically protected edge modes. 
	
	Assume we have a one-dimensional crystal in $\R^3$ with repeating unit cell $Y := [-\frac{L}{2}, \frac{L}{2}]\times \R^2$. Each unit cell contains a dimer surrounded by some background medium. Suppose the resonators together occupy the domain $D := D_1 \cup D_2$. We need two assumptions of symmetry for the analysis that follows. The first is that each individual resonator is symmetric in the sense that there exists some $x_1\in\mathbb{R}$ such that
	\begin{equation} \label{resonator_symmetry}
	R_1D_1 = D_1, \quad R_2D_2 = D_2,
	\end{equation}
	where $R_1$ and $R_2$ are the reflections in the planes $p_1=\{-x_1\}\times \R^2$ and $p_2=\{x_1\}\times \R^2$, respectively. We also assume that the dimer is symmetric in the sense that
	\begin{equation} \label{dimer_symmetry}
	D=-D.
	\end{equation}
	 Denote the full crystal by $\Crystal$, that is,
	\begin{equation} \label{crystal_def}
	\Crystal := \bigcup_{m\in \Z} \left(D + (mL,0,0)\right).
	\end{equation}
	We denote the separation of the resonators within each unit cell, along the first coordinate axis, by $d := 2x_1$ and the separation across the boundary of the unit cell by $d' := L - d$. See Figure \ref{fig:SSH}.
	
	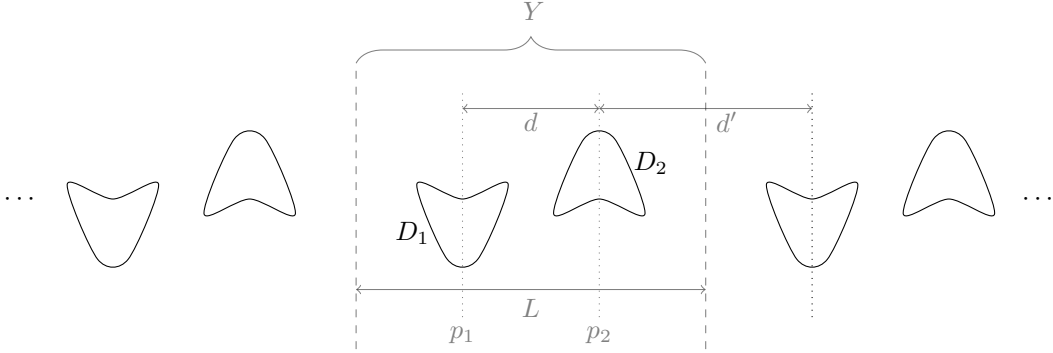
\begin{figure}[tbh]
		\centering
		\begin{tikzpicture}[scale=2]
		\begin{scope}
		\draw[dashed, opacity=0.5] (-0.5,0.85) -- (-0.5,-1);
		\draw[dashed, opacity=0.5]  (1.8,0.85) -- (1.8,-1)node[yshift=4pt,xshift=-7pt]{};
		\draw[{<[scale=1.5]}-{>[scale=1.5]}, opacity=0.5] (-0.5,-0.6) -- (1.8,-0.6)  node[pos=0.5, yshift=-7pt,]{$L$};
		\draw  plot [smooth cycle] coordinates {(-0.1,0.1) (0.2,0) (0.5,0.1) (0.3,-0.4) (0.1,-0.4)} node[xshift=-13pt, yshift=10pt]{$D_1$};
		\draw  plot [smooth cycle] coordinates {(0.8,-0.1) (1.1,0) (1.4,-0.1) (1.2,0.4) (1,0.4)} node[xshift=25pt, yshift=-9pt]{$D_2$};
		\draw[{<[scale=1.5]}-{>[scale=1.5]}, opacity=0.5] (0.2,0.6) -- (1.1,0.6) node[pos=0.5, yshift=-5pt,]{$d$};
		\draw[dotted,opacity=0.5] (0.2,0.7) -- (0.2,-0.8) node[at end, yshift=-0.2cm]{$p_1$};
		\draw[dotted,opacity=0.5] (1.1,0.7) -- (1.1,-0.8) node[at end, yshift=-0.2cm]{$p_2$};
		\draw[{<[scale=1.5]}-{>[scale=1.5]}, opacity=0.5] (1.1,0.6) -- (2.5,0.6) node[pos=0.6, yshift=-5pt,]{$d'$};
		\end{scope}
		\begin{scope}[xshift=-2.3cm]
		\draw  plot [smooth cycle] coordinates {(-0.1,0.1) (0.2,0) (0.5,0.1) (0.3,-0.4) (0.1,-0.4)};
		\draw  plot [smooth cycle] coordinates {(0.8,-0.1) (1.1,0) (1.4,-0.1) (1.2,0.4) (1,0.4)};
		\begin{scope}[xshift = 1.2cm]
		\draw (-1.6,0) node{$\cdots$};
		\end{scope};
		\end{scope}
		\begin{scope}[xshift=2.3cm]
		\draw  plot [smooth cycle] coordinates {(-0.1,0.1) (0.2,0) (0.5,0.1) (0.3,-0.4) (0.1,-0.4)};
		\draw  plot [smooth cycle] coordinates {(0.8,-0.1) (1.1,0) (1.4,-0.1) (1.2,0.4) (1,0.4)};
		\draw[dotted] (0.2,0.7) -- (0.2,-0.8);
		\begin{scope}[xshift = 1.1cm]
		\end{scope}
		\draw (1.7,0) node{$\cdots$};
		\end{scope}		
		\begin{scope}[yshift=0.9cm]
		\draw [decorate,opacity=0.5,decoration={brace,amplitude=10pt}]
		(-0.5,0) -- (1.8,0) node [black,midway]{};
		\node[opacity=0.5] at (0.67,0.35) {$Y$};	
		\end{scope}
		\end{tikzpicture}
		\caption{Example of a two-dimensional cross-section of a chain of subwavelength resonators satisfying the symmetry assumptions \eqref{resonator_symmetry}~and~\eqref{dimer_symmetry}. The repeating unit cell $Y$ contains the dimer $D_1 \cup D_2$.} \label{fig:SSH}
	\end{figure}	
	
	Wave propagation inside the infinite periodic structure is modelled by the Helmholtz problem
	\begin{equation} \label{eq:scatteringt}
	\left\{
	\begin{array} {ll}
	\ds \Delta {u}+ \omega^2 {u}  = 0 & \text{in } \R^3 \setminus \p \Crystal, \\
	\nm
	\ds  {u}|_{+} -{u}|_{-}  =0  & \text{on } \partial \Crystal, \\
	\nm
	\ds  \delta \frac{\partial {u}}{\partial \nu} \bigg|_{+} - \frac{\partial {u}}{\partial \nu} \bigg|_{-} =0 & \text{on } \partial \Crystal, \\
	\nm
	\ds u(x_1,x_2,x_3) & \text{satisfies the outgoing radiation condition as } \sqrt{x_2^2+x_3^2} \rightarrow \infty.
	\end{array}
	\right.
	\end{equation}
By applying the Floquet transform, the Bloch eigenmode $u_\alpha(x) := \mathcal{U}[u](x,\alpha)$ is the solution to the Helmholtz problem
	\begin{equation} \label{eq:scattering_quasi}
	\left\{
	\begin{array} {ll}
	\ds \Delta u_\alpha+ \omega^2 {u_\alpha}  = 0 &\text{in } \R^3 \setminus \p \Crystal, \\
	\nm
	\ds  {u_\alpha}|_{+} -{u_\alpha}|_{-}  =0  & \text{on } \partial \Crystal, \\
	\nm
	\ds  \delta \frac{\partial {u_\alpha}}{\partial \nu} \bigg|_{+} - \frac{\partial {u_\alpha}}{\partial \nu} \bigg|_{-} =0 & \text{on } \partial \Crystal, \\
	\nm
	\ds e^{-\iu  \alpha_1 x_1}  u_\alpha(x_1,x_2,x_3)  \,\,\,&  \mbox{is periodic in } x_1, \\
	\nm
	 \ds u_\alpha(x_1,x_2,x_3)& \text{satisfies the $\alpha$-quasi-periodic outgoing radiation condition} \\ &\hspace{0.5cm} \text{as } \sqrt{x_2^2+x_3^2} \rightarrow \infty.
	\end{array}
	\right.
	\end{equation} 
	 We  formulate the quasi-periodic resonance problem as an integral equation. 
	Let $\mathcal{S}_{D}^{\alpha,\omega}$ be the single-layer potential associated to the three-dimensional Green's function  which is quasi-periodic in one dimension,
	$$G^{\alpha,k}(x,y) := -\sum_{m \in \Z} \frac{e^{\iu k|x-y-(Lm,0,0)|}}{4\pi|x-y-(Lm,0,0)|}e^{\iu \alpha Lm}.$$	
 	The solution $u_\alpha$ of \eqref{eq:scattering_quasi} can be represented as
	\begin{equation*} \label{eq:helm-solution_quasi}
	u_\alpha = \mathcal{S}_{D}^{\alpha,\omega} [\Psi^\alpha],
	\end{equation*}
	for some density $\Psi^\alpha \in  L^2(\p D)$. Then, using the jump relations, it can be shown that~\eqref{eq:scattering_quasi} is equivalent to the boundary integral equation
	\begin{equation}  \label{eq:boundary_quasi}
	\mathcal{A}^\alpha(\omega, \delta)[\Psi^\alpha] =0,  
	\end{equation}
	where
	\begin{equation} \label{eq:A_quasi_defn}
	\mathcal{A}^\alpha(\omega, \delta) := -\lambda I + \left(\mathcal{K}_D^{ -\alpha,\omega}\right)^*, \quad \lambda := \frac{1+\delta}{2(1-\delta)}.
	\end{equation}

	Let $V_j^\alpha$ be the solution to 
	\begin{equation} \label{eq:V_quasi}
	\begin{cases}
	\ds \Delta V_j^\alpha =0 \quad &\mbox{in } \quad Y\setminus \overline{D},\\
	\ds V_j^\alpha = \delta_{ij} \quad &\mbox{on } \quad \partial D_i,\\
	\ds V_j^\alpha(x+(mL,0,0))= e^{\iu \alpha m} V_j^\alpha(x) & \forall m \in \Z, \\
	\ds V_j^\alpha(x_1,x_2,x_3) = O\left(\tfrac{1}{\sqrt{x_2^2+x_3^2}}\right) \quad &\text{as } \sqrt{x_2^2+x_2^2}\to\infty, \text{ uniformly in } x_1,
	\end{cases}
	\end{equation}
	where $\delta_{ij}$ is the Kronecker delta.
	Analogously to \eqref{defcapal}, we then define the quasi-periodic capacitance matrix $C^\alpha=(C_{ij}^\alpha)$ by
	\begin{equation} \label{eq:qp_capacitance} 
	C_{ij}^\alpha := \int_{Y\setminus \overline{D} }\overline{\nabla V_i^\alpha}\cdot\nabla V_j^\alpha  \; \dx x,\quad i,j=1, 2.
	\end{equation}
	Finding the eigenpairs of this matrix represents a leading order approximation to the differential problem \eqref{eq:scattering_quasi}. The following  properties of $C^\alpha$ are useful.

	\begin{lemma} \label{lem:quasi_matrix_form}
		The matrix $C^\alpha$ is Hermitian with constant diagonal, \ie{},
		$$C_{11}^\alpha = C_{22}^\alpha \in \R, \quad C_{12}^\alpha = \overline{C_{21}^\alpha} \in \mathbb{C}.$$
	\end{lemma}

Since $C^\alpha$ is Hermitian, the following lemma follows directly.
	\begin{lemma} \label{lem:evec}
		The eigenvalues and corresponding eigenvectors of the quasi-periodic capacitance matrix are given by 
		\begin{align*}
		\lambda_1^\alpha &= C_{11}^\alpha - \left|C_{12}^\alpha \right|, \qquad
		\begin{pmatrix}
		a_1  \\ b_1
		\end{pmatrix} = \frac{1}{\sqrt{2}}\begin{pmatrix}
		- e^{\iu \theta_\alpha}  \\ 1
		\end{pmatrix}, \\
		\lambda_2^\alpha &= C_{11}^\alpha + \left|C_{12}^\alpha \right|, \qquad
		\begin{pmatrix}
		a_2  \\ b_2
		\end{pmatrix} = \frac{1}{\sqrt{2}}\begin{pmatrix}
		e^{\iu \theta_\alpha}  \\ 1
		\end{pmatrix},
		\end{align*}
		where, for $\alpha$ such that $C_{12}^\alpha\neq0$, $\theta_\alpha\in[0,2\pi)$ is defined to be such that
		\begin{equation}
			e^{\iu \theta_\alpha} = \frac{C_{12}^\alpha}{|C_{12}^\alpha|}.
		\end{equation}
	\end{lemma}
	
	In the dilute regime, we are able to compute asymptotic expansions for the band structure and topological properties. In this regime, we assume that the resonators can be obtained by rescaling fixed domains $B_1, B_2$ as follows:
	\begin{equation}\label{eq:dilute}
	D_1=\epsilon B_1 - \left(\frac{d}{2},0,0\right), \quad  D_2=\epsilon B_2 + \left(\frac{d}{2},0,0\right),
	\end{equation}
	for some small parameter $\epsilon > 0$. 
	
	Let $\textrm{Cap}_{B}$ denote the capacity of  $B = B_i$ for $i=1$ or $i=2$ (see \eqref{defcap} for the definition of the capacity). Due to symmetry, the capacitance is the same for the two choices $i =1, 2$. It is easy to see that, by a scaling argument, 
	\begin{equation}\label{eq:cap_scale}
	\textrm{Cap}_{\epsilon B} = \epsilon \textrm{Cap}_B.
	\end{equation}
	
	\begin{lemma}\label{lem:cap_estim_quasi}
		We assume that the resonators are in the dilute regime specified by \eqref{eq:dilute}. We also assume that $\alpha \neq 0$ is fixed. Then we have the following asymptotics of the capacitance matrix $C_{ij}^\alpha$ as $\epsilon\rightarrow 0$:
		\begin{align}
		C_{11}^\alpha &= \epsilon \mathrm{Cap}_B - \frac{(\epsilon \mathrm{Cap}_B)^2}{4\pi}\sum_{m \neq 0}  \frac{e^{\iu m \alpha L}}{  |mL| } + \O(\epsilon^3), \label{eq:c1q}
		\\
		C_{12}^\alpha &= -\frac{(\epsilon \mathrm{Cap}_B)^2}{4\pi}\sum_{m =-\infty}^\infty \frac{e^{\iu m \alpha L} }{  |mL + d| } + \O(\epsilon^3). \label{eq:c2q}
		\end{align}
		Taking the imaginary part of \eqref{eq:c2q}, the corresponding asymptotic formula holds uniformly in $\alpha \in Y^*$.
	\end{lemma}

	Define normalized extensions of $V_j^\alpha$ as
	$$S_j^\alpha(x) := \begin{cases} \frac{1}{\sqrt{|D_1|}}\delta_{ij} \quad &x \in D_i, \ i=1,2, \\ \nm 
	\frac{1}{\sqrt{|D_1|}}V_j^\alpha(x) \quad &x \in Y\setminus \overline{D}, \end{cases}$$ where $|D_1|$ is the volume of one of the resonators ($|D_1|=|D_2|$ thanks to the dimer's symmetry \eqref{dimer_symmetry}). The following two approximation results hold. 
	
	\begin{theorem} \label{thm:char_approx_infinite}
		The characteristic values $\omega_j^\alpha=\omega_j^\alpha(\delta),~j=1,2$, of the operator $\mathcal{A}^{\alpha}(\omega,\delta)$, defined in \eqref{eq:A_quasi_defn}, can be approximated as
		$$ \omega_j^\alpha= \sqrt{\frac{\delta \lambda_j^\alpha }{|D_1|}}  + \O(\delta),$$
		where $\lambda_j^\alpha,~j=1,2$, are eigenvalues of the quasi-periodic capacitance matrix $C^\alpha$.
	\end{theorem}

	\begin{theorem} \label{thm:mode_approx}
		The Bloch eigenmodes $u_j^\alpha,~j=1,2$, corresponding to the resonances $\omega_j^\alpha$, can be approximated as
		$$ u_j^\alpha(x) = a_j S^\alpha_1(x) + b_j S^\alpha_2(x) + \O(\delta),$$
		where $\left(\begin{smallmatrix}a_j\\b_j\end{smallmatrix}\right),~j=1,2,$
		 are the eigenvectors of the quasi-periodic capacitance matrix $C^\alpha$, as given by \Cref{lem:evec}.
	\end{theorem}
	
Theorems~\ref{thm:char_approx_infinite} and \ref{thm:mode_approx} show that the capacitance matrix can be considered to be a discrete approximation of the differential problem \eqref{eq:scattering_quasi}, since its eigenpairs directly determine the resonant frequencies and the Bloch eigenmodes (at leading order in $\delta$).

	\begin{figure}[tb]
	\centering
	\begin{tikzpicture}[scale=2]
	\begin{scope}
	\draw[opacity=0.5] (-0.5,1) -- (-0.5,-1);
	\draw[opacity=0.5, dashed] (0.65,1) -- (0.65,-1);
	\draw[opacity=0.5, dotted] (1.225,0.9) -- (1.225,-0.9)node[right]{$q_2$};
	\draw[opacity=0.5, dotted] (0.075,0.9) -- (0.075,-0.9)node[left]{$q_1$};
	\draw[opacity=0.5, dotted] (1.1,0.9)  node[left]{$p_2$} -- (1.1,-0.9);
	\draw[opacity=0.5, dotted] (0.2,0.9)  node[right]{$p_1$} -- (0.2,-0.9);
	\draw[opacity=0.5]  (1.8,1) -- (1.8,-1);
	\draw  plot [smooth cycle] coordinates {(-0.1,0.1) (0.2,0) (0.5,0.1) (0.3,-0.4) (0.1,-0.4)} node[xshift=-13pt, yshift=10pt]{$D_1$};
	\draw  plot [smooth cycle] coordinates {(0.8,-0.1) (1.1,0) (1.4,-0.1) (1.2,0.4) (1,0.4)} node[xshift=25pt, yshift=-9pt]{$D_2$};
	\draw[<->, opacity=0.5] (0.2,0.6) -- (1.1,0.6) node[pos=0.65, yshift=-5pt,]{$d$};
	\end{scope}
	\draw (2.65,0) node{$\xrightarrow[\Ro_2]{\Ro_1}$};
	\begin{scope}[xshift=4cm]
	\draw[opacity=0.5] (-0.5,1) -- (-0.5,-1);
	\draw[opacity=0.5]  (1.8,1) -- (1.8,-1);
	\draw  plot [smooth cycle] coordinates {(-0.35,0.1) (-0.05,0) (0.25,0.1) (0.05,-0.4) (-0.15,-0.4)} node[xshift=25pt, yshift=10pt]{$D_1'$};
	\draw  plot [smooth cycle] coordinates {(1.05,-0.1) (1.35,0) (1.65,-0.1) (1.45,0.4) (1.25,0.4)} node[xshift=-15pt, yshift=-9pt]{$D_2'$};
	\draw[<->, opacity=0.5] (-0.05,0.6)  -- (1.35,0.6) node[pos=0.5, yshift=-5pt]{$d'$};
	\draw[opacity=0.5, dotted] (1.35,0.9)  node[left]{$p_2'$} -- (1.35,-0.9);
	\draw[opacity=0.5, dotted] (-0.05,0.9)  node[right]{$p_1'$} -- (-0.05,-0.9);
	\end{scope}
	\begin{scope}[yshift=1.05cm]
	\draw [decorate,opacity=0.5,decoration={brace,amplitude=10pt}]
	(-0.5,0) -- (0.64,0) node [black,midway]{};
	\draw [decorate,opacity=0.5,decoration={brace,amplitude=10pt}]
	(0.66,0) -- (1.8,0) node [black,midway]{};
	\node[opacity=0.5] at (0.1,0.3) {$Y_1$};
	\node[opacity=0.5] at (1.26,0.3) {$Y_2$};
	\end{scope}
	\begin{scope}[xshift=4cm,yshift=1.05cm]
	\draw [decorate,opacity=0.5,decoration={brace,amplitude=10pt}]
	(-0.5,0) -- (1.8,0) node [black,midway]{};
	\node[opacity=0.5] at (0.67,0.3) {$Y'$};	
	\end{scope}
	\end{tikzpicture}
	\caption{Reflections taking $D$ to $D'$.} \label{fig:YY'}
\end{figure}
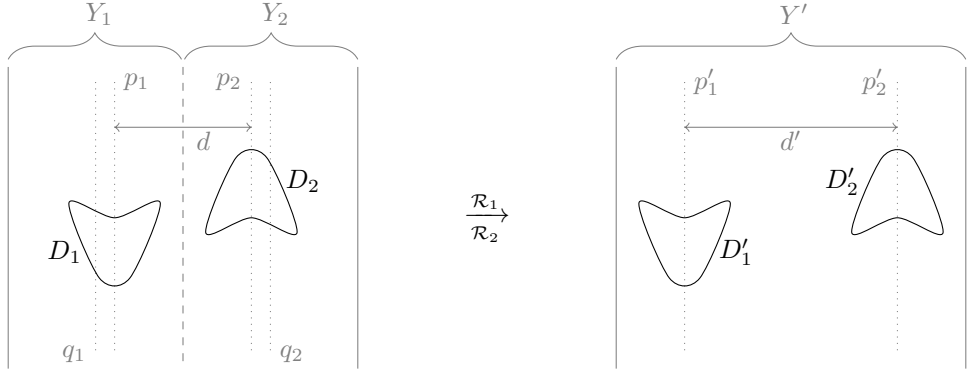

We now introduce notation which, thanks to the assumed symmetry of the resonators, will allow us to prove topological properties of the chain. Divide $Y$ into two subsets $Y=Y_1\cup Y_2$, where $Y_1 := [-\frac{L}{2},0]\times \R^2$ and let $Y_2 := [0,\frac{L}{2}]\times \R^2$, as depicted in \Cref{fig:YY'}. Define $q_1$ and $q_2$ to be the central planes of $Y_1$ and $Y_2$, that is, the planes $q_1 := \{ -\frac{L}{4}\} \times \R^2$ and $q_2 := \{ \frac{L}{4}\} \times \R^2$. Let $\Ro_1$ and $\Ro_2$ be reflections in the respective planes. Observe that, thanks to the assumed symmetry of each resonator \eqref{resonator_symmetry}, the ``complementary'' dimer $D' = D_1' \cup D_2'$, given by swapping $d$ and $d'$, satisfies $D_i' = \Ro_i D_i$ for $i=1,2$.
Define the operator $T_\alpha$ on the set of $\alpha$-quasi-periodic functions $f$ on $Y$ as
$$T_\alpha f(x) := \begin{cases} e^{-\iu \alpha L}\overline{f(\Ro_1x)}, \quad &x\in Y_1, \\ \overline{f(\Ro_2x)}, &x\in Y_2, \end{cases}$$
where the factor $e^{-\iu \alpha L}$ is chosen so that the image of a continuous ($\alpha$-quasi-periodic) function is continuous.

We now proceed to use $T_\alpha$ to analyse the different topological properties of the two dimer configurations. Define the quantity ${C_{12}^{\alpha}}'$ analogously to $C_{12}^\alpha$ but on the dimer $D'$, that is, to be the top-right element of the corresponding quasi-periodic capacitance matrix, defined in \eqref{eq:qp_capacitance}.

\begin{lemma}\label{lem:cc'}
	We have
	\begin{equation*}\label{eq:cc'}
		{C_{12}^{\alpha}}' = e^{-\iu \alpha L} \overline{C_{12}^\alpha}.\end{equation*}
	Consequently, if $d = d' = \frac{L}{2}$ then $C_{12}^{\pi/L} = 0$. 
\end{lemma}

\begin{lemma}\label{lem:c=0}
	We assume that $D$ is in the dilute regime specified by \eqref{eq:dilute}. Then, for $\epsilon$ small enough,
	\begin{itemize}
		\item[(i)] $\mathrm{Im}\ C_{12}^\alpha > 0$ for $0<\alpha<\pi/L$ and $\mathrm{Im}\ C_{12}^\alpha < 0$ for $-\pi/L<\alpha<0$. In particular, $\mathrm{Im}\ {C_{12}^{\alpha}}$ is zero if and only if $\alpha \in\{ 0, \pi/L \}$.
		\item [(ii)] $C_{12}^{\alpha}$ is zero if and only if both $d = d'$ and $\alpha = \pi/L$.
		\item [(iii)] $C_{12}^{\pi/L} < 0$  when $d<d'$ and $C_{12}^{\pi/L} > 0$ when $d>d'$. In both cases we have $C_{12}^{0} < 0$.
	\end{itemize}
\end{lemma}
 This lemma describes the crucial properties of the behaviour of the curve $\{C_{12}^\alpha:\alpha\in Y^*\}$ in the complex plane. The periodic nature of $Y^*$ means that this is a closed curve. Part (i) tells us that this curve crosses the real axis in precisely two points. Taken together with (iii), we know that this curve winds around the origin in the case $d>d'$, but not in the case $d<d'$. The following band gap result is from \cite{ammari2020robust}.
\begin{theorem} \label{thm:band_gap}
 If $d\neq d'$, the first and second bands form a band gap:
	$$\max_{\alpha \in Y^*} \omega_1^\alpha < \min_{\alpha \in Y^*} \omega_2^\alpha,$$
	for small enough $\epsilon$ and $\delta$.
\end{theorem}

Combining the above results, we obtain the following result concerning the band inversion that takes place between the two geometric regimes $d<d'$ and $d>d'$ as illustrated in Figure \ref{bandinvf}.
\begin{proposition} \label{prop:bandinv}
	For $\epsilon$ small enough, the band structure at $\alpha = \pi/L$ is inverted between the $d<d'$ and $d>d'$ regimes. In other words, the eigenfunctions associated with the first and second bands at $\alpha = \pi/L$ are given, respectively, by
	\begin{equation*}
	u_1^{\pi/L}(x) = S_1^{\pi/L}(x)+S_2^{\pi/L}(x)+ \O(\delta), \quad u_2^{\pi/L}(x) = S_1^{\pi/L}(x)-S_2^{\pi/L}(x)+ \O(\delta), 
	\end{equation*}
	when $d<d'$ and by
	\begin{equation*}
	u_1^{\pi/L}(x) = S_1^{\pi/L}(x)-S_2^{\pi/L}(x)+ \O(\delta), \quad u_2^{\pi/L}(x) = S_1^{\pi/L}(x)+S_2^{\pi/L}(x)+ \O(\delta), 
	\end{equation*}
	when $d>d'$.
\end{proposition}

\begin{figure}[h]
\centering
\begin{tikzpicture}[scale=1.4]
\draw [decorate,decoration={brace,amplitude=10pt}]
(-0.5,0) -- (3.5,0) node [red,midway]{};
\node at (1.5,0.6) {$d<d'$};	
\begin{scope}[xshift=5.3cm]
\draw [decorate,decoration={brace,amplitude=10pt}]
(-0.5,0) -- (3.5,0) node [red,midway]{};
\node at (1.5,0.6) {$d>d'$};
\end{scope}
\node at (0.1,-1.8) {\includegraphics[trim={0.65cm 0 2.6cm 0},clip,height=3.5cm]{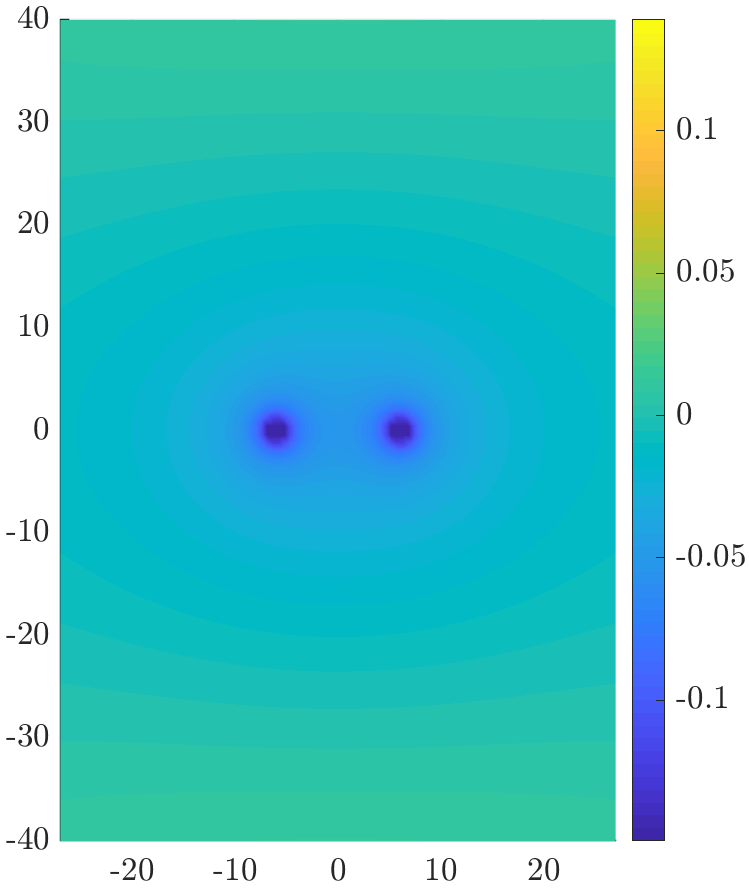}};
\node at (2.6,-1.8) {\includegraphics[trim={0.65cm  0 2.6cm 0},clip,height=3.5cm]{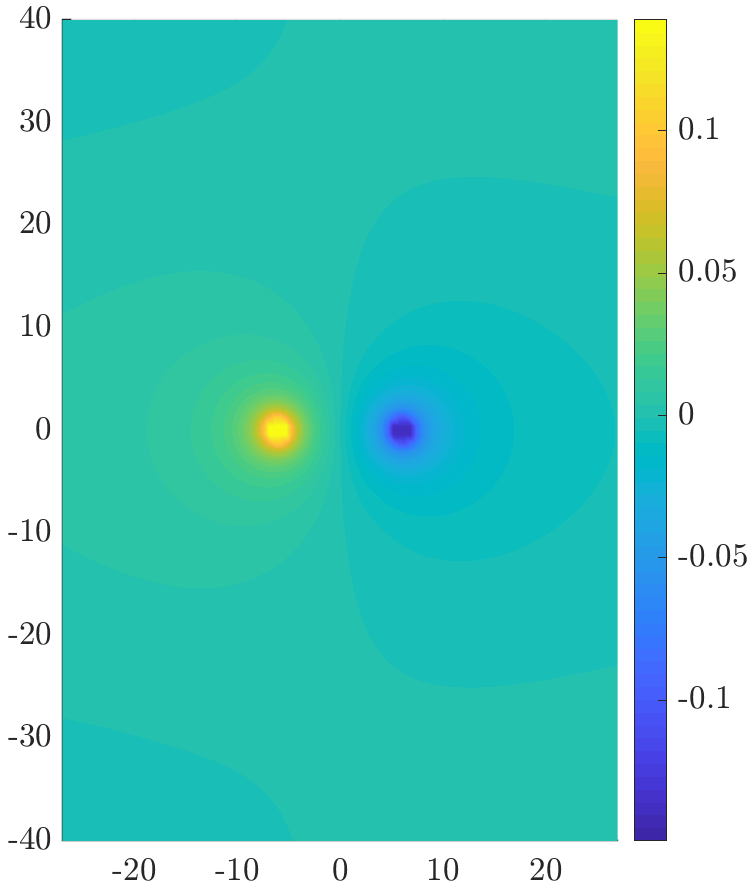}};
\begin{scope}[xshift=5.4cm]
\node at (0,-1.8) {\includegraphics[trim={0.65cm 0 2.6cm 0},clip,height=3.5cm]{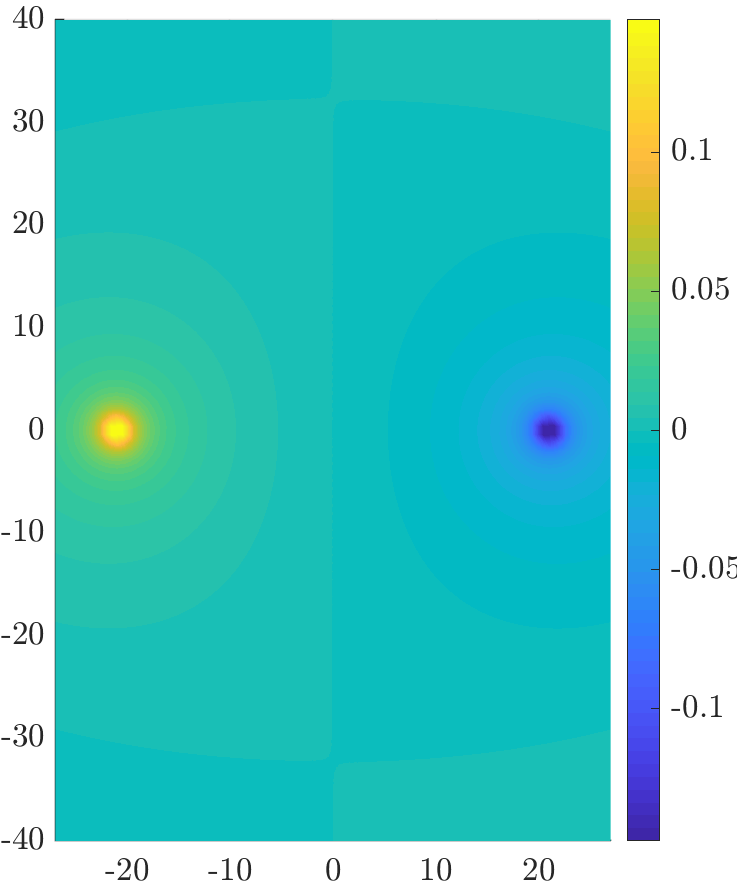}};
\node at (2.8,-1.8) {\includegraphics[trim={0.65cm  0 0 0},clip,height=3.5cm]{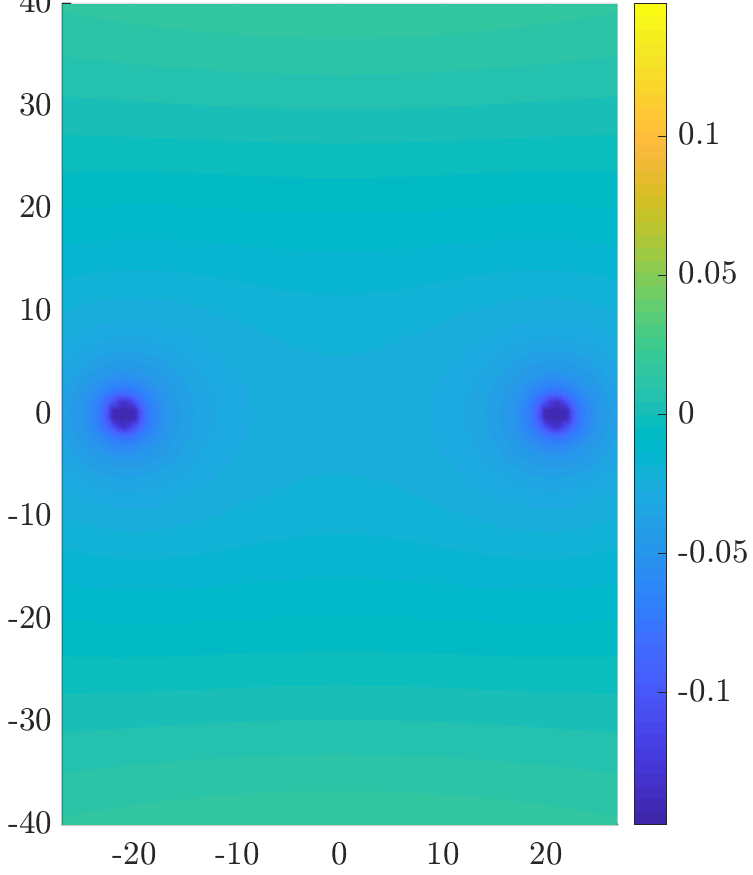}};
\end{scope}
\node at (0.9,-0.4) {\color{white}\small $u_1^{\pi/L}$};
\node at (3.4,-0.4) {\color{white}\small $u_2^{\pi/L}$};
\begin{scope}[xshift=5.3cm]
\node at (0.9,-0.4) {\color{white}\small $u_1^{\pi/L}$};
\node at (3.4,-0.4) {\color{white}\small $u_2^{\pi/L}$};
\end{scope}
\end{tikzpicture}
\caption{Band inversion: the monopole/dipole natures of the 1\textsuperscript{st} and 2\textsuperscript{nd} eigenmodes have swapped between the $d<d'$ and $d>d'$ regimes.}
\label{bandinvf}
\end{figure}
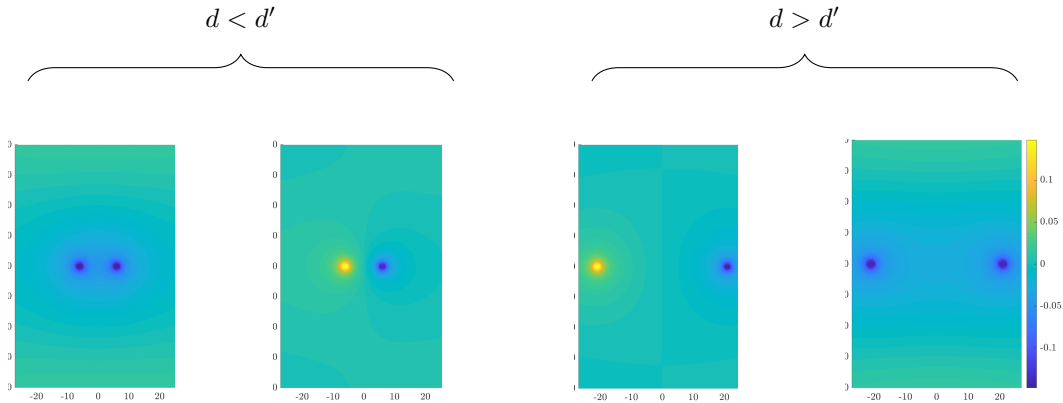

The eigenmode $S_1^{\pi/L}(x)+S_2^{\pi/L}(x)$ is constant and attains the same value on both resonators, while the eigenmode $S_1^{\pi/L}(x)-S_2^{\pi/L}(x)$ has values of opposite sign on the two resonators. They therefore correspond, respectively, to monopole and dipole modes, and \Cref{prop:bandinv} shows that the monopole/dipole nature of the first two Bloch eigenmodes are swapped between the two regimes. We will now proceed to define a topological invariant which we will use to characterise the topology of a chain and prove how its value depends on the relative sizes of $d$ and $d'$. This invariant is intimately connected with the band inversion phenomenon and is non-trivial only if $d>d'$ \cite{ammari2019topological}.

\begin{theorem} \label{thm:phase}
	We assume that $D$ is in the dilute regime specified by \eqref{eq:dilute}. Then the Zak phase $\varphi_j^z, j = 1,2$, defined by 
	$$\varphi_j^z := \iu \int_{Y^*} \int_D  u_j^\alpha  \frac{\p }{\p \alpha} \overline{u_j^\alpha} \; \dx x \, \dx \alpha,$$ 
	satisfies
	$$ \varphi_j^z = \begin{cases}
	0, \quad &\text{if} \ \ d < d', \\
	\pi, \quad &\text{if} \ \ d > d',
	\end{cases}$$
	for $\epsilon$ and $\delta$ small enough.
\end{theorem}

Theorem \ref{thm:phase} shows that the Zak phase of the crystal is non-zero precisely when $d > d'$. The bulk-boundary correspondence suggests that we can create topologically protected subwavelength edge modes by joining half-space subwavelength crystals, one with $\varphi_j^z = 0$ and the other with $\varphi_j^z = \pi$.

\begin{remark}
A second approach to creating chains with robust subwavelength localized modes is to 
start with  a one-dimensional array of pairs of subwavelength resonators that exhibits a {subwavelength band gap}.  We then introduce a defect by adding a dislocation within one of the resonator pairs.  As shown in \cite{ammari2020robust}, as a result of this dislocation, mid-gap frequencies enter the band gap from either side and converge to a single frequency, within the band gap, as the dislocation becomes arbitrarily large.
Such frequency can place localized modes at any point within the band gap and corresponds to a robust {edge modes}. 
\end{remark}

\section{Mimicking the cochlea with an array of graded subwavelength resonators}
\label{sec7}

In \cite{davies2019fully} an array of subwavelength resonators is used to design a to-scale artificial cochlea that mimics the first stage of human auditory processing and present a rigorous analysis of its properties. In order to replicate the spatial frequency separation of the cochlea, the array should have a size gradient, meaning each resonator is slightly larger than the previous, as depicted in Fig.~\ref{fig:geom}. The size gradient is chosen so that the resonator array mimics the spatial frequency separation performed by the cochlea. In particular, the structure can reproduce the well-known (tonotopic) relationship between incident frequency and position of maximum excitation in the cochlea. This is a consequence of the asymmetry of the eigenmodes $u_n(x)$, see \cite{davies2019fully} and \cite{davies2020hopf} for details.

 \begin{figure}[h]
	\centering
	\includegraphics[width=11.4cm]{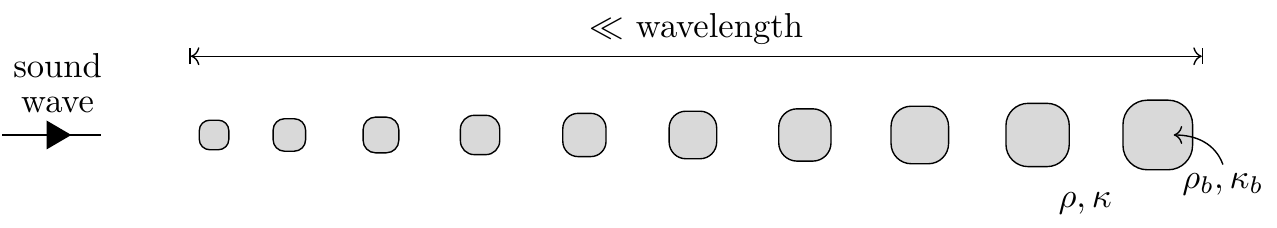}
	\caption{A graded array of subwavelength resonators mimics the biomechanical properties of the cochlea in response to a sound wave. } \label{fig:geom}
\end{figure}

Such graded arrays of subwavelength resonators can mimic the biomechanical properties of the cochlea, at the same scale. In \cite{perception}, a modal time-domain expansion for the scattered pressure field due to such a structure is derived from  first principles. It is  proposed there that these modes should form the basis of a signal processing architecture. The properties of such an approach is investigated and it is shown that higher-order gammatone filters appear by cascading. Further, an approach for extracting meaningful global properties from the coefficients, tailored to the statistical properties of so-called natural sounds is proposed. 
 
The subwavelength resonant frequencies of an array of $N=22$ resonators computed by using the formulation \eqref{eq-boundary}--\eqref{page450} are shown in Fig.~\ref{fig:spectrum}. This array measures 35~mm, has material parameters corresponding to air-filled resonators surrounded by water and has subwavelength resonant frequencies within the range 500~Hz~--~10~kHz. Thus, this structure has similar dimensions to the human cochlea, is made from realistic materials and experiences subwavelength resonance in response to frequencies that are audible to humans.

\begin{figure}[h]
	\centering
	\includegraphics[width=0.8\linewidth]{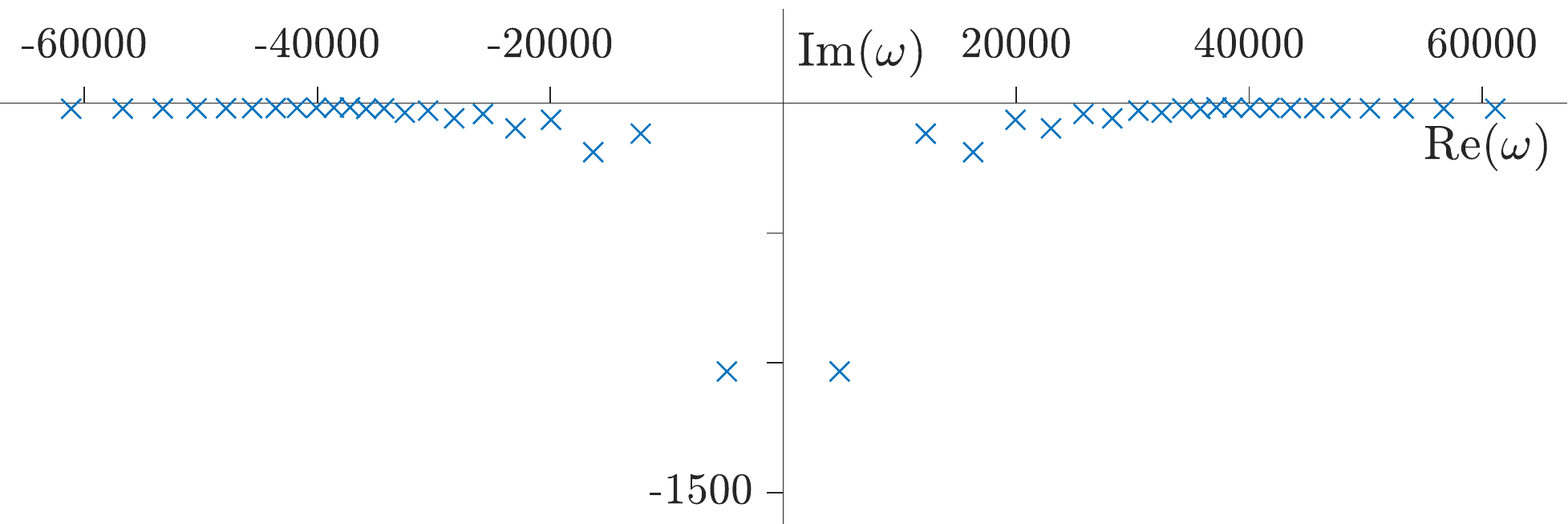}
	\caption{The resonant frequencies $\{\omega_n:n=1,\dots,N\}\subset\mathbb{C}$ lie in the right-hand complex plane, shown for an array of $N=22$ subwavelength resonators. The Helmholtz problem also has singularities in the left-hand plane, which are symmetric in the imaginary axis. The imaginary parts are all negative, due to energy losses. } \label{fig:spectrum}
\end{figure}

This analysis is useful not only for designing cochlea-like devices, but is also used in \cite{perception} as the basis for a machine hearing procedure which mimics neural processing in the auditory system. Consider the scattering of a signal, $s:[0,T]\to\mathbb{R}$, whose frequency support is wider than a single frequency and whose Fourier transform exists.  Consider the Fourier transform of the incoming pressure wave, given for $\omega\in\mathbb{C}$ and $x\in\mathbb{R}^3$ by
\begin{align*}
u^{in}(x,\omega)&=\int_{-\infty}^{\infty} s(x_1/v-t) e^{\i\omega t}\de t\\
&=e^{\i\omega x_1/v}\hat{s}(\omega) = \hat{s}(\omega)+ \O(\omega),
\end{align*}
where $\hat{s}(\omega):=\int_{-\infty}^{\infty} s(-u) e^{\i\omega u}\de u$. In \Cref{def:res^}, we defined resonant frequencies as having positive real parts. However, the scattering problem \eqref{eq:scattering} is known to be symmetric in the sense that if it has a pole at $\omega\in\mathbb{C}$ then it has a pole with the same multiplicity at $-\overline{\omega}$ \cite{dyatlov2019mathematical}. As depicted in \Cref{fig:spectrum}, this means the resonant spectrum is symmetric in the imaginary axis.

Suppose that the scattered acoustic pressure field $u$ in response to the Fourier transformed signal $\hat s$ can, for $x\in\D$, be decomposed as
\begin{equation} \label{eq:gen_modal_decomp} 
u(x,\omega)
= \sum_{n=1}^N \frac{-\hat{s}(\omega)\nu_n\Re(\omega_n^+)^2}{(\omega-\omega_n)(\omega+ \overline{\omega_n})} u_n(x) 
+ r(x,\omega),
\end{equation} 
for some remainder $r$. We are interested in signals whose energy is mostly concentrated within the subwavelength regime. In particular, we want that
\begin{equation} \label{eq:subw_regime}
\sup_{x\in\mathbb{R}^3}\int_{-\infty}^\infty |r(x,\omega) | \de\omega = \O(\delta).
\end{equation}
Then, under the assumptions \eqref{eq:gen_modal_decomp} and \eqref{eq:subw_regime}, we can apply the inverse Fourier transform \cite{perception} to find that the scattered pressure field satisfies, for $x\in\D$ and $t\in\mathbb{R}$, 
\begin{equation} \label{thm:timedom}
p(x,t)= \sum_{n=1}^N a_n[s](t) u_n(x) + \O(\delta),
\end{equation}
where the coefficients are given by the convolutions $a_n[s](t)=\left( s*h_n \right)(t)$ with the kernels
\begin{equation} \label{eq:hdef}
h_n(t)=
\begin{cases}
0, & t<0, \\
c_n e^{\Im(\omega_n)t} \sin(\Re(\omega_n)t), & t\geq0,
\end{cases}
\end{equation}
for $c_n=\nu_n\Re(\omega_n)$.
\begin{remark}
	The assumption \eqref{eq:subw_regime} is a little difficult to interpret physically. For the purposes of informing signal processing approaches, however, it is sufficient.
\end{remark}


\begin{figure}
	\centering
	\includegraphics[width=0.8\linewidth]{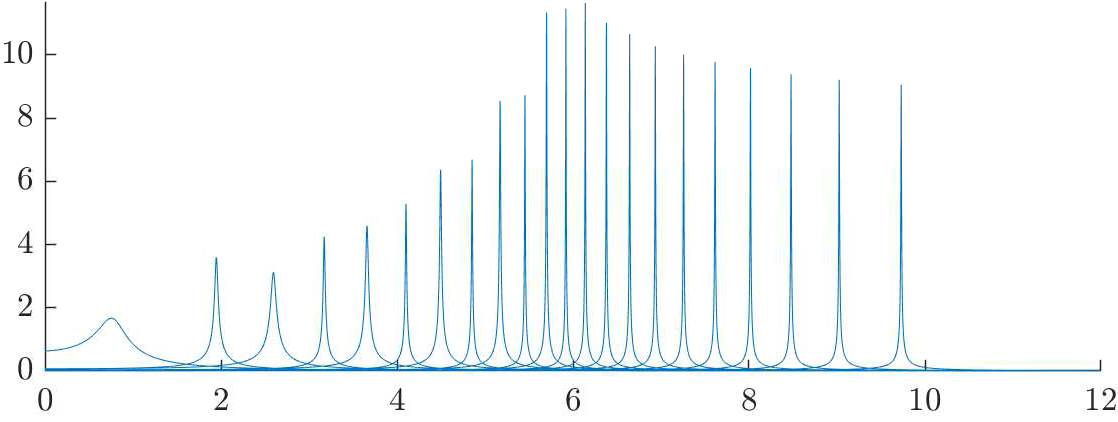}
	\caption{The frequency support of the band-pass filters $h_n$ created by an array of  22 subwavelength resonators. } \label{fig:bandpass}
\end{figure}

On the one hand, note that the fact that $h_n(t)=0$ for $t<0$ ensures the causality of the modal expansion in \eqref{thm:timedom}. 
Moreover, as shown in Figure~\ref{fig:bandpass}, $h_n$ is a windowed oscillatory mode that acts as a band-pass filter centred at $\Re(\omega_n)$. On the other hand, the asymmetry of the spatial eigenmodes $u_n(x)$ means that the decomposition from \eqref{thm:timedom} replicates the cochlea's  travelling wave behaviour. That is, in response to an impulse the position of maximum amplitude moves slowly from left to right in the array, see \cite{davies2019fully} for details. In \cite{perception}, a signal processing architecture is developed, based on the convolutional structure of \eqref{thm:timedom}. This further mimics the action of biological auditory processing by extracting global properties of behaviourally significant sounds, to which human hearing is known to be adapted. 

Finally, it is worth mentioning that biological hearing is an inherently nonlinear process. In \cite{davies2020hopf} nonlinear amplification is introduced to the model in order to replicate the behaviour of the cochlear amplifier. This active structure takes the form of a fluid-coupled array of Hopf resonators. Clarifying the details of the nonlinearities that underpin cochlear function remains the largest open question in understanding biological hearing. One of the motivations for developing devices such as the one analysed here is that it will allow for the investigation of these mechanisms, which is particularly difficult to do on biological cochleas.

\section{Concluding remarks} 

In this review, recent mathematical results on focusing, trapping, and guiding waves at subwavelength scales have been described in the Hermitian case. Systems of subwavelength resonators that exhibit topologically protected edge modes or that can mimic the biomechanical properties of the cochlea have been designed. A variety of mathematical tools for solving wave propagation problems at subwavelength scales have been introduced.  

When sources of energy gain and loss are introduced to a wave-scattering system, the underlying mathematical formulation will be non-Hermitian. This paves the way for
new ways to control waves at subwavelength scales \cite{nh1,nh2,miri2019exceptional}. 
In \cite{high-order,ammari2020exceptional},  the existence of asymptotic exceptional points,  where eigenvalues coincide and eigenmodes are linearly dependent at leading order,  in a parity--time-symmetric pair of subwavelength resonators is proved. Systems exhibiting exceptional points can be used for sensitivity enhancement. Moreover, a structure which exhibits asymptotic unidirectional reflectionless transmission at certain frequencies is designed. In \cite{active},  the phenomenon of topologically protected edge states in  systems of subwavelength resonators with gain and loss is studied. It is demonstrated that localized edge modes appear in a periodic structure of subwavelength resonators with a defect in the gain/loss distribution, and the corresponding frequencies and decay lengths are explicitly computed. Similarly to the Hermitian case, these edge modes can be attributed to the winding of the eigenmodes. In the non-Hermitian case the topological invariants fail to be quantized, but can nevertheless predict the existence of localized edge modes. 

The codes used for the numerical illustrations of the results described in this review can be downloaded at \url{http://www.sam.math.ethz.ch/~hammari/SWR.zip}.

\bibliographystyle{abbrv}
\bibliography{bibreview}{}

\end{document}